\documentclass[aap,preprint]{imsart}
\RequirePackage[OT1]{fontenc}
\RequirePackage{amsthm,amsmath,amssymb,nicefrac,bm,upgreek,verbatim,mathtools}
\RequirePackage[colorlinks,citecolor=blue,urlcolor=blue]{hyperref}
\RequirePackage{dsfont}
\RequirePackage[mathscr]{eucal}
\RequirePackage[numbers,sort&compress]{natbib}

\RequirePackage{enumitem}
\RequirePackage[normalem]{ulem}

\startlocaldefs
\usepackage[margin=1in]{geometry}

\numberwithin{equation}{section}

\theoremstyle{plain}
\newtheorem{thm}{Theorem}[section]
\newtheorem{lem}{Lemma}[section]
\newtheorem{cor}{Corollary}[section]

\theoremstyle{definition}

\newtheorem{notation}{Notation}[section]

\theoremstyle{remark}
\newtheorem{remark}{Remark}[section]

\usepackage[capitalize,nameinlink]{cleveref}

\crefname{section}{Section}{Sections}
\crefname{subsection}{Subsection}{Subsections}
\crefname{notation}{Notation}{Notations}
\crefname{hypothesis}{Hypothesis}{Conditions}
\crefname{assumption}{Assumption}{Assumptions}
\crefname{lemma}{Lemma}{Lemmas}
\crefname{lem}{Lemma}{Lemmas}
\crefname{cor}{Corollary}{Corollaries}
\crefname{thm}{Theorem}{Theorems}
\crefname{claim}{Claim}{Claims}

\Crefname{figure}{Figure}{Figures}

\crefformat{equation}{\textup{#2(#1)#3}}
\crefrangeformat{equation}{\textup{#3(#1)#4--#5(#2)#6}}
\crefmultiformat{equation}{\textup{#2(#1)#3}}{ and \textup{#2(#1)#3}}
{, \textup{#2(#1)#3}}{, and \textup{#2(#1)#3}}
\crefrangemultiformat{equation}{\textup{#3(#1)#4--#5(#2)#6}}%
{ and \textup{#3(#1)#4--#5(#2)#6}}{, \textup{#3(#1)#4--#5(#2)#6}}%
{, and \textup{#3(#1)#4--#5(#2)#6}}

\Crefformat{equation}{#2Equation~\textup{(#1)}#3}
\Crefrangeformat{equation}{Equations~\textup{#3(#1)#4--#5(#2)#6}}
\Crefmultiformat{equation}{Equations~\textup{#2(#1)#3}}{ and \textup{#2(#1)#3}}
{, \textup{#2(#1)#3}}{, and \textup{#2(#1)#3}}
\Crefrangemultiformat{equation}{Equations~\textup{#3(#1)#4--#5(#2)#6}}%
{ and \textup{#3(#1)#4--#5(#2)#6}}{, \textup{#3(#1)#4--#5(#2)#6}}%
{, and \textup{#3(#1)#4--#5(#2)#6}}

\crefdefaultlabelformat{#2\textup{#1}#3}

\newcommand{\Ind}{\mathds{1}} 
\newcommand{\process}[1]{{\{#1(t)\}_{t\ge0}}}
\newcommand{\ttup}[1]{\textup{(}#1\textup{)}}

\newcommand{\Exp}{{\mathbb{E}}} 
\newcommand{\Prob}{{\mathbb{P}}} 
\newcommand{\RR}{{\mathbb R}} 
\newcommand{\Rd}{{{\mathbb R}^d}}
\newcommand{\Rds}{{{\mathbb R}^d_*}}
\newcommand{\NN}{{\mathbb N}} 
\newcommand{\D}{\mathrm{d}} 
\newcommand{\E}{\mathrm{e}} 
\newcommand{\df}{\coloneqq} 
\newcommand{\Id}{\mathbb{I}}
\newcommand{\sB}{{\mathscr{B}}} 
\newcommand{\cB}{{\mathcal{B}}} 
\newcommand{\cM}{{\mathcal{M}}} 
\newcommand{\Usm}{\mathfrak{U}_{\mathrm{SM}}}  
  
\newcommand{\Ag}{{\mathcal{A}}}

\newcommand{\fI}{{\mathfrak{I}}}
\newcommand{\hfI}{{\widehat{\mathfrak{I}}}} 
\newcommand{\fJ}{{\mathfrak{J}}}

\newcommand{\cP}{{\mathcal{P}}} 
\newcommand{\cK}{{\mathcal{K}}} 
\newcommand{\tv}{{\rule[-.45\baselineskip]{0pt}{\baselineskip}\mathsf{TV}}}

\newcommand{\tw}{{\Tilde{w}}}

\newcommand{\uuptau}{\Hat{\uptau}}

\newcommand{\abs}[1]{\lvert#1\rvert}
\newcommand{\norm}[1]{\lVert#1\rVert}
\newcommand{\babs}[1]{\bigl\lvert#1\bigr\rvert}

\newcommand{\babss}[1]{\biggl\lvert#1\biggr\rvert}

\newcommand{\bnorm}[1]{\bigl\lVert#1\bigr\rVert}
\newcommand{\Bnorm}[1]{\Bigl\lVert#1\Bigr\rVert}
\DeclareMathOperator*{\diag}{diag}
\DeclareMathOperator{\trace}{Tr}
\DeclareMathOperator{\support}{supp}

\endlocaldefs
\begin{document}
\begin{frontmatter}

\title{Ergodicity of a L{\'e}vy-driven SDE arising\\
from multiclass many-server queues}
\runtitle{Ergodicity of a L{\'e}vy-driven SDE}

\begin{aug}
\author{\fnms{Ari} \snm{Arapostathis}
\thanksref{m1}\ead[label=e1]{ari@ece.utexas.edu}},
\author{\fnms{Guodong} \snm{Pang}
\thanksref{m2}\ead[label=e2]{gup3@psu.edu}}
\and
\author{\fnms{Nikola} \snm{Sandri{\'c}}\corref{}
\thanksref{m3}\ead[label=e3]{nsandric@math.hr}}

\runauthor{A.~Arapostathis, A.~Pang, and N.~Sandri{\'c}}

\affiliation{University of Texas at Austin\thanksmark{m1},
Pennsylvania State University\thanksmark{m2},\\ and
University of Zagreb\thanksmark{m3}}

\address{Department of Electrical\\ and Computer Engineering\\
University of Texas at Austin\\
2501 Speedway, EER 7.824\\
Austin, TX~~78712\\
\printead{e1}}

\address{The Harold and Inge Marcus Department\\
of Industrial and Manufacturing Engineering\\
College of Engineering\\
Pennsylvania State University\\
University Park, PA~~16802\\
\printead{e2}}

\address{Department of Mathematics\\
 University of Zagreb, Bijeni\v{c}ka cesta 30\\
10000 Zagreb, Croatia\\
\printead{e3}}
\end{aug}

\begin{abstract}
We study the ergodic properties of
a class of multidimensional piecewise Ornstein--Uhlenbeck processes
with jumps, which contains the limit of the queueing processes arising in 
multiclass many-server queues with heavy-tailed arrivals and/or asymptotically
negligible service interruptions in the Halfin-Whitt regime as special cases.
In these queueing models, the It\^o equations have a piecewise linear drift,
and are driven by either (1) a Brownian motion and a pure-jump L{\'e}vy process,
or (2) an anisotropic L{\'e}vy process with independent one-dimensional symmetric
$\alpha$-stable components, or (3) an anisotropic 
L{\'e}vy process as in (2) and a pure-jump L{\'e}vy process. 
We also study the class of models driven by a subordinate Brownian motion, which
contains an isotropic (or rotationally invariant) $\alpha$-stable L{\'e}vy process
as a special case. 
We identify conditions on the parameters in the drift, 
the L{\'e}vy measure and/or covariance function which result in
subexponential and/or exponential ergodicity.
We show that these assumptions are sharp, and we identify some key necessary
conditions for the process to be ergodic.
In addition, we show that for the queueing models described above with no abandonment,
the rate of convergence is polynomial, and we provide
a sharp quantitative characterization of the rate via matching upper
and lower bounds.
\end{abstract}

\begin{keyword}[class=MSC]
\kwd[Primary ]{60J75; 60H10}
\kwd[; secondary ]{60G17; 60J25; 60K25}
\end{keyword}

\begin{keyword}
\kwd{multidimensional piecewise Ornstein--Uhlenbeck processes with jumps}
\kwd{pure-jump L\'evy process}
\kwd{(an)isotropic L\'evy process}
\kwd{(sub)exponential ergodicity}
\kwd{multiclass many-server queues}
\kwd{Halfin--Whitt regime}
\kwd{heavy-tailed arrivals}
\kwd{service interruptions}
\end{keyword}

\end{frontmatter}

\tableofcontents

\section{Introduction}

We consider a $d$-dimensional stochastic differential equation (SDE) of the form
\begin{equation}\label{E-sde}
\D X(t)\,=\,b(X(t))\,\D{t} +\upsigma(X(t))\, \D W(t)+\D L(t),\qquad X(0)=x\in\Rd\,,
\end{equation} 
where 
\begin{enumerate}[wide]
\item [(A1)]
the function $b\colon\Rd\to\Rd$ is given by
\begin{align*}
b(x) &\,=\, \ell-M(x-\langle e,x\rangle^+v)-\langle e,x\rangle^+\varGamma v\\
&\,=\, \begin{cases}
\ell - \bigl(M +(\varGamma-M)v e'\bigr) x\,, & e'x>0\,,\\[2pt]
\ell - Mx\,, &e'x\le0\,,
\end{cases}
\end{align*}
where $\ell \in \Rd$, $v\in\RR^d_+$ satisfies $\langle e,v\rangle=e'v=1$
with $e = (1,\dotsc,1)'\in\Rd$, $M\in\RR^{d\times d}$ is a
nonsingular M-matrix
such that the vector $e'M$ has nonnegative components, 
and $\varGamma=\diag(\gamma_1,\dotsc,\gamma_d)$ with $\gamma_i\in\RR_{+}$,
$i=1,\dotsc,d\,$;

\item [(A2)] 
$\process{W}$ is a standard $n$-dimensional Brownian motion, and
the covariance function $\upsigma\colon \Rd \to \RR^{d\times n}$ is locally
Lipschitz and satisfies, for some constant $\kappa>0$,
\begin{equation*}
\norm{\upsigma(x)}^2\le\kappa(1+\abs{x}^2)\,,\qquad x\in\Rd\,;
\end{equation*}

\item [(A3)]
$\process{L}$ is a $d$-dimensional
pure-jump L\'evy process specified by a drift $\vartheta\in\Rd$ and L\'evy measure
$\nu(\D{y})$. 
\end{enumerate}

In (A1)--(A3), $\norm{M}\df\bigl(\trace\, MM'\bigr)^{\nicefrac{1}{2}}$
denotes the Hilbert-Schmidt norm of a
$d\times n$ matrix $M$, and $\langle\cdot,\cdot\rangle$ stands
for the inner product on $\Rd$.
For a square matrix $M$, $\trace\, M$ stands for the
trace of $M$, and
for a vector $x$ and a matrix $M$, $x'$ and $M'$ stand for their
transposes, respectively.
A $d\times d$ matrix $M$ is called an M-matrix
if it can be expressed as $M=s\mathbb{I}-N$ for some $s>0$ and some
nonnegative $d\times d$ matrix $N$ with the property that $\rho(N)\le s$, where
$\mathbb{I}$ and $\rho(N)$ denote the $d\times d$ identity matrix
and spectral radius of $N$, respectively. 
Clearly, the matrix $M$ is nonsingular if $\rho(N)<s$. 

Such an SDE is often called a
piecewise Ornstein--Uhlenbeck (O--U) process with jumps.
Recall that a L\'evy measure $\nu(\D{y})$ is a $\upsigma$-finite measure on
$\Rds\df\Rd\setminus\{0\}$ satisfying
$\int_{\Rds} (1\wedge\abs{y}^2)\,\nu(\D{y})<\infty$.
It is well-known that the SDE \cref{E-sde} admits a unique nonexplosive strong solution
$\process{X}$ which is a strong Markov process and it satisfies the $C_b$-Feller property
(see \cite[Theorem~3.1, and Propositions 4.2 and 4.3]{Albeverio-Brzezniak-Wu-2010}).
In addition, in the same reference, it is shown that the infinitesimal generator
$(\Ag^X,\mathcal{D}_{\Ag^X})$ of $\process{X}$
(with respect to the Banach space $(\mathcal{B}_b(\RR^{d}),\norm{\,\cdot\,}_\infty)$)
satisfies $C_c^{2}(\Rd)\subseteq\mathcal{D}_{\Ag^X}$ and 
\begin{equation}\label{E1.2}
\Ag^X\bigr|_{C_c^{2}(\Rd)}f(x) \,=\,
\frac{1}{2}\trace\bigl(a(x)\nabla^2f(x)\bigr)
+\bigl\langle b(x)+\vartheta,\nabla f(x)\bigr\rangle
+\int_{\Rds}\mathfrak{d}_1 f(x;y)\nu(\D{y})\,,
\end{equation}
with $\nabla^2f(x)$ denoting the Hessian of $f(x)$.
Here, $\mathcal{D}_{\Ag^X}$, $\mathcal{B}_b(\RR^{d})$ and $C_c^{2}(\Rd)$
denote the domain of $\Ag^X$, the space of bounded Borel measurable functions
and the space of twice continuously differentiable functions with compact support,
respectively.
In \cref{E1.2} we use the notation
$a(x)=\bigl(a^{ij}(x)\bigr)_{1\le i,j\le d}\df \upsigma(x)\upsigma(x)'$, 
and
\begin{equation}\label{frakd1}
\mathfrak{d}_1 f(x;y) \,\df\,
f(x+y)-f(x)-\Ind_{\sB}(y)\langle y,\nabla f(x)\rangle\,,\qquad f\in C^1(\Rd)\,,
\end{equation}
where $\sB$ denotes the unit ball in $\Rd$ centered at $0$,
and $\Ind_{\sB}$ its indicator function.

The goal of this paper is to investigate the ergodic properties of
$\process{X}$. 
This process arises as a limit of the suitably scaled queueing processes of
multiclass many-server queueing networks with heavy-tailed (bursty) arrivals and/or
asymptotically negligible service interruptions. 
In these models, if the scheduling policy is based on a static priority assignment
on the queues, then the vector $v$ in the limiting
diffusion \cref{E-sde} corresponds a constant control, i.e.,
an element of the set
\begin{equation*}
\varDelta \,\df\, \{ v\in\RR^d_+ \,\colon  \langle e,v\rangle = 1\}\,.
\end{equation*}
The process $\process{X}$ also arises in many-server queues with phase-type
service times, where the constant vector $v$ corresponds to the probability
distribution of the phases.

These queueing models are described in detail in \cref{S4}.
It is important to note that for a multiclass queueing network with independent
heavy-tailed arrivals, the process $\process{L}$ in \cref{E-sde}
is an anisotropic L\'evy process consisting of independent 
one-dimensional symmetric $\alpha$-stable components.
For a detailed description see \cref{S4.1}. 
Such processes have a highly singular L\'evy measure and lack the regularity properties
of the standard isotropic (or rotationally invariant) $\alpha$-stable
$d$-dimensional L{\'e}vy processes.
Notably, as shown in \cite{Bass-Chen-10}, the Harnack inequality, an essential
tool in showing regularity of the invariant probability measure for nondegenerate
continuous diffusions, fails for SDEs driven by this anisotropic
L\'evy process.
In \cref{T3.1} we establish the open-set irreducibility of solutions
of \cref{E-sde} driven by an anisotropic $\alpha$ stable process.
This is required in the results which follow. 
Other than the work in \cite{Bass-Chen-06,Bass-Chen-10,Chaker-16}
they have not been studied much.
Under service interruptions, $\process{L}$ is either
a compound Poisson process (under $\sqrt{n}$ scaling),
or an anisotropic L\'evy process described above
together with a compound Poisson component (under $n^{\nicefrac{1}{\alpha}}$
scaling for $\alpha \in (1,2)$). 
In this paper, however, we study the ergodic properties of \cref{E-sde} for
a much broader class of L\'evy processes $\process{L}$.

If the control (scheduling policy) is a function of the state of the system, then
$v(x)$ in the diffusion limit is, in general, a Borel measurable map from
$\Rd$ to $\varDelta$.
We call such a $v(x)$ a stationary Markov control and denote the set of
such controls by $\Usm$.
If \cref{E-sde} is driven by a Wiener process only, 
it follows from the results in \cite{Gyongy-96} that,
under any $v\in\Usm$, the diffusion has a unique strong solution.
On the other hand, as shown in \cite{Skorokhod-89}, 
if the L\'evy measure is finite, the solution of \cref{E-sde} can be constructed in a
piecewise fashion, and thus, in such a case we have a unique
strong solution under any $v\in\Usm$.
There are no such sharp results on existence
of solutions to \cref{E-sde} with a measurable drift,
when this is driven by a general L\'evy process.
However, the well-posedness of the martingale problem for SDEs
with measurable drifts driven by an $\alpha$-stable process
has been studied (see \cite{Zhang-13} and references therein).
We are not concerned with this problem in this paper, especially since the results
involving Markov controls concern only on necessary conditions, and we
clarify that whenever we state a result involving Markov controls it is implied that
the martingale problem is well posed.  Parenthetically we mention here
that for locally Lipschitz Markov controls the problem is always well posed
for the model we consider (see the discussion in the beginning of \cref{S5.4}). 

\subsection{Summary of the results}
Broadly speaking, the results in this paper have two flavors.
On the one hand, we present sufficient conditions under which
$\process{X}$ is ergodic under any constant control $v\in\varDelta$
(Theorems \ref{T3.2}, \ref{T3.4}\,(a), and \ref{T3.5}),
while on the other, we present necessary conditions
for ergodicity under any Markov control (Theorems \ref{T3.3}, \ref{T3.4}\,(b),
\cref{L5.7,C5.1}).

It turns out that these conditions are sharp and they match.
We discuss these results in the context of a many-server queueing network with
heavy-tailed arrivals and/or service interruptions, even though
the results are applicable to a larger class of SDEs.
There are two important parameters involved.
One concerns the heaviness of the tail
of the L\'evy measure, and to describe this we define
\begin{equation}\label{E-thetac}
\Theta_c\,\df\,\Bigl\{\theta>0\, \colon
\int_{\sB^c}\abs{y}^{\theta}\nu(\D{y})\,<\,\infty\Bigl\}\,,\quad\text{and}\quad
\theta_c\,\df\,\sup\,\bigl\{\theta\in\Theta_c\bigl\}\,.
\end{equation}
It follows from its definition that, if bounded, $\Theta_c$ is an open or
left-open interval, i.e., the interval
$(0,\alpha)$ in the case of an $\alpha$-stable process (isotropic or not),
or it could be an interval of the form $(0,\theta_c]$.
When more than one L\'evy components are involved, $\Theta_c$ refers
to the intersection of the individual intervals.
The other parameter is the constant term $\ell$ in the drift which
arises as the limit of the
spare capacity of the network, when driven only by a Wiener process
(see \cref{E4.2}).
It turns out that this constant should be modified to account for the
drift in the L\'evy process.
Recall that $\ell$ is the constant in the drift in (A1) (see \eqref{E-sde})
and that $\vartheta$ and $\nu(\D{y})$ are the drift and the L\'evy measure of the process
$\process{L}$.
We define 
\begin{equation}\label{E-ell}
\Tilde{\ell} \,\df\, \begin{cases}
\ell+\vartheta + \int_{\sB^c}y\nu(\D{y})\,,&\text{if\ \ }
\int_{\sB^c}\abs{y}\nu(\D{y})<\infty\,,\\[5pt]
\ell+\vartheta,& \text{otherwise},
\end{cases}
\end{equation}
and
\begin{equation}\label{E-tvarrho}
\Tilde\varrho\,\df\,  -\bigl\langle e,M^{-1}\Tilde\ell\,\bigr\rangle\,.
\end{equation}
We refer to $\Tilde\varrho$ as the (effective) spare capacity. 

The richest and most interesting set of results concerns networks where
the abandonment rate is $0$, and this corresponds to $\varGamma=0$, or
more generally, when the control gives lowest priority to queues
whose abandonment rate is $0$ (this is equivalent to $\varGamma v=0$).
In this scenario, we establish in \cref{T3.3,L5.7}
that $\Tilde\varrho>0$ and $1\in\Theta_c$ are
both necessary conditions for the state process $\process{X}$ to
have an invariant probability measure under some Markov control $v\in\Usm$
 (see also \cref{C4.1}).
This translates to the requirement that
 $\alpha>1$ if the system has heavy-tailed arrivals, and/or
$\theta_c\ge1$ if there are service interruptions.
If these conditions are met, we show in \cref{T3.2} that the process
is ergodic under any constant control $v\in\varDelta$.
Moreover, we prove in \cref{T3.2,T3.4}  that convergence to the invariant measure
in total variation has a polynomial rate
$r(t)\thickapprox t^{\theta_c-1}$ for any constant control,
and by this we mean that the rate
is $r(t)=t^{\theta_c-\epsilon-1}$ for all $\epsilon\in(0,\theta_c-1)$
for $\theta_c>1$, and $\epsilon=0$ for $\theta_c=1$.
This is accomplished by deriving matching upper and lower bounds
for convergence (see \cref{ET3.4A}).
An interesting related result is that the spare capacity $\Tilde\varrho$
is equal to the average idleness of the system (idle servers) under
any Markov control $v$ satisfying $\varGamma{v}(x)=0$ a.e.,
and such that the state process is ergodic (see \cref{C5.1}). 

In the context of many-server queueing networks,
stability is defined as the finiteness of the average value
of the sum of the queue lengths, and this translates into the
requirement that the map $x\mapsto \langle e,x\rangle^+$ be integrable
under the invariant probability measure of the process. 
In turn, a necessary and sufficient condition for this is that
the invariant probability measure has a finite first absolute moment (see \cref{R5.1}).   
We refer to a control attaining this property as stabilizing.
\cref{L5.7} shows that if there is
no abandonment, then no Markov control is stabilizing unless $2\in\Theta_c$,
while under abandonment it is necessary that $1\in\Theta_c$ (see \cref{C4.1,T3.5}).
This means that for a system with heavy-tailed arrivals
(resulting in an $\alpha$-stable limit with $\alpha \in (1,2)$), 
there are no stabilizing controls, unless some abandonment rate is positive.
On the other hand, for a system under service interruptions and no heavy-tailed
arrivals, a necessary and sufficient condition for the existence of
stabilizing controls, under no abandonment,
is that the L\'evy measure has a finite second moment.
If such is the case, then every constant control is stabilizing
by \cref{T3.2}.
   
Another set of results concern the case $\varGamma v \ne 0$.
Here, we show that $\process{X}$ has an exponential rate of convergence (\cref{T3.5}),
and that every constant control is stabilizing, provided $1\in\Theta_c$.

\subsection{Literature review} 
Our work relates to the active research on L{\'e}vy-driven (generalized)
O--U processes, and the vast literature on SDEs with jumps. 
In \cite{Douc-Fort-Guilin-2009,Fort-Roberts-2005,Kevei16,Masuda-04,
Sato-Yamazato-1984,Wang11,Wang12}, the ergodic properties of a general class
of L{\'e}vy-driven O--U processes are established using Foster-Lyapunov and
coupling methods.
In all these works the process is governed by a linear drift function. 
In \cite{RZ11}, a one-dimensional piecewise O--U process driven by a
spectrally one-sided
L{\'e}vy process is studied. The authors have shown the existence and 
characterization of the invariant distribution, and ergodicity of the process.
In \cite{DG13}, motivated by the many-server 
queuing model with phase-type service times, the authors have established
ergodicity and exponential ergodicity
of a piecewise O--U process driven by Brownian motion only.
See also \cref{R3.5} on the comparison of the models and contributions. 

For general diffusions with jumps, ergodic properties are studied in 
\cite{Kulik-2009, Masuda-2007, Masuda-Erratum-2009,Priola-Shirikyan-Xu-Zabczyk-2012,
Sandric-ESAIM-2016, Wang-2008,Wee-1999}, under suitable
conditions on the drift, covariance function and jump component.
In this paper, we take advantage of the explicit form of the drift and carry out
detailed calculations
which yield important insights on the rates of convergence and ergodic properties. 
Some of the estimates in the proofs may be of independent
interest to future work on the subject. Our results also lay important
foundations for the study of ergodic control problems for L{\'e}vy-driven SDEs
(see a recent development in \cite{AA18-arXiv1}), 
especially those arising from the multiclass many-server queueing systems; 
recent studies on Markovian queueing models are in
\cite{ABP14, Arapostathis-Pang-2016, Arapostathis-Pang-2017, Arapostathis-Pang-2018}. 

A surprising discovery of this study is a class of models in \cref{E-sde}
possessing a ``polynomial'' ergodicity property in the total variation norm. 
Subexponential ergodicity of Markov processes, including diffusions and SDEs with jumps,
has been a very active research area in recent years; see, e.g., \cite{AFV15,
Butkovsky-14, CF09,Douc-Fort-Guilin-2009, DFMS-04,DFM-16,Fort-Roberts-2005,
G-09, Hairer-16, LZZ-10,TT-94}
and references therein. 
Note that in \cite{Douc-Fort-Guilin-2009}, some interesting diffusion models 
and an O--U process (linear drift) driven by a compound Poisson process 
with a heavy-tailed jump is studied as examples for the general theory of
subexponential ergodicity. 
Our work identifies a concrete, yet highly nontrivial, class of SDEs with jumps
that satisfy the conditions for polynomial ergodicity in
\cite{Douc-Fort-Guilin-2009} (see also \cite{Hairer-16}). 
This may be of great interest to a broad audience on the subject of ergodicity of
Markov processes.

The rate of convergence for the limiting queueing process of 
multiclass many-server networks under heavy-tailed arrivals and/or
asymptotically negligible service interruptions has not been studied up to now. 
Finally, it is worth mentioning that there exist very scarce results on
subexponential ergodicity in queueing theory; see, e.g., \cite{HLZ-05}.

\subsection{Organization of the paper}
In the next subsection, we summarize some notation used in this paper.
In \cref{S2}, we review some background material
on the ergodicity of Markov 
processes that is relevant to our study.
The main results are presented in \cref{S3}.
In \cref{S4}, we provide some motivating examples of multiclass
many-server queues which have queueing process limits as in \cref{E-sde},
and state the relevant ergodic properties. 
\cref{S5} is devoted to the proofs of \cref{T3.2,T3.3,T3.4,T3.5}, and
contains some additional results. 
\cref{S-A} contains the proof of \cref{T3.1}.

\subsection{Notation}
We summarize some notation used throughout the paper. 
We use $\Rd$ (and $\mathbb{R}^d_+$), $d\ge 1$, to denote real-valued
$d$-dimensional (nonnegative) vectors, and write $\RR$ ($\RR_+$) for $d=1$.
For $x, y\in \RR$,
$x \vee y = \max\{x,y\}$, $x\wedge y = \min\{x,y\}$, 
$x^+ = \max\{x, 0\}$ and $x^- = \max\{-x,0\}$. 
Let $D^d = D([0,\infty), \Rd)$ denote the $\Rd$-valued function space of all
right-continuous functions on $[0,\infty)$ with left limits everywhere in $(0,\infty)$.
Let $(D^d,M_1)$ denote the space $D^d$ equipped with
the Skorohod $M_1$ topology. Denote $D \equiv D^1$. 
Let $(D_d, M_1) = (D, M_1) \times \dotsb \times (D, M_1)$ be the $d$-fold
product of $(D, M_1)$ with the product topology \cite{WW02}. 
For a set $A\subseteq\Rd$, we use
 $A^{c}$ and $\Ind_{A}$ to denote 
the complement and the indicator function of $A$, respectively.
A ball of radius $r>0$ in $\Rd$ around a point $x$ is denoted by $\sB_{r}(x)$,
or simply as $\sB_{r}$ if $x=0$.
We also let $\sB \equiv \sB_{1}$.
The Euclidean norm on $\Rd$ is denoted by $|\cdot|$.
We let  $\mathfrak{B}(\Rd)$ stand for the Borel $\sigma$-algebra on $\Rd$.
For a Borel probability measure $\uppi(\D{x})$
on $\mathfrak{B}(\Rd)$ and a measurable function $f(x)$,
which is integrable under $\uppi(\D{x})$, we often use the convenient notation
$\uppi(f)=\int_\Rd f(x)\,\uppi(\D{x})$.

\section{Preliminaries} \label{S2}

Let $\bigl(\Omega,\mathcal{F},\mathcal{F}(t),M(t),
\theta(t),\{\Prob^{x}\}_{x\in\Rd} \bigr)$, $t\in[0,\infty)$,
denoted by $\process{M}$ in the sequel,
be a Markov process with c\`adl\`ag sample paths and state space
$(\Rd,\mathfrak{B}(\Rd))$ (see \cite[p.~20]{BG-68}).  
We let $P^M_t(x,\D{y})\df\Prob^x(M(t)\in\D{y})$, $t\ge0$ and $x\in\Rd$,
denote the transition probability of $\process{M}$.
Also, in the sequel we assume that $P^M_t(x,\D{y})$, $t\ge0$ and $x\in\Rd$,
is a probability measure, i.e., $\process{M}$ does not admit a cemetery point
in the sense of \cite{BG-68}. Observe that this is not a restriction since,
as we have already commented, $\process{X}$ is nonexplosive.
The process $\process{M}$ is called
\begin{enumerate}
\item [(i)]
$\varphi$-irreducible if there exists a $\upsigma$-finite measure $\varphi(\D{y})$ on
$\mathfrak{B}(\Rd)$ such that whenever $\varphi(B)>0$ we have
$\int_0^{\infty}P^M_t(x,B)\,\D{t}>0$ for all $x\in\Rd$;

\item [(ii)]
transient if it is $\varphi$-irreducible, and if there exists a countable
covering of $\Rd$ with sets
$\{B_j\}_{j\in\NN}\subseteq\mathfrak{B}(\Rd)$, and for each
$j\in\NN$ there exists a finite constant $c_j\ge0$ such that
$\int_0^{\infty}P^M_t(x,B_j)\,\D{t}\le c_j$ holds for all $x\in\Rd$;

\item [(iii)]
recurrent if it is $\varphi$-irreducible, and $\varphi(B)>0$ implies
$\int_{0}^{\infty}P^M_t(x,B)\,\D{t}=\infty$ for all $x\in\Rd$.
\end{enumerate}

Let us remark that if $\{M(t)\}_{t\ge0}$ is a $\varphi$-irreducible
Markov process, then the irreducibility measure $\varphi(\D{y})$ can be
maximized. This means that there exists a unique ``maximal" irreducibility
measure $\psi$ such that for any measure $\Bar{\varphi}(\D{y})$,
$\{M(t)\}_{t\ge0}$ is $\Bar{\varphi}$-irreducible if, and only if,
$\Bar{\varphi}\ll\psi$ (see \cite[Theorem~2.1]{Tweedie-1994}).
In view to this, when we refer to an irreducibility
measure we actually refer to the maximal irreducibility measure.
It is also well known that every $\psi$-irreducible Markov
process is either transient or recurrent (see \cite[Theorem
2.3]{Tweedie-1994}). 
 
Recall, a Markov process $\process{M}$ is called 
\begin{enumerate}
\item[(1)] open-set irreducible
if its maximal irreducibility measure $\psi(\D{y})$ is fully supported, i.e.,
$\psi(O)>0$ for every open set $O\subseteq\Rd$;
\item [(2)] aperiodic if it admits an irreducible skeleton chain, i.e., 
there exist $t_0>0$ and a $\upsigma$-finite measure $\phi(\D{y})$ on
$\mathfrak{B}(\Rd)$, such that $\phi(B)>0$ implies
$\sum_{n=0}^{\infty} P^M_{nt_0}(x,B) >0$ for all $x\in\Rd$.
\end{enumerate}

Let $\cB(\Rd)$ and $\cP(\Rd)$ denote the classes of Borel
measurable functions
and Borel probability measures on $\Rd$, respectively.
We adopt the usual notation
\begin{equation*}
\uppi\, P^M_t(\D{y})=\int_{\Rd}\uppi(\D{x})P^M_t(x,\D{y})
\end{equation*}
for $\uppi\in\cP(\Rd)$, $t\ge0$, and
$P^M_t f(x) = \int_{\Rd}P^M_t(x,\D{y})f(y)$ for
$t\ge0$, $x\in\Rd$ and $f\in\cB(\Rd)$.
Therefore, with $\delta_x$ denoting the Dirac measure concentrated
at $x\in\Rd$, we have $\delta_x P^M_t(\D{y}) = P^M_t(x,\D{y})$, $t\ge0$.

A
 probability measure $\overline\uppi\in\mathcal{P}(\Rd)$ is called invariant for
$\process{M}$ if
$\int_{\Rd}P^M_t(x,\D{y})\,\overline\uppi(\D{x})
= \overline\uppi(\D{y})$
for all $t>0$. 
It is well known that if $\process{M}$ is
recurrent, then it possesses a unique (up to constant
multiples) invariant measure $\overline\uppi(\D{y})$
(see \cite[Theorem~2.6]{Tweedie-1994}).
If the 
invariant measure is
finite, then it may be normalized to a probability measure. If
$\process{M}$ is recurrent with finite invariant measure, then $\process{M}$ is called 
positive recurrent; otherwise it is called null recurrent. Note that a transient 
Markov process cannot have a finite invariant measure. 
Indeed, assume that
$\process{M}$ is transient and that it admits a
finite invariant measure $\overline\uppi(\D{y})$, and fix some $t>0$.
Then, for each $j\in\NN$, with $c_j$ and $B_j$ as in (ii) above, we have
\begin{equation*}
t\,\overline\uppi(B_j) \,=\,
\int_0^{t}\int_{\Rd}P^M_t(x,B_j)\,\overline\uppi(\D{x})\,\D{s}
\,\le\, c_j\overline\uppi(\Rd)\,.
\end{equation*}
 Now, by
letting $t\rightarrow\infty$ we obtain $\overline\uppi(B_j)=0$ for all
$j\in\NN$, which is impossible.

A Markov process $\process{M}$ is called ergodic
if it possesses an invariant probability 
measure $\overline\uppi(\D{y})$ and there exists a nondecreasing function
$r\colon\RR_+\to[1,\infty)$ such that 
\begin{equation*}
\lim_{t\to\infty}\;r(t)\,\bnorm{P^M_t(x,\D{y})
-\overline\uppi(\D{y})}_\tv \,=\,0\,,\qquad x \in \Rd\,.
\end{equation*}
Here, $\norm{\,\cdot\,}_\tv$ denotes the total variation norm on the space
of signed measures on
$\mathfrak{B}(\Rd)$. 
For a function $f\colon\Rd\to[1,\infty)$ we define the $f$-norm of a signed
measure $\mu$ as
\begin{equation*}
\norm{\,\mu\,}_f\,\df\,\sup_{\substack{g\in\cB(\Rd), \  \abs{g}\le f}}\;
\babss{\int_{\Rd}g(y)\mu(\D{y})}\,.
\end{equation*}
Observe that $\norm{\,\cdot\,}_1=\norm{\,\cdot\,}_\tv$. 
We say that $\process{M}$ is subexponentially ergodic if it is ergodic and 
$\lim_{t\to\infty}\nicefrac{\ln r(t)}{t}=0$, and 
 that it is exponentially ergodic if it is ergodic and
$r(t)=\E^{\kappa t}$ for some $\kappa>0$. 
Let us remark that (under the assumptions of open-set irreducibility
and aperiodicity) ergodicity is equivalent 
to positive recurrence (see \cite[Theorem 13.0.1]{Meyn-Tweedie-Book-2009},
\cite[Theorem 6.1]{Meyn-Tweedie-AdvAP-II-1993},
and \cite[Theorems 4.1, 4.2 and 7.1]{Tweedie-1994}).

Since $\process{M}$ is a Markov process,
$P^M_tf(x)=\int_{\Rd}f(y)P_t^M(x,\D{y})$,  $x\in\Rd$,
defines a semigroup of linear operators $\{P^M_t\}_{t\ge0}$  on the
Banach space $(\mathcal{B}_b(\Rd),\norm{\,\cdot\,}_\infty)$, i.e.,
$P^M_s\circ P^M_t=P^M_{s+t}$ for all $s,t\ge0$, and $P^M_0f=f$.
Here, $\norm{\,\cdot\,}_\infty$ denotes the supremum norm on the space
$\mathcal{B}_b(\Rd)$.
The infinitesimal generator
$(\mathcal{A}^M,\mathcal{D}_{\mathcal{A}^M})$ of the semigroup
$\{P^M_t\}_{t\ge0}$ of a Markov process $\process{M}$ is a linear operator
$\mathcal{A}^M\colon\mathcal{D}_{\mathcal{A}^M}\longrightarrow \mathcal{B}_b(\Rd)$
defined by
\begin{align*}
\mathcal{A}^Mf &\,\df\,
\lim_{t\to 0}\,\frac{P^M_tf-f}{t}\,, \\
 f\in\mathcal{D}_{\mathcal{A}^M} &\,\df\,
\biggl\{f\in \mathcal{B}_b(\Rd)\,\colon
\lim_{t\to0}\frac{P^M_t f-f}{t} \ \text{exists in\ }
\norm{\,\cdot\,}_\infty\biggr\}\,.
\end{align*}
Let $C_b(\Rd)$ denote the space of continuous bounded functions.
A Markov process $\process{M}$ is 
called $C_b$-Feller process if its corresponding semigroup satisfies 
$P^M_t(C_b(\Rd))\subseteq
C_b(\Rd)$ for all $t\ge0$, and it is called a strong Feller
 process if $P^M_t(\mathcal{B}_b(\Rd))\subseteq C_b(\Rd)$ for all
$t>0$.

Recall that the extended domain of $\process{M}$, denoted by
$\mathcal{D}_{\Bar{\mathcal{A}}^M}$,  is defined as the set of
all $f\in\mathcal{B}(\Rd)$ such that
$f(M(t))-f(M(0))-\int_{0}^{t}g(M(s))\D{s}$ is
a local $\{\Prob^{x}\}_{x\in\Rd}$-martingale for some
$g\in\mathcal{B}(\Rd)$.
Let us remark that in general the function $g$ does not have to be
unique (see \cite[Page 24]{Ethier-Kurtz-Book-1986}). For
$f\in\mathcal{D}_{\Bar{\mathcal{A}}^M}$ we define
\begin{align*}
\Bar{\mathcal{A}}^Mf &\,\df\,\biggl\{g\in\mathcal{B}(\Rd)\,\colon
f(M(t))-f(M(0))-\int_{0}^{t}g(M(s))\,\D{s}
\text{\ \ is a local $\{\Prob^{x}\}_{x\in\Rd}$-martingale}\biggr\}.
 \end{align*}
We call $\Bar{\mathcal{A}}^M$ the extended generator of
$\process{M}$. A function $g\in\Bar{\mathcal{A}}^Mf$ is usually
abbreviated by $\Bar{\mathcal{A}}^Mf(x)\df g(x)$.
A well-known fact is that
if $(\mathcal{A}^M,\mathcal{D}_{\mathcal{A}^M})$ is the infinitesimal generator of 
$\process{M}$, then
$\mathcal{D}_{\mathcal{A}^M}\subseteq\mathcal{D}_{\Bar{\mathcal{A}}^M}$
and for $f\in\mathcal{D}_{\mathcal{A}^M}$ the function
$\mathcal{A}^Mf$ is contained in $\Bar{\mathcal{A}}^Mf$ (see
\cite[Proposition IV.1.7]{Ethier-Kurtz-Book-1986}). 
In the case of the process $\process{X}$,
it has been shown in \cite[Lemma~3.7]{Masuda-2007, Masuda-Erratum-2009} that
\begin{equation}\label{ext}
\mathcal{D} \,\df\, \biggl\{f\in C^{2}(\Rd)\,\colon
x\longmapsto \babss{\int_{\sB^c}f(x+y)\,\nu(\D{y})}
\text{\ is locally bounded}\biggr\}
\end{equation}
is a subset of $\mathcal{D}_{\Bar{\mathcal{A}}^X}$,
and on this set, for the function $\Bar{\mathcal{A}}^Xf(x)$ we can take exactly
$\Ag^Xf(x)$, where $\Ag^X$ is given by \cref{E1.2}.

\section{Ergodic Properties}\label{S3}
We start by examining the irreducibility and aperiodicity of the process $\process{X}$
in \cref{E-sde}.
This is the topic of the following theorem whose proof can be found
in \cref{S-A}.

\begin{thm}\label{T3.1}
Suppose that one of the following four conditions holds.
\begin{enumerate}[wide]
\item[\ttup i]
 $\nu(\Rd)<\infty$, 
and for every $R>0$ there exists $c_R>0$ such that
\begin{equation*}
\langle y,a(x)y\rangle \,\ge\, c_R \abs{y}^2,
\qquad x,y\in\Rd,\ \abs{x},\abs{y}\le R\,.
\end{equation*}

\item[\ttup{ii}]
$\nu(O)>0$ for any non-empty open set $O\subseteq\sB$, and
$\upsigma\colon\Rd\to\RR^{d\times d}$ is Lipschitz continuous and
invertible for any $x\in\Rd$, satisfying
\begin{equation*}
\delta\df\sup_{x\in\Rd}\,\bnorm{\upsigma^{-1}(x)}\,>\,0\,.
\end{equation*}

\item[\ttup{iii}] $\upsigma(x)\equiv \upsigma$ and $\process{L}$
is of the form $L(t)=L_1(t)+L_2(t)$, $t\ge0$, where $\process{L_1}$ and $\process{L_2}$
are independent $d$-dimensional pure-jump L\'evy processes,
such that $\process{L_1}$ is 
a subordinate Brownian motion. 

\item[\ttup{iv}] $\upsigma(x)\equiv 0$ and $\process{L}$
is of the form $L(t)=L_1(t)+L_2(t)$, $t\ge0$, where $\process{L_1}$ and $\process{L_2}$
are independent $d$-dimensional pure-jump L\'evy processes, such that
$\process{L_1}$ is an anisotropic L\'evy process
with independent symmetric one-dimensional $\alpha$-stable components
for $\alpha \in (0,2)$, and $\process{L_2}$ is a compound Poisson process. 
\end{enumerate}
Then the process $\process{X}$ is open-set irreducible and aperiodic.
\end{thm}

Recall that a L\'evy process $\process{L}$ is a $d$-dimensional subordinate
Brownian motion if it is of the form $L(t)=W(S(t))$, $t\ge0$, where $\process{W}$
is a $d$-dimensional Brownian motion and $\process{S}$ is a subordinator
(a one-dimensional non-negative increasing L\'evy process with $S(0)=0$)
independent of $\process{W}$.
Moreover, any isotropic
$\alpha$-stable L\'evy process can be obtained as a subordinate Brownian motion with
$\nicefrac{\alpha}{2}$-stable subordinator, hence part (iii) of \cref{T3.1} includes 
a $d$-dimensional
isotropic stable L\'evy process as a special case.
We also note that in \cref{T3.1}\,(iii), the component $\process{L_2}$ can be
any pure-jump L\'evy process or vanish, and 
in addition, we require that $\upsigma(x)$ is constant, but it can either be a
$d\times n$ or $d\times d$ singular or non-singular matrix, and it can vanish. 
In the interest of brevity, we often refer to the process $\process{L_1}$ in
\cref{T3.1}~(iv) as the anisotropic $\alpha$-stable process.
Unless otherwise specified, by an $\alpha$-stable process we refer to both the
isotropic and anisotropic models.

We remark that the hypotheses in \cref{T3.1} include a broader class of
processes $\process{X}$ than those encountered in multiclass many-server queues
described in \cref{S4}.

We continue with the main results of the paper concerning the
ergodicity of the process $\process{X}$ in \eqref{E-sde}.
We present four theorems whose proofs can be found in \cref{S5}.
In all these theorems, the hypotheses of \cref{T3.1}
are granted in order to
guarantee that $\process{X}$ is open-set irreducible and aperiodic.
This is important when applying the Foster--Lyapunov drift condition in
\cref{ET3.2C}, \cref{ET3.2F} and \cref{ET3.5B} in order to conclude
\cref{ET3.2D}, \cref{ET3.2G} and \cref{ET3.5C}, respectively 
(see \cite[Theorem~3.2]{Douc-Fort-Guilin-2009}
and \cite[Theorem 5.2]{Down-Meyn-Tweedie-1995}).

We start by introducing the following notation.

\begin{notation}\label{N3.1}
For a vector $z\in\Rd$, we write $z\ge0$ ($z>0$) to indicate that
all the components of $z$ are nonnegative (positive), and analogously
for a matrix in $\RR^{d\times d}$.
The notation $z\lneqq 0$ stands for $-z\ge0$ and $z\ne0$.
For a symmetric matrix $S\in\RR^{d\times d}$, we write $S\succeq 0$ ($S\succ 0$)
to indicate that it is positive semidefinite (positive definite),
and we let $\cM_+$ denote the class of positive definite
symmetric matrices in $\RR^{d\times d}$. 
For $Q\in\cM_+$, we let
$\norm{x}_Q \df \langle x,Qx\rangle^{\nicefrac{1}{2}}$
for $x\in\Rd$.
Let $\Hat\phi(x)$ be some fixed positive, convex smooth function which agrees
with $\norm{x}_Q$ on the complement of the unit ball
centered at $0$ in $\Rd$.
For $\delta>0$, we define $V_{Q,\delta}(x)\df\bigl(\Hat\phi(x)\bigr)^\delta$,
and $\widetilde{V}_{Q,\delta}(x)\df\E^{\delta\Hat\phi(x)}$.
For $r>0$, we let $\uptau_r$ denote the first hitting time
of $\sB_r$,
and $\uuptau_r$ the first hitting time of $\sB_r^c$.
Recall also that a continuous function $V\colon\Rd\to\RR$
is called inf-compact if the set
$\bigl\{x\,\colon V(x)\le r\bigr\}$
is compact (or empty) for all $r\in\RR$.
By $\cP_p(\Rd)$, $p>0$, we denote the subset of $\cP(\Rd)$
containing all probability measures $\mu(\D{x})$ with the property that
$\int_{\Rd}\abs{x}^p\mu(\D{x})<\infty$. 
We let $\cK_\delta\subset\Rd$, $\delta>0$, stand for the cone
\begin{equation}\label{E-cone}
\cK_{\delta} \,\df\, \bigl\{ x\in\Rd\,\colon \langle e,x\rangle
> \delta \abs{x}\bigr\}\,. 
\end{equation}
\end{notation}
 
In the multiclass queueing context, the theorem that follows concerns the
case where the jobs do not abandon the queue, or more generally
when those jobs that abandon the queue are
given higher priority in service than those that do not
(i.e, when not all $\gamma_i$'s are positive, then $v_i$ must be equal to $0$
if $\gamma_i>0$).
The L\'evy process here refers to any, or a combination, of processes
in \cref{T3.1}.
Recall the definitions in \cref{E-thetac,E-ell,E-tvarrho}.

\begin{thm}\label{T3.2}
Assume the hypotheses of \cref{T3.1}, $1\in\Theta_c$, and suppose $\varGamma v=0$,
with $v\in\varDelta$.
Then, provided that $\Tilde\varrho>0$, the following hold.
\begin{enumerate}[wide]
\item[\ttup{i}]
Suppose
\begin{equation}\label{ET3.2A}
\limsup_{\abs{x}\to\infty}\;\frac{\norm{a(x)}}{\abs{x}}\,<\,\infty\,.
\end{equation}
Then, there exist $Q\in\cM_+$, depending on $v$, and positive
constants $c_0=c_0(\theta)$, $c_1$, and $\delta$, such that
for any $\theta\in\Theta_c$, $\theta\ge1$, we have
\begin{equation}\label{ET3.2C}
\Ag^X V_{Q,\theta} (x) \,\le\, c_0(\theta)
- c_1 V_{Q,\theta}(x)\Ind_{\cK_\delta^c}(x)
- c_1 V_{Q,(\theta-1)}(x)\Ind_{\cK_\delta}(x)
\end{equation}
for all $x\in\Rd$.
The process $\process{X}$ admits a unique invariant probability measure
$\overline\uppi\in\cP(\Rd)$, and satisfies 
\begin{equation}\label{ET3.2D}
\lim_{t\to\infty}\;t^{\theta-1}\,\bnorm{\uppi P^X_t(\D{y})
-\overline\uppi(\D{y})}_\tv\,=\,0\,,\qquad \uppi\in\cP_\theta(\Rd)\,.
\end{equation}
In addition, when $\theta=1$, then \cref{ET3.2D} holds for
any $\uppi\in\mathcal{P}(\Rd)$.

\item[\ttup{ii}] 
If $\upsigma(x)$ is bounded and
\begin{equation}\label{ET3.2E}
\int_{\sB^c}\E^{\theta\abs{y}}\nu(\D{y})\,<\,\infty\,
\end{equation}
for some $\theta>0$, then there exist $Q\in\cM_+$ and positive constants
$\Tilde{c}_0$, $\Tilde{c}_1$ such that
\begin{equation}\label{ET3.2F}
\Ag^X \widetilde{V}_{Q,p} (x) \,\le\, \Tilde{c}_0
- \Tilde{c}_1 \widetilde{V}_{Q,p} (x)\,,\qquad x\in\Rd\,,
\end{equation} where $0<p<\theta\norm{Q}^{-\nicefrac{1}{2}}$.
The process $\process{X}$
admits a unique invariant probability measure $\overline\uppi\in\cP(\Rd)$,
and for any $\gamma\in(0, c_1)$ there exists
a positive constant $C_\gamma$ such that
\begin{equation}\label{ET3.2G}
\bnorm{\delta_x P^X_t(\D{y})-\overline\uppi(\D{y})}_{\widetilde{V}_{Q,p}}\,\le\,
C_\gamma \widetilde{V}_{Q,p}(x)\, \E^{-\gamma t}\,,\quad x\in\Rd\,,\ 
t\ge0\,.
\end{equation}
\end{enumerate} 
\end{thm}

\begin{remark}
Note that $\Tilde\varrho>0$ is always
satisfied if $\Tilde{\ell}\lneqq 0$.
This is because $M^{-1}$ is a positive matrix (see p.~1307 in \cite{DG13}).
The same is true if $M$ is diagonal matrix with positive
diagonal elements.
\end{remark}

The assumption $\Tilde\varrho>0$ in \cref{T3.2} is rather sharp
as the following theorem shows.

\begin{thm}\label{T3.3}
Suppose that \eqref{E-sde} is driven
by any or a combination of \ttup{a}--\ttup{c} below, while conforming to
\ttup{i}--\ttup{iv} of \cref{T3.1}.
\begin{enumerate}[wide]
\item[\ttup a]
A Brownian motion with $\upsigma(x)$ bounded.

\item[\ttup b]
A L\'evy process $L(t)$ which is either an anisotropic process
with independent symmetric one-dimensional $\alpha$-stable components,
or an $\alpha$-stable process, with $\alpha\in(1,2)$.

\item[\ttup c]
A L\'evy process with a finite L\'evy measure $\nu(\D{y})$, supported on a half-line
in $\Rd$ of the form $\{tw\,\colon t\in[0,\infty)\}$, and with $1\in\Theta_c$.
\end{enumerate}
Under these hypotheses, if $\Tilde\varrho<0$ 
\textup{(}$\Tilde\varrho=0$\textup{)}, then
the process $\process{X}$ is transient 
\textup{(}cannot be positive recurrent\textup{)}
under any Markov control $v(x)$ satisfying $\varGamma v(x)=0$ a.e. 
\end{thm}

\Cref{T3.3} should be compared to \cref{L5.7}
which does not assume that $\upsigma(x)$ is bounded.
However, \cref{T3.3} establishes a stronger result when $\Tilde\varrho<0$.

In general, if $\theta_c<\infty$, then under the assumptions of \cref{T3.2}~(i)
we cannot have exponential ergodicity, as the next theorem shows.
This is always true for models where the L\'evy process is a subordinate Brownian motion,
or an anisotropic process
with independent symmetric one-dimensional $\alpha$-stable components.
However, in the case of a compound Poisson process, the asymmetry of the L\'evy
measure may be  beneficial to ergodicity (see \cref{R5.3}).
For the limiting SDEs that arise from stochastic networks under
service interruptions, the L\'evy measure is supported on a half-line in the direction
of some  $w\in\RR_+^d$.
In the theorem that follows we enforce this as a hypothesis in part (ii).

\begin{thm}\label{T3.4}
Grant the structural hypotheses of \cref{T3.1}, assume the growth
condition \cref{ET3.2A}, and either, or both of the following:
\begin{enumerate}[wide]
\item[\ttup i]
The L\'evy process is a subordinate Brownian motion
such that  $1\in\Theta_c$ and $\theta_c<\infty$, or an anisotropic process
with independent symmetric one-dimensional $\alpha$-stable components,
with $\alpha\in(1,2)$.

\item[\ttup{ii}]
The L\'evy process has a finite L\'evy measure such that
$1\in\Theta_c$ and $\theta_c<\infty$, and
$\nu(\D{y})$ is supported on a half-line
of the form $\{tw\,\colon t\in[0,\infty)\}$,
with $\langle e,M^{-1} w\rangle>0$.
\end{enumerate}

Then the following hold.
\begin{enumerate}[wide]
\item[\ttup a]
Suppose $v\in\varDelta$, $\varGamma v=0$, and $\Tilde\varrho>0$.
Then
$\process{X}$ is polynomially ergodic, and its rate
of convergence is $r(t)\thickapprox t^{\theta_c-1}$.
In particular, in the case of an $\alpha$-stable process (isotropic or not), we
obtain the following quantitative bounds.
There exist positive constants $\Tilde{C}_1$, and $\Tilde{C}_2(\epsilon)$
 such that for all $\epsilon\in(0,\alpha-1)$, we have
\begin{align}\label{ET3.4A}
\Tilde{C}_1 \Bigl(\frac{t\vee 1}{\epsilon}
+ \abs{x}^{\alpha-\epsilon}\Bigr)^{\frac{1-\alpha}{1-\epsilon}}
&\,\le\, \bnorm{\delta_{x} P^X_t(\D{y})-\overline\uppi(\D{y})}_\tv \\
&\,\le\,
\Tilde{C}_2(\epsilon) (t\vee1)^{1+\epsilon-\alpha}\abs{x}^{\alpha-\epsilon}
\nonumber
\end{align}
for all $t>0$, and all $x\in\Rd$.

On the other hand, in the case of a L\'evy process in \ttup{ii} we obtain the
following lower bound.
There exists a positive constant $\Tilde{C}_3(\epsilon)$ such that for all
$\epsilon\in(0,\nicefrac{1}{3})$,
and all $x\in\Rd$, we have
\begin{equation}\label{ET3.4B}
\bnorm{\delta_{x} P^X_{t_n}(\D{x}) -\overline\uppi(\D{x})}_\tv
\,\ge\,\Tilde{C}_3(\epsilon)\,
\bigl (t_n+\abs{x}^{\theta_c-\epsilon}\bigr)^{-\frac{\theta_c-1+2\epsilon}{1-3\epsilon}}
\end{equation}
for some sequence $\{t_n\}_{n\in\NN}\subset[0,\infty)$, $t_n\to\infty$, depending on $x$.
The upper bound has the same form as the one in \cref{ET3.4A},
but with $\alpha$ replaced by $\theta_c$.

\item[\ttup b]
If under some control $v\in\Usm$,
such that $\varGamma v(x)=0$ a.e.  the process $\process{X}$ has an invariant
probability measure $\overline\uppi\in\cP_p(\Rd)$, $p\ge0$,
then $p+1\in\Theta_c$, and $\Tilde\varrho>0$.
In addition, $\Tilde\varrho\,=\,\int_\Rd \langle e,x\rangle^-\,\overline\uppi(\D{x})$.
Conversely,  if $v$ is a constant control such that $\varGamma v=0$,
$\Tilde\varrho>0$, and $p\ge0$ is such that $p+1\in\Theta_c$, then
$\process{X}$ admits a unique invariant probability measure
$\overline{\uppi}\in\cP_p(\Rd)$. 
\end{enumerate}
\end{thm}

\begin{remark}
Roughly speaking, the mechanism that results in polynomial ergodicity can be
described as follows.
In rough terms, exponential ergodicity is related to the existence
of a supersolution $V(x)\ge1$ of
$\Ag^XV(x) \le c_0\Ind_\sB(x) - c_1 V(x)$,
for some positive constants $c_0$ and $c_1$ (see \cite{Down-Meyn-Tweedie-1995}). %
Consider the model where $\process{L}$ is an isotropic $\alpha$-stable process
with $\alpha>1$.
For $V(x)$ to be integrable under the L\'evy measure
inherited from the $\alpha$-stable kernel $\nicefrac{1}{\abs{y}^{d+\alpha}}$,
it cannot grow faster than $\abs{x}^\alpha$.
On the other hand, the nonlocal part of the infinitesimal generator acts like
a derivative of order $\alpha$ (see \cref{L5.3}).
Thus, since $\alpha>1$, such a supersolution must satisfy
$\langle b(x),\nabla V(x)\rangle \le - \epsilon V(x)$ for some
$\epsilon>0$, and all $\abs{x}$ large enough.
Since $V(x)$ has polynomial growth,
this requires the drift $b(x)$ to have at least linear growth in $x$
(see also \cref{C5.3}).
\end{remark}

\begin{remark}
Suppose that the L\'evy process is 
an anisotropic process
with independent symmetric one-dimensional $\alpha$-stable components.
Then, by \cref{T3.4}, the invariant probability measure $\Bar\uppi(\D{x})$ of
$\process{X}$ cannot have a finite first absolute moment.
In the context of queueing networks this means that, under
heavy-tailed arrivals with an $\alpha$-stable limit, $\alpha \in (1,2)$, a constant control $v$
cannot stabilize the network unless $\varGamma v\ne0$.
\end{remark}

The next theorem asserts exponential ergodicity
for models corresponding to queueing problems with reneging (abandonment).
\begin{thm}\label{T3.5}
Grant the hypotheses of \cref{T3.1}.
Suppose that $\theta\in\Theta_c$,
\begin{equation}\label{ET3.5A}
\limsup_{\abs{x}\to\infty}\;\frac{\norm{a(x)}}{\abs{x}^2}\,=\,0\,,
\end{equation}
and that one of the following holds:
\begin{enumerate}[wide]
\item[\ttup i]
$Mv\ge\varGamma v\gneqq0$; 
\item[\ttup{ii}]
$M=\diag(m_1,\dotsc,m_d)$ with $m_i>0$,
$i=1,\dotsc,d$, and $\varGamma v\ne 0$.
\end{enumerate}
Then there exists $Q\in \mathcal{M}_+$ such that
\begin{equation*}
MQ+QM\succ0\,,
\quad\text{and}\quad (M- ev'(M-\varGamma))Q + Q(M- (M-\varGamma)v e')\,\succ\,0\,,
\end{equation*}
and positive constants
$\Bar{c}_0$, $\Bar{c}_1$ satisfying
\begin{equation}\label{ET3.5B}
\Ag^X V_{Q,\theta} (x) \,\le\, \Bar{c}_0 - \Bar{c}_1 V_{Q,\theta} (x)\,,
\qquad x\in\Rd\,.
\end{equation}
The process $\process{X}$ admits a unique invariant probability measure
$\overline\uppi\in\cP(\Rd)$, and for any $\gamma\in(0, \Bar{c}_1)$ there exists
a positive constant $C_\gamma$ such that for $p \in (0, \theta]$,
\begin{equation}\label{ET3.5C}
\bnorm{\delta_x P_t(\D{y})-\overline\uppi(\D{y})}_{V_{Q,p}}\,\le\,
C_\gamma V_{Q,p}(x)\, \E^{-\gamma t}\,,\qquad x\in\Rd\,,\ t\ge0\,.
\end{equation}
In addition, 
$\overline\uppi\in\cP_q(\Rd)$ if, and only
if, $q\in\Theta_c$.
\end{thm}
\begin{remark}\label{R3.5}
\cref{T3.2,T3.5} generalize \cite[Theorems~2 and~3]{DG13}
for the corresponding diffusion models. 
In \cite[Theorem~2]{DG13}, the model in \cref{E-sde}
is driven by a Brownian motion
$\process{W}$, $\ell=-l v$ for some $l>0$ and $\varGamma=0$, and it is shown
that $\process{X}$ admits a unique invariant probability
measure $\overline\uppi(\D{x})$ and is ergodic.
For the same model, but with
$\ell=-l v$ and $\varGamma=c \mathbb{I}$ for some
$l\in\RR$ and $c>0$, \cite[Theorem~3]{DG13} establishes exponential ergodicity for
$\process{X}$.
\cref{T3.2} improves \cite[Theorem~2]{DG13} to exponential ergodicity
of the process $\process{X}$, under a weaker hypothesis on $\ell$, which
is shown to be also necessary for positive recurrence.
Moreover, in the proof of \cite[Theorem~3]{DG13}, a sophisticated
non-quadratic Lyapunov function is constructed, whereas
we employ a quadratic type Lyapunov function
(e.g., $V_{Q, \theta}(x)$ in \cref{ET3.5B})
in the proof of \cref{T3.5}.

We note that
the hypothesis that $M$ is diagonal in (ii) of \cref{T3.5}
can be waived if we assume that
$\varGamma v=\gamma v$ for some $\gamma>0$ (which is a rather restrictive assumption).
In such a case
a slight modification of the arguments in \cite[Theorem~3]{DG13} and
\cref{T3.2}, shows that the process $\process{X}$ is exponentially ergodic.
\end{remark}

\section{Multiclass Many-Server Queueing Models}\label{S4}

In this section, we present some  examples of many-server queueing systems
for which the
class of piecewise O-U processes with jumps in \cref{E-sde} arises as a limit in the
so-called (modified) Halfin--Whitt (H--W) heavy-traffic regime \cite{HW81}. 

In the queueing context, we identify three classes of processes $\process{X}$: 
\begin{enumerate}
\item[(C1)]
$\upsigma(x) \equiv \upsigma$ is a $d\times d$ nonsingular matrix
and the process $\process{L}$ is a $d$-dimensional pure-jump L\'evy process,
with $\nu(\Rd)<\infty$;

\item[(C2)]
$\upsigma(x)\equiv 0$, and $L(t)=L_1(t)$, with
$\process{L_1}$ the anisotropic L\'evy
process from \cref{T3.1}\,(iv) with $\alpha\in(1,2)$; 

\item[(C3)]
$\upsigma(x) \equiv 0$ and the process $\process{L}$ takes the form in
\cref{T3.1}\,(iv) with $\alpha\in(1,2)$.
\end{enumerate}

Case (C1) corresponds to a multiclass many-server queueing network having
service interruptions
(with the $\sqrt{n}$ scaling), (C2) to heavy-tailed arrivals, and (C3) to a combination
of both (with the $n^{\nicefrac{1}{\alpha}}$ scaling for $\alpha \in (1,2)$). 
Case (C1) is covered by (i) in \cref{T3.1},
and cases (C2) and (C3) are covered by (iv).
We describe how these arise,
and summarize the ergodic properties of the limiting processes
for these queueing models in \cref{S4.1,S4.2}.

\subsection{Multiclass \texorpdfstring{$G/M/n+M$}{G/M/n+M} queues with heavy-tailed arrivals}
\label{S4.1}

In \cite{PW10}, a functional central limit theorem (FCLT) is proved for the
queueing process in the $G/M/n+M$ model with first-come-first-served (FCFS)
service discipline in a modified H-W regime. Customers waiting in queue can
abandon before receiving service (the $+M$ in the notation).
The limit process is a one-dimensional SDE with a piecewise-linear drift,
driven by a symmetric $\alpha$-stable L{\'e}vy process (a special case of the process
$\process{X}$ in \cref{E-sde}). This analysis can be easily extended to multiclass 
$G/M/n+M$ queues under a constant Markov control.

Consider a sequence of $G/M/n+M$ queues with $d$ classes of customers,
indexed by $n$ and let $n\to\infty$.
Customers of each class form their own queue and are served in the order
of their arrival. 
Let $A^n_i$, $i=1,\dotsc,d$, be the arrival process of class-$i$ customers
with arrival rate $\lambda^n_i$. Assume that $A^n_i$'s are mutually independent. 
The service and patience times are exponentially distributed, with
class-dependent rates, $\mu_i$ and $\gamma_i$, respectively, for class-$i$
customers.
The arrival, service and abandonment processes of each class are mutually independent. 
Define the FCLT-scaled arrival processes
$\Hat{A}^n= (\Hat{A}^n_1,\dotsc,\Hat{A}^n_d)'$ by
$\Hat{A}^n_i\df n^{-\nicefrac{1}{\alpha}} (A^n_i- \lambda^n_i \varpi)$,
$i=1,\dotsc,d$,
where $\varpi(t)\equiv t$ for each $t\ge 0$, and $\alpha \in (1, 2]$. 
We assume that 
\begin{equation} \label{E4.1}
\nicefrac{\lambda^n_i}{n} \,\to\, \lambda_i\,>\,0, \quad \text{and}\quad
\Hat{\ell}^n_i\,\df\, n^{-\nicefrac{1}{\alpha}}(\lambda^n_i - n \lambda_i)
\,\to\, \Hat{\ell}_i\in \RR\,,
\end{equation}
for each $i=1,\dotsc,d$, as $n\to\infty$.
It follows from \cref{E4.1} that
\begin{equation} \label{E4.2}
n^{1-\frac{1}{\alpha}} (1-\rho^n)\,\xrightarrow[n\to\infty]{}\, \Hat\rho
\,=\, - \sum_{i=1}^d\frac{\Hat\ell_i}{\mu_i}\,,
\end{equation}
where $\rho^n \df \sum_{i=1}^d \nicefrac{\lambda^n_i}{n \mu_i}$
is the aggregate traffic intensity.
Under \cref{E4.1,E4.2},
the system is critically loaded, i.e., it satisfies
$\sum_{i=1}^d \nicefrac{\lambda_i}{\mu_i} =1$.
Assume that the arrival processes satisfy an FCLT
\begin{equation}\label{E-FCLT}
\Hat{A}^n \;\Rightarrow\; \Hat{A} = (\Hat{A}_1,\dotsc, \Hat{A}_d)'
\qquad\text{in\ } (D_d, M_1), \ \text{as\ } n\to\infty,
\end{equation}
where the limit processes $\Hat{A}_i$, $i=1,\dotsc,d$, are mutually
independent symmetric $\alpha$-stable processes with $\Hat{A}_i(0)\equiv 0$,
and $\Rightarrow$ denotes
weak convergence. 
The processes $\Hat{A}_i$ have the same stability parameter $\alpha$, with possibly
different ``scale'' parameters $\eta_i$.
These determine the characteristic
function of $\Hat{A}$ which takes the form
\begin{equation*}
\varphi_{\Hat A(t)}(\xi)\,=\,\E^{-t\sum_{i=1}^d\eta_i\abs{\xi_i}^{\alpha}}\,,
\qquad \xi=(\xi_1,\dots,\xi_d)\in\Rd,\ t\ge0\,.
\end{equation*}
Note that if the arrival process of each class is renewal with regularly
varying interarrival times of parameter $\alpha$, then we obtain the
above limit process.

Next, we provide a representation of the generator of the process $\Hat{A}$. 
Let $\widetilde{\mathcal{N}} (\D{t},\D{y})$ be a martingale measure in
$\RR_*$, corresponding to a standard
Poisson random measure $\mathcal{N}(t,\D{y})$, and 
$\widetilde{\mathcal{N}}(t,\D{y})=\widehat{N}(t,\D{y})-t\mathcal{N}(\D{y})$,
with $\Exp\,\widetilde{\mathcal{N}}(t,\D{y})=t\mathcal{N}(\D{y})$,
and with $\mathcal{N}$ being a $\sigma$-finite measure on $\RR_*$ given
by $\mathcal{N}(\D{y}) = \nicefrac{\D{y}}{\abs{y}^{1+\alpha}}$.
Let $\widetilde{\mathcal{N}}_1,\dotsc,\widetilde{\mathcal{N}}_d$ be $d$ independent
copies of $\widetilde{\mathcal{N}}$.
We can then write 
\begin{equation*}
\D \Hat{A}_i (t) \,\df\,
\eta_i C(1,\alpha)\int_{\RR_*} y\,\widetilde{\mathcal{N}}_i (\D{t},\D{y})\,,
\end{equation*}
where
\begin{equation}\label{E-cda}
C(d,\alpha)\,\df\,\frac{\alpha 2^{\alpha-1}\varGamma(\nicefrac{\alpha+d}{2})}
{\pi^{\nicefrac{d}{2}}\varGamma(1-\nicefrac{\alpha}{2})}\,.
\end{equation}
Note that for $\alpha$ close to $2$ we have
$C(d,\alpha) \approx (2-\alpha)d$.

Thus, the generator $\mathcal{L}$ of the process $\Hat{A}$ then takes the form
\begin{equation*}
\mathcal{L}f(x) \,=\,C(1,\alpha)\sum_{i=1}^d\int_{\RR_*}
\mathfrak{d}_1 f(x; y_i e_i)\,\frac{\eta_i\,\D{y}_i}{\abs{y_i}^{1+\alpha}}
\,=\,\int_{\RR_*^d}
\mathfrak{d}_1 f(x; y)\,\nu(\D{y})\,,
\end{equation*}
where $\nu(\D{y})$ is of the form $\nu(\D{y}) = \sum_{i=1}^d\nu_i(\D{y}_i)$
with $\nu_i(\D{y}_i)$ supported on the $i^\text{th}$ coordinate axis in $\Rd$.
Recall that the characteristic function of an isotropic $\alpha$-stable process
$\process{L}$ has the form $\varphi_{L(t)}(\xi)=\E^{-t\eta \abs{\xi}^{\alpha}}$
for some $\eta>0$.
Thus, $\Hat A$ is not an isotropic $\alpha$-stable L\'evy process.
According to \cite[Theorem~2.1.5]{Samorodnitsky-Taqqu-Book-1994},
$\Hat{A}$ is a symmetric $d$-dimensional $\alpha$-stable L\'evy process.
Since it is not isotropic, it is not a subordinate Brownian motion with
$\nicefrac{\alpha}{2}$-stable subordinator, although each component $\Hat{A}_i$ is.

Let $X^n=(X^n_1,\dotsc,X^n_d)'$, $Q^n= (Q^n_1,\dotsc,Q^n_d)'$ and
$Z^n= (Z^n_1,\dotsc,Z^n_d)'$ be the processes counting the number of customers
of each class in the system, in queue, and in service, respectively. 
Then, it is evident that $X^n_i = Q^n_i + Z^n_i$ for each $i$ and
$\sum_{i=1}^d Z^n_i \le n$. 
We consider work-conserving scheduling policies that are non-anticipative and
allow preemption. 
Namely, no server will idle if there is any customer waiting in a queue, and
service of a customer can be interrupted at any time to serve some other
class of customers and will be resumed at a later time. 
Scheduling policies determine the allocation of service capacity, i.e.,
the $Z^n$ process, which must satisfy the condition that
$\langle e, Z^n\rangle = \langle e,X^n\rangle \wedge n$ at each time. 
Define the FCLT-scaled processes $\Hat{X}^n=(\Hat{X}^n_1,\dotsc,\Hat{X}^n_d)'$,
$\Hat{Q}^n= (\Hat{Q}^n_1,\dotsc,\Hat{Q}^n_d)'$ and
$\Hat{Z}^n= (\Hat{Z}^n_1,\dotsc,\Hat{Z}^n_d)'$ by
\begin{equation}\label{E-HatX}
\Hat{X}_i^n \,\df\, n^{-\nicefrac{1}{\alpha}} (X^n_i - \rho_i n)\,,
\ \Hat{Q}_i^n \,\df\, n^{-\nicefrac{1}{\alpha}} Q^n_i\,,\ \Hat{Z}_i^n
\,\df\, n^{-\nicefrac{1}{\alpha}} (Z_i^n - \rho_i n) \,.
\end{equation}
Then under the work-conserving preemptive scheduling policies, given the
controls $Z^n$, the processes $\Hat{Q}^n$ and $\Hat{Z}^n$ can be parameterized as 
follows: for adapted $\Hat{V}^n \in \varDelta$, 
\begin{equation*}
\Hat{Q}_i^n \,=\, \langle e,\Hat{X}^n\rangle^+ \Hat{V}^n_i\,,
\qquad \Hat{Z}_i^n \,=\, \Hat{X}^n_i - 
\langle e,\Hat{X}^n\rangle^+ \Hat{V}^n_i\,. 
\end{equation*}
The controls $\Hat{V}^n$ represent the fraction of class-$i$ customers in the queue when 
the total queue size is positive. When $\Hat{Q}^n\equiv 0$, we set $\Hat{V}^n = 
(0,\dotsc,0,1)'$. 
In the limit process, the control takes values in $\varDelta$, and will be regarded as 
a fixed parameter, i.e., this falls into the framework of our study when the control 
is constant. 
We obtain the following FCLT.

\begin{thm}\label{T4.1}
Under a fixed constant scheduling control $V \in \varDelta$,
provided there exists $\Hat{X}(0)$ such that
$\Hat{X}^n(0) \Rightarrow \Hat{X}(0)$ as $n\to\infty$,
 we have
\begin{equation}\label{ET4.1A}
\Hat{X}^n \;\Rightarrow\; \Hat{X} \qquad\text{in} \quad (D^d, M_1)
\quad \text{as\ } n \to\infty\,,
\end{equation}
where the limit process $\Hat{X}$ is a unique strong solution to the SDE
\begin{equation}\label{ET4.1B}
\D \Hat{X}(t) \,=\, \Hat{b}(\Hat{X}(t), V )\, \D{t} + \D \Hat{A}(t)
- \upsigma_{\alpha}\, \D W(t)\,,
\end{equation}
with an initial condition $\Hat{X}(0)$. Here the drift
$\Hat{b}(x,v)\colon \Rd \times \varDelta \to\Rd$ takes the form
\begin{equation}\label{ET4.1C}
\Hat{b}(x,v) \,=\, \Hat{\ell} - R(x - \langle e, x\rangle^+ v)
- \langle e, x\rangle^+ \varGamma v\,,
\end{equation}
with $R = \diag(\mu_1,\dotsc,\mu_d)$,
$\varGamma = \diag(\gamma_1,\dotsc,\gamma_d)$, and
$\Hat{\ell}\df (\Hat{\ell}_1,\dotsc,\Hat{\ell}_d)'$ for
$\Hat{\ell}_i$ in \cref{E4.1}.
In \cref{ET4.1B},
$\Hat{A}$ is the limit of the arrival process, $W$ is a standard $d$-dimensional 
Brownian motion, independent of $\Hat{A}$, and the covariance matrix $\upsigma_{\alpha}$ 
satisfies $\upsigma_{\alpha}\upsigma_{\alpha}' = \diag(\lambda_1,\dotsc,\lambda_d)$ if $
\alpha = 2$ and $\upsigma_{\alpha} = 0$ if $\alpha \in (1,2)$. 
\end{thm}

\begin{proof}
The FCLT-scaled processes $\Hat{X}^n_i$, $i=1,\dotsc,d$, can be represented as
\begin{align*}
\Hat{X}^n_i(t) \,=\, \Hat{X}^n_i(0) + \Hat{\ell}^n_i t
- \mu_i \int_0^t \Hat{Z}^n_i(s)\,\D{s} - 
\gamma_i \int_0^t \Hat{Q}^n_i(s)\, \D{s}
+ \Hat{A}^n_i(t) - \Hat{M}^n_{S,i}(t) - \Hat{M}
^n_{R,i}(t)\,,
\end{align*}
where $\Hat{\ell}^n_i$ is defined in \cref{E4.1}, 
\begin{equation*}
\Hat{M}^n_{S,i}(t)\,=\, n^{-\frac{1}{\alpha}} \bigg(S_i^n \bigg(\mu_i \int_0^t 
Z^n_i(s)\,\D{s} \bigg) - \mu_i \int_0^t Z^n_i(s) \D{s} \bigg) \,,
\end{equation*}
\begin{equation*}
\Hat{M}^n_{R,i}(t) \,=\, n^{-\frac{1}{\alpha}} \bigg(R_i^n \bigg(\theta_i \int_0^t 
Q^n(s)\,\D{s} \bigg) - \theta_i \int_0^t Q^n_i(s) \D{s} \bigg) \,,
\end{equation*}
and $S^n_i, R^n_i$, $i=1,\dotsc,d$, are mutually independent rate-one Poisson
processes, representing the service and reneging (abandonment), respectively. 
We can then establish an FCLT for the processes $\Hat{X}^n$, by following a similar 
argument as Theorem~2.1 in \cite{PW10}, if we prove the continuity in the Skorohod $M_1$ 
topology of the $d$-dimensional integral mapping $\phi\colon D^d \to D^d$ defined by
\begin{equation*}
y(t) \,=\, x(t) + \int_0^t h(y(s))\, \D{s}\,, \quad t\ge 0\,, 
\end{equation*}
where $h\colon \Rd\to \Rd$ is a Lipschitz function. In Theorem~1.1 of \cite{PW10}, the 
integral mapping is from $D$ to $D$, but a slight modification of the argument of that 
proof can show our claim in the multidimensional setting. Specifically, the parametric 
representations can be constructed in the same way with the spatial component being 
multidimensional, and the time component satisfying the conditions in Theorem~1.2 of 
\cite{PW10}. 
\end{proof}

In analogy to \cref{T3.2,T3.4,T3.5,T3.3}, we obtain \cref{C4.1} which follows.
For multiclass many-server queues, the model in \cref{T3.2} corresponds
to systems without abandonment, i.e., $\varGamma=0$.
In such systems, $M$ is a diagonal matrix, so the results in
\cref{T3.2} are more general than needed for the queueing models.

Recall the quantity $\Hat\rho$ defined in \cref{E4.2}.
As mentioned earlier, this quantity is the spare capacity of the network,
and when it is positive it amounts to the so-called $\sqrt{n}$
safety staffing for the network.
By \cref{T3.3} and \cref{L5.7}\,\ttup{a} both
$\alpha>1$ and $\Hat\rho>0$ are necessary for the process to be ergodic.
However, for the queueing model, the limit process has an interpretation
only if $\alpha>1$, and this is reflected in the statement of the
corollary.

\begin{cor}\label{C4.1}
For the multiclass many-server queues with heavy-tailed arrivals with
$\alpha \in (1,2)$, the following hold.
\begin{enumerate}[wide]
\item[\ttup1]
For the process $\process{\Hat{X}}$ in \cref{ET4.1B} to be ergodic
under some Markov control $v\in\Usm$, satisfying
$\varGamma v(x)=0$ a.e., it is necessary
and sufficient that $\Hat\rho>0$.
\item[\ttup2]
Suppose that $\Hat\rho>0$.
\begin{enumerate}[wide]
\item[\ttup{2a}]
The process
$\process{\Hat{X}}$ is polynomially ergodic
under any constant control satisfying $\varGamma v=0$, and its rate
of convergence is $r(t)\thickapprox t^{\alpha-1}$.
In addition, the conclusions of \cref{T3.2}\,\ttup{i} hold for any
$\theta\in[1,\alpha)$.
\item[\ttup{2b}]
For any Markov control $v\in\Usm$ satisfying $\varGamma v(x)=0$ a.e.,
which renders the process $\process{\Hat{X}}$ ergodic,
the associated invariant probability measure $\overline\uppi(\D{x})$ satisfies
\begin{equation*}
\int_\Rd \bigl(\langle e,x\rangle^+\bigr)^{\alpha-1}\,\overline\uppi(\D{x})\,=\,\infty\,,
\end{equation*}
and thus the queue is not stable.
In addition,
$\Hat\rho=\int_\Rd \langle e,x\rangle^-\,\overline\uppi(\D{x})$.
\end{enumerate}
\item[\ttup3]
For any constant control such that $\varGamma v \ne 0$, 
the conclusions of \cref{T3.5} apply for any  $\theta<\alpha$.
In addition, $\overline\uppi\in\cP_p(\Rd)$ for all $p<\alpha$,
and therefore, the queue is stable.
\end{enumerate} 
\end{cor}

\begin{proof}
The assertion in (1) follows by \cref{T3.3} and \cref{L5.7}\,\ttup{b}.
Item (2a) is a direct consequence of
\cref{T3.2}\,\ttup{i} and
\cref{T3.4}, while (2b) follows from \cref{L5.7}\,\ttup{b}
and \cref{C5.1}.
The assertion in (3) follows by \cref{T3.5}.
\end{proof}

We also remark that when the arrival limit is a Brownian motion ($\alpha=2$),
the limit is a diffusion with piecewise linear drift. 
In this case, the conclusions in \cref{C4.1}\,(2) hold for any
$\theta\in[1,\infty)$, and those in  \cref{C4.1}\,(3) hold for any
$\theta>0$,  and in both cases, we have exponential ergodicity. 
The basic reason behind this discontinuity at $\alpha=2$ is the
fact that the scaling constant $C(d,\alpha)$ of the fractional Laplacian
given in \cref{E-cda} tends to $0$ as $\alpha\nearrow2$, and thus
the singular integral in the generator $\mathcal{A}$ vanishes.
Comparing the tails of the stationary distributions $\overline\uppi(\D{x})$,
when $\alpha\in(1,2)$, as shown in
\cref{T3.4}, $\overline\uppi(\D{x})$ does not have any absolute moments
of order $\alpha-1$ or larger in case (2), and that this is true
under any Markov control $v=v(x)$.
In case (3), $\overline\uppi(\D{x})$ does not have any absolute moments
of order $\alpha$ or larger.

It is worth noting that the piecewise diffusion model $\Hat{X}$ in \cref{ET4.1B}
is more general than that considered in \cite{DG13}, as noted in \cref{R3.5}, 
and the rate of convergence is not identified there when $\varGamma=0$. 
For the multiclass $M/M/N+M$ queues with abandonment, 
exponential ergodicity of the limiting diffusion under the
constant control $v=(0,\dotsc,0,1)'$ is established in \cite[Theorem~3]{DG13},
and this is used in \cite{ABP14} to prove asymptotic optimality.
\cref{T3.5} extends this result, by asserting
exponential ergodicity under any constant control $v$ such that
$\varGamma v \ne 0$. 
We summarize these findings in the following corollary.

\begin{cor}
Assume $\alpha=2$. 
\begin{enumerate}[wide]
\item[\ttup{a}]
If $\varGamma v=0$, then $\Hat\rho>0$ is both  necessary and sufficient 
for the process to be ergodic, and in such a case,
\cref{ET3.2F} and \cref{ET3.2G} hold for for any $p>0$. 

\item[\ttup{b}]
If $\varGamma v\ne 0$, 
then \cref{ET3.5B,ET3.5C} hold for any $\theta>0$.
\end{enumerate}
In particular, in either case, $\process{\Hat{X}}$ is exponentially ergodic.
\end{cor}

\subsection{Multiclass \texorpdfstring{$G/M/n+M$}{G/M/n+M} queues with
service interruptions}\label{S4.2}

In \cite{PW09}, $G/M/n+M$ queues with service interruptions are studied in the H--W 
regime. It is shown that the limit queueing process is a one-dimensional
L{\'e}vy-driven SDE if the interruption times are asymptotically negligible.

We consider a sequence of multiclass $G/M/n+M$ queues in the same renewal alternating 
(up-down, or on-off) random environment, where all the classes of customers are affected 
simultaneously.
We make the same assumptions on the arrival, service and abandonment 
processes as well as the control processes as in \cref{S4.1}. 
For the random environment, we assume that the system functions normally during up time 
periods, and a portion of servers stop functioning during down periods, while customers 
continue entering the system and may abandon while waiting in queue and those 
that have started service will wait for the system to resume. Here we focus on the 
special case of all servers stopping functioning during down periods. 
Let $\{(u^n_k, d^n_k)\colon k\in \NN\}$ be a sequence of i.i.d. positive random vectors 
representing the up-down cycles. Assume that
\begin{equation*}
\bigl\{(u^n_k, n^{\frac{1}{\alpha}}d^n_k)\,\colon k\in \NN\bigr\}
\;\Rightarrow\; \bigl\{(u_k, d_k)\,\colon k\in \NN\bigr\} \qquad\text{in\ }
(\RR^2)^{\infty}\ \text{as\ } n\to\infty\,,
\end{equation*}
where $(u_k, d_k)$, $k\in \NN$, are i.i.d. positive random vectors
and $\alpha \in (1, 2]$.
This assumption is referred to as asymptotically negligible service interruptions. 
Define the counting process of down times,
$N^n(t) \df \max\{k\ge 0\colon T^n_k \le t\}$, where 
$T^n_k \df \sum_{i=1}^k (u^n_i+d^n_i)$ for each $k \in \NN$ and $T^n_0\equiv 0$. 
This assumption implies that $N^n \Rightarrow N$ in $(D, J_1)$ as $n\to\infty$, where 
the limit process is defined as $N(t) \df \max\{k\ge 0\colon T_k \le t\}$, $t\ge 0$, 
with $T_k \df \sum_{i=1}^k u_i$ for 
$k \in \NN$, and $T_0\equiv 0$. Here we assume that the process $\process{N}$ is Poisson. 

Let $X^n=(X^n_1,\dotsc,X^n_d)'$ be the processes counting the number of customers of 
each class in the system, and define the FCLT-scaled processes $\Hat{X}^n$ as in 
\cref{E-HatX}. Following a similar argument as in \cite{PW09} and \cite{PW10}, we can 
then show the following FCLT, whose proof is omitted for brevity.

\begin{thm}
Under a fixed constant scheduling control $V \in \varDelta$, 
if there exists $\Hat{X}(0)$ such that $\Hat{X}^n(0) \Rightarrow \Hat{X}(0)$
as $n\to\infty$, then \cref{ET4.1A} holds, 
where the limit process $\Hat{X}$ is a unique strong solution to the L{\'e}vy-driven 
SDE
\begin{equation*}
\D \Hat{X}(t) \,=\, \Hat{b}(\Hat{X}(t), V)\, \D{t}
+ \D \Hat{A}(t) - \upsigma_{\alpha}\, \D W(t) + 
c\, \D \Hat{J}(t)\,,
\end{equation*}
with initial condition $\Hat{X}(0)$.
The drift takes the same form as in \cref{ET4.1C} with 
$\Hat{\ell}_i$ in \cref{E4.1}, the matrices $\upsigma_{\alpha}$ and 
$R$ are as given in \cref{T4.1}, $c = (\lambda_1,\dotsc,\lambda_d)'$,
and the process 
$\Hat{J}$ is a compound Poisson process, defined by
\begin{equation*}
\Hat{J}(t) \,\df\, \sum_{k=1}^{N(t)} d_k\,, \quad t\ge 0\,.
\end{equation*}
\end{thm} 

Observe that the jump component $\process{\Hat{J}}$
is a one-dimensional spectrally positive pure-jump L\'evy process.
Hence, $\process{c\Hat{J}}$ should be regarded as the component
$\process{L_2}$ described in \cref{T3.1}\,(i) and (iv).
Let $\vartheta_d$ be the drift and $\nu_{\Hat J}(\D u)$ be the L\'evy measure
of $\process{\Hat{J}}$.
Clearly, $\vartheta_d \,=\,\eta\int_{\sB}u\delta(\D u)$, and
$\nu_{\Hat J}(\D u)=\eta\delta(\D u)$,
where $\eta>0$ is the rate of $\process{N}$ and $\delta(\D u)$ is the
distribution of $d_1$.
In this case, $\process{L_2}$ 
is determined by a L\'evy measure $\nu_2(\D{y})$ which is supported on
$C\df\{uc\colon u\ge0\}$ 
and satisfies $\nu_2(\D (uc))=\nu_{\Hat J}(\D u)$, and drift 
\begin{equation*}
\Hat{\vartheta}\,\df\,\vartheta_d c+\int_{\Rd}y(\Ind_{\{y\in C\colon \abs{y} \le1\}}(y)
-\Ind_{\{y\in C\colon \abs{y} \le \abs{c}\}}(y))\,\nu_2(\D{y})\,.
\end{equation*}
 Namely, we have 
\begin{align*}
\Exp\Bigl[  \E^{\imath\langle L_2(1), \, \xi \rangle}\Bigr]
&\,=\,
\Exp\Bigl[\E^{\imath\Hat{J}(1)\langle c,\, \xi \rangle}\Bigr]\\
&\,=\,\exp\biggl(\imath\vartheta_d\langle c,\, \xi \rangle+\int_{(0,\infty)}
\Bigl(\E^{\imath\langle c,\,\xi 
\rangle u} -\imath u\langle c,\, \xi \rangle\Ind_{\sB}(u)-1\Bigr)\,
\nu_{\Hat J}(\D u)\biggr) \\
&\,=\,\exp\Biggl(\imath\vartheta_d\langle c,\, \xi \rangle+\imath\int_{\Rd}y
\bigl(\Ind_{\{y\in C\colon \abs{y} \le1\}}(y)
-\Ind_{\{y\in C\colon \abs{y} \le \abs{c}\}}(y)\bigr)\,
\nu_2(\D{y})\\
&\mspace{160mu} +\int_{C}\
\Bigl(\E^{\imath\langle y,\xi \rangle } -\imath\langle y,\xi \rangle
\Ind_{\{y\in C\colon \abs{y}\le1\}}(y)-1\Bigr)\,\nu_2(\D{y})\Biggr)\,,
\end{align*} 
 where $\imath=\sqrt{-1}$.
When $\alpha=2$, the arrival limit is a Brownian motion, and thus,
we obtain a limit process as in case (C1).
When $\alpha \in (1,2)$, the arrival limit is an
anisotropic L\'evy process as in case (C3). 
In analogy to \cref{T3.2,T3.5},
we obtain the following corollary for cases (C1) and (C3).
Here, the spare capacity takes the form
\begin{equation*}
\Tilde\varrho\,\df\,  \Hat\rho - e'R^{-1}
\biggl(\Hat\vartheta+\int_{\sB^c}y\nu_2(\D{y})\biggr)\,.
\end{equation*}

\begin{cor}
Suppose that $\varGamma v=0$, and $\Tilde\varrho>0$.
In order for the process $\process{\Hat{X}}$ to be ergodic, 
it is necessary and sufficient that
\begin{enumerate}[wide]
\item[\ttup a]
$\alpha \in (1,2)$ and $\Exp[d_1^\theta]<\infty$ for some
$1\le\theta<\alpha$, or
\item[\ttup b]
$\alpha=2$ and $\Exp[d_1^\theta]<\infty$ for some $\theta\ge1$\,.
\end{enumerate}
If any of these conditions are met, the conclusions of
\cref{T3.2}\,\ttup{i} and \cref{T3.4} follow, i.e., the process
	$\process{\Hat{X}}$ is polynomially ergodic, and its rate
		of convergence is $r(t)\thickapprox t^{\theta_c-1}$.

On the other hand, if $\varGamma v\ne0$, then under either \ttup{a} with
$0 <\theta< \alpha$, or \ttup{b} with $\theta>0$,
the conclusions of \cref{T3.5} hold. 
\end{cor}

\subsection{Other Queueing Models}

An FCLT is proved 
in \cite{PR00, PR00E} for $GI/Ph/n$ queues with renewal
arrival processes 
and phase-type service-time distributions in the H-W regime, where the limit processes 
tracking the numbers of customers in service at each phase form a multidimensional 
piecewise-linear diffusion. In \cite{DHZ10},
$G/Ph/n+GI$ queues with abandonment are studied 
and a multidimensional piecewise-linear diffusion limit is also proved in the H-W regime. 
When the arrival process is heavy-tailed, satisfying an FCLT as in \cref{E-FCLT},
and/or when 
there are service interruptions, it can be shown that the limit processes are piecewise 
O-U processes with jumps as in \cref{E-sde}, where in the drift function the constant 
coefficient $\ell$ is replaced by $-lv$ for a constant
$l \in \RR$ and $v\in\varDelta$, and the vector $\varGamma v$
equals $c v$ for some constant $c\in\RR$. Our results include this
limiting process as a special case.

\section{Proofs of Main Theorems and Other Results} \label{S5}
In this section we prove the main results, with the exception of \cref{T3.1},
whose proof is in \cref{S-A}.

\subsection{Technical lemmas}

This section concerns some estimates for nonlocal operators
that we use in the proofs to establish Foster--Lyapunov equations.

For a $\upsigma$-finite measure $\nu(\D{y})$
on $\mathfrak{B}(\Rds)$, we let
\begin{equation*}
\mathfrak{J}_{1,\nu}[\Phi](x)\,\df\,\int_{\Rds}\mathfrak{d}_1 \Phi(x;y)\, \nu(\D{y})\,,
\quad\text{and}\quad
\mathfrak{J}_\nu[\Phi](x)\,\df\,\int_{\Rds}\mathfrak{d} \Phi(x;y)\, \nu(\D{y})\,,
\end{equation*}
with $\mathfrak{d}_1 \Phi(x,y)$ as defined in \eqref{frakd1},  and
\begin{equation*}
\mathfrak{d} f(x;y) \,\df\,
f(x+y)-f(x)-\langle y,\nabla f(x)\rangle\,,\qquad f\in C^1(\Rd)\,. 
\end{equation*} 
Also define
\begin{equation*}
\Breve{C}_0(r;\theta) \,\df\, \int_{\sB_r^c}\abs{y}^{\theta}\nu(\D{y})\,,\qquad
\widehat{C}_0 \,\df\, \int_{\sB\setminus\{0\}}\abs{y}^{2}\,\nu(\D{y})\,.
\end{equation*}
Note that $\Breve{C}_0(r;\theta)\to0$ as $r\to\infty$.

\begin{lem}\label{L5.1}
Suppose that $\nu(\D{y})$ is a $\upsigma$-finite measure on $\mathfrak{B}(\Rds)$,
which satisfies  $\widehat{C}_0+\Breve{C}_0(1;\theta)<\infty$ for some $\theta>0$.
We have the following.
\begin{enumerate}[wide]
\item[\ttup a]
If $\Phi\in C^2(\Rd)$ satisfies
\begin{equation}\label{EL5.1A}
\sup_{\abs{x}\ge1}\; \abs{x}^{1-\theta}\,\max\;\bigl(\abs{\nabla\Phi(x)},
\abs{x}\,\norm{\nabla^2\Phi(x)}\bigr)\,<\,\infty\,,
\end{equation}
then $\mathfrak{J}_\nu[\Phi]$
vanishes at infinity when $\theta\in[1,2)$, and the map
$x\mapsto(1+\abs{x})^{2-\theta}\,\mathfrak{J}_\nu[\Phi](x)$
is bounded when $\theta\ge2$.
\item[\ttup b]
If $\theta\in(0,1)$, and $\Phi\in C^2(\Rd)$ satisfies
\begin{equation*}
\sup_{\abs{x}\ge1}\; \abs{x}^{-\theta}\,
\max\;\bigl(\abs{\Phi(x)},\abs{x}\,\abs{\nabla\Phi(x)},
\abs{x}^2\,\norm{\nabla^2\Phi(x)}\bigr)\,<\,\infty\,,
\end{equation*}
then the function
$x\mapsto\mathfrak{J}_{1,\nu}[\Phi](x)$
vanishes at infinity.
\end{enumerate}
\end{lem}

\begin{proof}  We first consider the case $\theta\in(1,2)$.
By \cref{EL5.1A} there exist positive constants $c_0$ ad $c_1$
such that
\begin{equation}\label{PL5.1A}
\begin{split}
\abs{\nabla\Phi(x)} &\,\le\, c_0\,\Ind_{\sB}(x)
+ c_1\,\abs{x}^{\theta-1}\,\Ind_{\sB^c}(x)\,,\\[3pt]
\norm{\nabla^2\Phi(x)} &\,\le\, c_0\,\Ind_{\sB}(x)
+ c_1\,\abs{x}^{\theta-2}\,\Ind_{\sB^c}(x)\,,
\end{split}
\end{equation}
for all $x\in\Rd$.
Let $z\colon[1,\infty)\to\RR_+$ be defined by
$z(r) \df r \bigl(\Breve{C}_0(r;\theta)\bigr)^{\nicefrac{1}{2(1-\theta)}}$.
Then $z(r)$ is a strictly increasing function, whose range is an interval
of the form $[z_0,\infty)$, $z_0>0$.
Let $r(z)$ denote the inverse of this map defined on the range of $z(r)$.
Then of course $r(z)\to\infty$ as $z\to\infty$
and we have
\begin{equation}\label{PL5.1Aa}
\Bigl(\tfrac{z}{r(z)}\Bigr)^{\theta-1}\,
\Breve{C}_0(r(z);\theta)\,=\,\sqrt{\Breve{C}_0(r(z);\theta)}
\,\xrightarrow[z\to\infty]{}\,0\,,
\text{\ and\ }\frac{r(z)}{z}\,\xrightarrow[z\to\infty]{}\,0\,.
\end{equation}
We split the integral as follows: 
\begin{align}\label{PL5.1B}
\int_{\Rds}\mathfrak{d}\Phi(x;y)\, \nu(\D{y})
&\,=\int_{\sB_{r(\abs{x})}\setminus\{0\}}\int_0^1(1-t)
\bigl\langle y,\nabla^2\Phi(x+ty)y\bigr\rangle\, \D{t}\,\nu(\D{y})\\
&\mspace{100mu}+\int_{\sB^c_{r(\abs{x})}}\int_0^1
\bigl\langle y,\nabla \Phi(x+ty)-\nabla \Phi(x)\bigr\rangle\, \D{t}\,\nu(\D{y})\,.
\nonumber
\end{align}
Let $\Bar{r}>0$ be such that $z \ge r(z) + 1$ for
all $z\ge\Bar{r}$.
We estimate the integrals in \cref{PL5.1B} for $x\in\sB_{\Bar{r}}^c$.
For the first integral on the right hand side of \cref{PL5.1B}, we use
the estimate in \cref{PL5.1A} which implies that
\begin{equation*}
\langle y, \nabla^{2}\Phi(z) y\bigr\rangle
\,\le\, c_1\,\abs{z}^{\theta-2}\,\abs{y}^2\,,
\qquad y\in\Rd\,,\ z\in\sB^c\,,
\end{equation*}
to write
\begin{align*}
\bigl\langle y,\nabla^2 \Phi(x+ty)y\bigr\rangle &\,\le\,
c_1\,\abs{x+ty}^{\theta-2}\,\abs{y}^2 \\
&\,\le\,c_1\,\bigl(\abs{x}-r(\abs{x})\bigr)^{\theta-2}\,
\abs{y}^2 \\
&\,=\,c_1\,\Bigl(r(\abs{x})\bigl(\tfrac{\abs{x}}{r(\abs{x})}-1\bigr)\Bigr)^{\theta-2}
\abs{y}^2\\
&\,\le\,c_1\,\Bigl(\tfrac{\abs{x}}{r(\abs{x})}-1\Bigr)^{\theta-2}\,
\abs{y}^{\theta}\,,
\end{align*}
where the last inequality follows since $\abs{y}\le r(\abs{x})$.
So integrating with respect to $\nu(\D{y})$, we deduce that
the first integral is bounded by
\begin{align}\label{PL5.1C}
\tfrac{1}{2}\,c_1\,\Bigl(\tfrac{\abs{x}}{r(\abs{x})}-1\Bigr)^{\theta-2} &
\biggl(r(\abs{x})^{\theta-2}\, 
\int_{\sB\setminus\{0\}}\abs{y}^2\,\nu(\D{y})
+ \int_{\sB_{r(\abs{x})}\setminus\sB}\abs{y}^\theta\,\nu(\D{y})\biggr)\\
&\,\le\,
\tfrac{1}{2}\,c_1\,\widehat{C}_0\,
\bigl(\abs{x}-r(\abs{x})\bigr)^{\theta-2}+
\tfrac{1}{2}\,c_1\,\Breve{C}_0(r(\abs{x});\theta)
\Bigl(\tfrac{\abs{x}}{r(\abs{x})}-1\Bigr)^{\theta-2}\,.\nonumber
\end{align}
We use the inequality
$\abs{y}\le \bigl(r(\abs{x})\bigr)^{1-\theta}\,\abs{y}^\theta$
on $\sB^c_{r(\abs{x})}$, to derive the estimate
\begin{align*}
\babs{\bigl\langle y,\nabla \Phi(x+ty)-\nabla \Phi(x)\bigr\rangle} &\,\le\,
\abs{y}\,\bigl(c_0+c_1\abs{x+ty}^{\theta-1}+c_1\abs{x}^{\theta-1}\bigr)\\
&\,\le\, \abs{y}\, \Bigl(c_0+c_1\,
\bigl(2\abs{x}^{\theta-1}+\abs{y}^{\theta-1}\bigr)\Bigr)\\
&\,\le\, \bigl(r(\abs{x})\bigr)^{1-\theta}\,\abs{y}^\theta\,
\bigl(c_0+2c_1\,
\abs{x}^{\theta-1}\bigr) +\abs{y}^{\theta}\,.
\end{align*}
Integrating this with respect to $\nu(\D{y})$, we obtain a bound
for the absolute value of second integral on the right hand side of \cref{PL5.1B},
which takes the form
\begin{equation}\label{PL5.1D}
\Bigl[c_0(r(\abs{x})\bigr)^{1-\theta}+c_1\,
\Bigr(1+2\Bigl(\tfrac{\abs{x}}{r(\abs{x})}\Bigr)^{\theta-1}\Bigr)\Bigr]\,
\Breve{C}_0(r(\abs{x});\theta)\,.
\end{equation}
Combining \cref{PL5.1Aa,PL5.1C,PL5.1D}, we obtain a bound for
$\abs{\mathfrak{J}_\nu[\Phi](x)}$ that
clearly vanishes as $\abs{x}\to\infty$ when $\theta\in(1,2)$.

For $\theta=1$, we select $r(\abs{x})=\nicefrac{1}{2}\abs{x}$ and follow the same method.
For $\theta\ge2$, we select $r(z)=z$
and we use the bounds (for $t\in[0,1]$, and $\abs{x}\ge2$)
\begin{equation*}
\bigl\langle y,\nabla^2 \Phi(x+ty)y\bigr\rangle
\,\le\, c_1\,2^{\theta-2}\abs{x}^{\theta-2}\,\abs{y}^2\,,
\qquad\text{when\ } \abs{y}\le \abs{x}-1\,,
\end{equation*}
and 
\begin{equation*}
\bigl\langle y,\nabla \Phi(x+ty)\bigr\rangle \,\le\, \abs{y}\,
\bigl(c_0+2^{2\theta-4}+2^{\theta-2}(2^{\theta-2}+1)c_1\,\abs{y}^{\theta-1}\bigr)\,,
\end{equation*}
when $\abs{y}\ge\abs{x}-1$, to obtain the result as stated.
This completes the proof of part (a).

We continue with part (b).
Here, in addition to \cref{PL5.1A}, we have the bound
\begin{equation*}
\abs{\Phi(x)} \,\le\, c_0\,\Ind_{\sB}(x) + c_1\,\abs{x}^{\theta}\,\Ind_{\sB^c}(x)
\qquad\forall\,x\in\Rd\,.
\end{equation*}
Further, since $\int_{\sB^c}\abs{x}^\theta\, \nu(\D{y})<\infty$, by
the de la Vall\'ee-Poussin theorem, there exists a nonnegative
increasing convex function $\phi\colon\RR_+\to\RR_+$ with $\nicefrac{\phi(t)}{t}\to\infty$
as $t\to\infty$, such that
$\int_{\sB^c}\phi\bigl(\abs{y}^\theta\bigr)\, \nu(\D{y})<\infty$.
Without loss of generality, we may assume $\phi(t)>t$ for all $t\in\RR_+$.
Let $r(t) \df \bigl(\phi^{-1}(t^\theta)\bigr)^{\nicefrac{1}{\theta}}$. 
Clearly, $r\colon\RR_+\to\RR_+$ is increasing, $r(t)\to\infty$,
and  $\nicefrac{r(t)}{t}\to0$
as $t\to\infty$.
Now, we have 
\begin{align}\label{PL5.1F}
\int_{\Rds}\mathfrak{d}_1\Phi(x;y)\, \nu(\D{y})
&\,=\,\int_{\sB_{r(\abs{x})}\setminus\{0\}}\int_0^1(1-t)
\bigl\langle y,\nabla^2\Phi(x+ty)y\bigr\rangle\, \D{t}\,\nu(\D{y})\\
&\mspace{40mu}+\int_{\sB^c_{r(\abs{x})}}\bigl(\Phi(x+y)-\Phi(x)\bigr)\,\nu(\D{y})
+\int_{\sB_{r(\abs{x})}\setminus\sB}\langle y,\nabla \Phi(x)\rangle\,\nu(\D{y})\,.
\nonumber
\end{align}
For the first integral on the right hand side of \cref{PL5.1F}, we use
the bound derived in part~(a).
For the last integral on the right hand side of \cref{PL5.1F} (for $\abs{x}\ge1$),
we use the bound
$\bigl\langle y,\nabla \Phi(x)\bigr\rangle \le \abs{y}\,c_1\abs{x}^{\theta-1}$.
Thus,
\begin{align*}
\int_{\sB_{r(\abs{x})}\setminus\sB}\bigl\langle y,\nabla \Phi(x)\bigr\rangle
\,\nu(\D{y}) &\,\le\,
c_1\abs{x}^{\theta-1}\int_{\sB_{r(\abs{x})}\setminus\sB}\abs{y}\,\nu(\D{y})\\
&\,\le\, c_1\abs{x}^{\theta-1}r(\abs{x})^{1-\theta}
\int_{\sB_{r(\abs{x})}\setminus\sB}\abs{y}^\theta\,\nu(\D{y})\\
&\,\le\, c_1\abs{x}^{\theta-1}r(\abs{x})^{1-\theta}\Breve{C}_0(1;\theta)\,, 
\end{align*}
which tends to $0$ as $\abs{x}$ tends to $\infty$.
Lastly, for the second integral on the right hand side of \cref{PL5.1F},
we proceed as follows.
First, in view of the bound of $\Phi(x)$ (for $\abs{x}\ge1$), we have
\begin{equation*}
\int_{\sB^c_{r(\abs{x})}}\Phi(x)\,\nu(\D{y})
\,\le\,c_1\abs{x}^\theta\,\int_{\sB^c_{r(\abs{x})}}\nu(\D{y})\,\le\,
c_1\int_{\sB^c_{r(\abs{x})}}\phi\bigl(\abs{y}^\theta\bigr)\,\nu(\D{y})\,,
\end{equation*}
which tends to $0$ as $\abs{x}$ tends to $\infty$.
Second, since
\begin{equation*}
\Phi(x+y)\,\le\, c_0\,\Ind_{\sB}(x+y)
+ c_1\,\abs{x+y}^{\theta}\,\Ind_{\sB^c}(x+y)
\,\le\, c_0+2c_1\,\phi\bigl(\abs{y}^\theta\bigr)
\end{equation*}
for $y\in\sB_{r(\abs{x})}^c$,
we obtain
\begin{equation*}
\int_{\sB^c_{r(\abs{x})}}\Phi(x+y)\nu(\D{y})
\,\le\, c_0\int_{\sB^c_{r(\abs{x})}}\,\nu(\D{y})+
2c_1\int_{\sB^c_{r(\abs{x})}}\phi\bigl(\abs{y}^\theta\bigr)\,\nu(\D{y})\,,
\end{equation*}
which also tends to  $0$ as $\abs{x}$ tends to $\infty$.
This completes the proof.
\end{proof}

Recall the notation $\widetilde{V}_{Q,\theta}(x)$ from \cref{N3.1}.

\begin{lem}\label{L5.2}
Suppose that $\nu(\D{y})$ satisfies
\begin{equation*}
\int_{\Rds}\bigl(\abs{y}^{2}\,\Ind_{\sB\setminus{\{0\}}}(y)+\E^{\theta\abs{y}}\,
\Ind_{\sB^c}(y)\bigr)\,\nu(\D{y}) \,<\,\infty
\end{equation*}
for some $\theta>0$.
Then
$x\mapsto\bigl(1+\widetilde{V}_{Q,\theta}(x)\bigr)^{-1}\,
\mathfrak{J}_\nu[\widetilde{V}_{Q,\theta}](x)$ is bounded on $\Rd$.
\end{lem}

\begin{proof}
We estimate $\mathfrak{J}_\nu[\widetilde{V}_{Q,\theta}](x)$
by using the first integral on
the right hand side of \cref{PL5.1B} for $y\in\sB\setminus{\{0\}}$,
while for $y\in\sB^c$, we estimate the integral using the identity
\begin{equation*}
\mathfrak{d} \widetilde{V}_{Q,\theta}(x;y)\,=\,
\widetilde{V}_{Q,\theta}(x+y)
+\Bigl[\bnorm{x}_Q^{-1}\,\Bigl(\bnorm{\tfrac{\theta}{2}y}_Q^2
- \bnorm{x+\tfrac{\theta}{2}y}_Q^2\Bigr)+\bnorm{x}_Q-1\Bigr]
\widetilde{V}_{Q,\theta}(x)
\end{equation*}
for $x\in\sB^c$.
\end{proof}

In the proof of \cref{T3.2,T3.5}, we apply the results in \cref{L5.1,L5.2}. 
It is worth noting that in the special case of SDEs driven by
an isotropic $\alpha$-stable processes alone, sharper estimates than
\cref{L5.1} can be obtained (see the proof of Proposition~5.1 in \cite{ABC-16}).
We state such an estimate in \cref{L5.3} which follows. 
For $\Phi \in C^2(\Rd)$, and a positive vector
$\eta=(\eta_1,\dotsc,\eta_d)$, we define
\begin{align*}
\fI_\alpha[\Phi](x)&\,\df\,
\int_{\Rds}\mathfrak{d}_1 \Phi(x;y)\frac{ \D{y}}{\abs{y}^{\alpha+d}}\,,\\
\hfI_\alpha[\Phi](x)&\,\df\,
\sum_{i=1}^{d}\,\eta_i\,\int_{\RR_*}\mathfrak{d}_1
\Phi(x;y_i e_i)\frac{\D{y_i}}{\abs{y_i}^{\alpha+1}}\,,
\end{align*}
where $\mathfrak{d}_1 \Phi(x;y)$ is defined in \cref{frakd1},
and $e_i$ denotes a vector in $\Rd$ whose elements are all $0$,
except the $i^{\mathrm{th}}$ element which equals $1$.
Recall the notation $V_{Q,\delta}(x)$ from \cref{N3.1}.

\begin{lem}\label{L5.3}
The map $x\mapsto \abs{x}^{\alpha-\theta}\,\fI_\alpha[V_{Q,\theta}](x)$
is bounded on $\Rd$ for any $Q\in\cM_+$ and $\theta\in(0,\alpha)$.
The same holds for the anisotropic operator $\hfI_\alpha$.
\end{lem}

The following lemma,
whose proof follows from a similar argument 
to the one used in \cref{L5.1},
is not utilized in the proofs, but may be of independent interest. 
\begin{lem}
Assume the hypotheses of \cref{L5.1}\,\ttup{a}, but replace the
bound of $\nabla^2\Phi(x)$ in \cref{EL5.1A} by
$\sup_{x\in\Rd}\,\nicefrac{\norm{\nabla^2\Phi(x)}}{1+\abs{x}^{\gamma}}<\infty$
for $\gamma\in[0,\theta-1]$.
Then
\begin{equation*}
\limsup_{\abs{x}\to\infty}\; \abs{x}^{(1-\theta)(2-\theta+\gamma)}\,
\int_{\Rds}\mathfrak{d} \Phi(x;y)\, \nu(\D{y})\,<\,\infty\,.
\end{equation*}
\end{lem}

\begin{proof}
In the proof of \cref{L5.1}\,(a) we set
$r(z) = z^{\theta-1-\gamma}$.
The rest of the proof is the same.
\end{proof}

\subsection{Proofs of \texorpdfstring{\cref{T3.2,T3.3,T3.4}}{}}
We first state two lemmas needed for the proof. 
The first part of the lemma that follows is in \cite[Theorem~2]{DG13}.

\begin{lem}\label{L5.5}
Let $M$ be a nonsingular M-matrix such that $M'e\ge0$, and $v\in\varDelta$. 
There exists a positive definite matrix
$Q$ such that
\begin{equation}\label{EL5.5A} 
Q M + M' Q\;\succ\;0\,,\quad\text{and}\quad Q M(\mathbb{I} - ve')
+ (\mathbb{I}-ev')M'Q\,\succeq\,0\,.
\end{equation}
In addition,
\begin{equation} \label{EL5.5B}
(\Id-tev')M'Q + QM (\Id - tve')\,\succ\,0\,,\qquad t\in[0,1)\,.
\end{equation}
\end{lem}

\begin{proof}
We only need to prove \cref{EL5.5B}.
We argue by contradiction.
Let $S\df M'Q+QM$ and $T\df ev'M'Q+QMve'$.
Suppose that $x'(S-tT)x\le 0$ for some $t\in(0,1)$ and $x\in\Rd$, $x\ne0$.
Then, since $S\succ0$, we must have $x'(S-T)x<0$ which contradicts the hypothesis.
\end{proof}

Recall the constant $\Tilde\ell$ in \cref{E-ell},
and the cone $\cK_\delta$ in \cref{E-cone}.
Define
\begin{equation*}
\Tilde{b}(x)\,\df\, b(x) + \Tilde\ell -\ell\,,\qquad x \in \Rd\,.
\end{equation*}

\begin{lem}
Let $\Bar{\kappa}_1>0$ be such that
$\langle x,(QM+M'Q)x\rangle\ge 2\Bar{\kappa}_1\abs{x}^2$ for all $x\in\Rd$.
Set $\delta=\nicefrac{1}{4}\Bar{\kappa}_1\,\abs{QMv}^{-1}$ and 
$\zeta = - \langle \Tilde\ell, Q v\rangle$.
Then, $\zeta>0$, and for $Q$ given in \cref{EL5.5A}, we have
\begin{equation*}
\bigl\langle \Tilde{b}(x), \nabla V_{Q,2}(x)\bigr\rangle
\,\le\, \begin{cases}
\Bar{\kappa}_0 - \Bar{\kappa}_1\,\abs{x}^2\,,&\text{if\ } x\in\cK_\delta^c\,,\\[3pt]
\Bar{\kappa}_0- \delta\zeta\abs{x}\,,&\text{if\ } x\,\in\cK_\delta\,,
\end{cases}
\end{equation*}
for some constant $\Bar{\kappa}_0>0$.
Consequently, there are positive constants $\kappa_0$ and $\kappa_1$, such that
\begin{equation}\label{EL5.6A}
\bigl\langle \Tilde{b}(x), \nabla V_{Q,2}(x)\bigr\rangle
\,\le\, \kappa_0
- \kappa_1\,V_{Q,2}(x)\,\Ind_{\cK_\delta^c}(x)
- \kappa_1\,V_{Q,1}(x)\,\Ind_{\cK_\delta}(x)
\end{equation}
for all $x\in\Rd$.
\end{lem}

\begin{proof}
Assume first that $x\in\cK_\delta^c$. We have 
\begin{align*}
\bigl\langle \Tilde{b}(x), \nabla V_{Q,2}(x)\bigr\rangle 
&\,=\, 2 \langle\Tilde{\ell}, Qx\rangle - \langle x,(QM+M'Q)x\rangle
+2\langle x, QM v\rangle \langle e, x\rangle^+\\
&\,\le\, -2\Bar{\kappa}_1\abs{x}^2 + 2\abs{Q\Tilde\ell}\abs{x}
+2 \delta \abs{QMv}\abs{x}^2\,.
\end{align*}
Thus, by the definition of $\delta$, we obtain
\begin{equation*}
\bigl\langle \Tilde{b}(x), \nabla V_{Q,2}(x)\bigr\rangle
\,\le\, \Bar{\kappa}_0 - \Bar{\kappa}_1\,\abs{x}^2\,,\qquad \forall\,x\in\cK_\delta^c\,,
\end{equation*}
for some constant $\Bar{\kappa}_0>0$.

Now, assume that $x\in\cK_\delta$. We have 
\begin{equation*}
\bigl\langle \Tilde{b}(x), \nabla V_{Q,2}(x)\bigr\rangle 
\,=\, 2 \langle\Tilde{\ell}, Qx\rangle - \bigl\langle x,(Q M(\mathbb{I} - ve')
+ (\mathbb{I}-ev')M'Q)x\bigr\rangle \,.
\end{equation*}
We follow the technique in \cite{DG13} by using the unique
orthogonal decomposition
$x=\eta v + z$, with $z\in\Rd$ such that $\langle z,v\rangle=0$.
In other words, $z= x-\langle x,v\rangle \nicefrac{v}{\abs{v}^2}$.
As shown in (5.19) of \cite{DG13}, we have
\begin{equation}\label{EL5.6B}
\bigl\langle x,(Q M(\mathbb{I} - ve')+ (\mathbb{I}-ev')M'Q)x\bigr\rangle\,\ge\,
2\Hat\kappa_1\abs{z}^2
\end{equation}
for some $\Hat\kappa_1>0$.
Solving $v'Q M(\mathbb{I} - ve')=0$, which follows from \cref{EL5.6B}, we obtain
$$v'Q \,=\, \langle v,QMv\rangle e'M^{-1}\,.$$
Thus
$\langle\Tilde\ell, Q v\rangle =
 \langle v,QMv\rangle \langle e, M^{-1} \Tilde\ell\rangle$.
Since $\langle v,(QM+M'Q)v\rangle>0$, we have $\langle v,QMv\rangle>0$.
This implies that $\langle \Tilde\ell, Q v\rangle <0$.
Note that $\langle e,x\rangle =\eta \langle e,v\rangle + \langle e,z\rangle$,
and therefore
\begin{equation}\label{EL5.6C}
\eta\,=\, \langle e,x\rangle - \langle e,z\rangle\,.
\end{equation}
Using \cref{EL5.6B,EL5.6C}, and the
orthogonal decomposition of $x$, we obtain
\begin{align*}
\bigl\langle \Tilde{b}(x), \nabla V_{Q,2}(x)\bigr\rangle
&\,\le\, - 2\Hat\kappa_1 \abs{z}^2 + \langle\Tilde\ell, Q z\rangle
+ \eta \langle\Tilde\ell, Q v\rangle\\
&\,=\, - 2\Hat\kappa_1 \abs{z}^2 + \langle\Tilde\ell, Q z\rangle
+\zeta\langle e,z\rangle - \zeta \langle e,x\rangle\\
&\,\le\, \Hat\kappa_0 - \Hat\kappa_1\abs{z}^2 - \delta\zeta\abs{x}\\
&\,\le\, \Hat\kappa_0 - \delta\zeta\abs{x}\,,
\qquad x\,\in\cK_\delta\,,
\end{align*}
for some constant $\Hat\kappa_0>0$.
It is clear that \cref{EL5.6A} follows from these estimates.
\end{proof}

\begin{proof}[Proof of \cref{T3.2}]
We start with part (i).
Consider $V_{Q,\theta}(x)$, with $Q$ is given in \cref{EL5.5A}. 
Clearly, $V_{Q,\theta}(x)$ is an inf-compact function contained in $\mathcal{D}$.
By \cref{EL5.6A}, for any given $\theta>0$, there exist positive constants
$\kappa'_0$ and $\kappa'_1$, such that
\begin{align}\label{PT3.2A}
\bigl\langle \Tilde{b}(x), \nabla V_{Q,\theta}(x)\bigr\rangle
&\,=\, \bigl\langle \Tilde{b}(x), \nabla V_{Q,\theta}(x)\bigr\rangle\Ind_{\sB}(x)
+ \frac{\theta}{2}\,\frac{V_{Q,\theta-1}(x)}{V_{Q,1}(x)}\,
	\bigl\langle \Tilde{b}(x), \nabla V_{Q,2}(x)\bigr\rangle\Ind_{\sB^c}(x)\\
&\,\le\,\kappa'_0\Ind_{\sB}(x) - \kappa'_1\,V_{Q,\theta}(x)\,\Ind_{\cK_\delta^c}(x)
- \kappa'_1\,V_{Q,\theta-1}(x)\,\Ind_{\cK_\delta}(x)\nonumber
\end{align}
for all $x\in\Rd$.
By \cref{ET3.2A}, there exists some compact set $K\supset\sB$,
independent of $\theta$, such that
\begin{equation}\label{PT3.2B}
\trace\bigl(a(x)\nabla^2 V_{Q,\theta}(x)\bigr)
\,\le\,
\bigl\langle \Tilde{b}(x), \nabla V_{Q,\theta}(x)\bigr\rangle\,,
\qquad x\in K^c\,.
\end{equation}
First suppose $\theta\in[1,2]$. Then $\fJ[V_{Q,\theta}](x)$ is bounded
by \cref{L5.1}.
Thus, \cref{ET3.2C} holds with $c_1=\nicefrac{\kappa_1'}{2}$,
and for $c_0(\theta)$ we can use the sum of $\kappa_0'$, the supremum of
the left hand side of \cref{PT3.2B} on $K$, and a bound of $\fJ[V_{Q,\theta}](x)$.
When $\theta>2$, $(1+\abs{x})^{2-\theta}\,\mathfrak{J}_\nu[V_{Q,\theta}](x)$
is bounded
by \cref{L5.1}, and the result follows by comparing 
$\mathfrak{J}_\nu[V_{Q,\theta}](x)$ to
$\bigl\langle \Tilde{b}(x), \nabla V_{Q,\theta}(x)\bigr\rangle$ in \cref{PT3.2A}.

\Cref{ET3.2D} follows from
\cite[Theorems~3.2 and 3.4]{Douc-Fort-Guilin-2009}
and \cite[Theorem~5.1]{Meyn-Tweedie-AdvAP-III-1993} (for the case when $\theta=1$).

We now turn to part (ii). 
Consider $\widetilde{V}_{Q,p}(x)$, where $Q$
is given in \cref{EL5.5A}, and $p>0$ such that
$p\norm{Q}^{\nicefrac{1}{2}}<\theta$. 
We have 
\begin{equation*}
\bigl\langle \Tilde{b}(x), \nabla \widetilde{V}_{Q,p}(x)\bigr\rangle \,=\,
\bigl\langle \Tilde{b}(x), \nabla \widetilde{V}_{Q,p}(x)\bigr\rangle
\Ind_{\sB}(x) + p\,\E^{p\langle x,Q x\rangle^{\nicefrac{1}{2}}}
\frac{\bigl\langle \Tilde{b}(x),Qx\bigr\rangle}{\langle x,Qx\rangle^{\nicefrac{1}{2}}}\,
\Ind_{\sB^c}(x)
\end{equation*}
for all $x \in \Rd$. 
By \cref{L5.2},
it is clear that there exist constants $\Tilde{\kappa}_0>0$ and $\Tilde{\kappa}_1>0$,
such that
\begin{equation*}
\bigl\langle \Tilde{b}(x), \nabla \widetilde{V}_{Q,p}(x)\bigr\rangle
\,\le\, \Tilde{\kappa}_0\Ind_{\sB}(x)
- \Tilde{\kappa}_1 \widetilde{V}_{Q,p}(x)\,,
\qquad x\,\in\Rd\,. 
\end{equation*}
Thus we obtain \cref{ET3.2F}.
Finally, according to \cite[Theorem~6.1]{Meyn-Tweedie-AdvAP-III-1993}
(see also \cite[Theorem~5.2]{Down-Meyn-Tweedie-1995}) we conclude that
$\process{X}$ admits a unique invariant probability measure
$\Bar\uppi(\D{y})$ such that for any $\delta>0$ and $0<\gamma<\delta\Tilde c_1$, 
\begin{equation*}
\norm{\delta_x P^X_t(\D{y})-\overline\uppi(\D{y})}_{\widetilde{V}_{Q,p}}\,\le\,
C \widetilde{V}_{Q,p}(x)\, \E^{-\gamma \frac{t-\delta}{\delta}}\,,
\qquad x\in\Rd\,,\ t\ge0\,,
\end{equation*}
for some $C>0$.
\end{proof}

\begin{proof}[Proof of \cref{T3.3}]
We first consider the case $\Tilde\varrho<0$
(note that $\Tilde\ell$ depends on the noise present.
If the noise is only a Brownian motion,  then $\Tilde\ell=\ell$).
We use a common test function for all three cases.
In this manner, the result is established for any combination of
the driving processes (a)--(c).
We let
\begin{equation*}
G(t) \,\df\,\int_{-\infty}^t \frac{1}{\abs{s}^\gamma + 1}\,\D{s}\,,\qquad t\in\RR\,,
\end{equation*}
for an appropriately chosen constant $\gamma>1$,
and define $\tw\df (M^{-1})'e$, $\Hat{h}(x) \df \langle \tw,x\rangle$,
and $V(x)\df G\bigl(\Hat{h}(x)\bigr)$.
Then $\Tilde\varrho<0$ is equivalent to $\bigl\langle \tw, \Tilde\ell\bigr\rangle>0$. 
Note that the second derivative of $G(t)$ takes the form
\begin{equation*}
G''(t) \,=\,\pm \gamma \frac{\abs{t}^{\gamma-1}}{\bigl(\abs{t}^\gamma + 1\bigr)^2}\,,
\end{equation*}
where we use the positive sign for $t\le0$, and the negative sign for $t\ge0$.

Suppose \eqref{E-sde} is driven by a Brownian motion.
We select a constant $\beta>0$ such that
$\beta^{-1} >\gamma \langle \tw,  \tilde\ell\rangle^{-1}
\sup_{x\in\Rd}\babs{\upsigma'(x) \tw}^2$.
Let $V_\beta(x)=V(\beta x)$ for $\beta>0$.
An easy calculation shows that 
\begin{align}\label{PT3.3A}
\Ag^X V_\beta(x) &\,=\, \frac{1}{2}\trace\bigl(a(x)\nabla^2V_\beta(x)\bigr)
+ \langle \tilde b(x),\nabla V_\beta(x)\rangle \\
& \,\ge\, - \beta^2
\gamma \frac{\abs{\Hat{h}(\beta x)}^{\gamma-1}}
{\bigl(\abs{\Hat{h}(\beta x)}^\gamma + 1\bigr)^2}
\,\babs{\upsigma'(x) \tw}^2
+ \frac{\beta\,\bigl(\bigl\langle \tw, \Tilde\ell\bigr\rangle
+ \langle e,x\rangle^-\bigr)}{\abs{\Hat{h}(\beta x)}^\gamma + 1} \nonumber\\
&\,>\,0\qquad \forall\,x\in\Rd\,.\nonumber
\end{align} 
Thus, $\{V\bigl(\beta X(t)\bigr)\}_{t\ge0}$ is a bounded submartingale,
so it converges almost surely.
Since $\process{X}$ is irreducible, it can be either recurrent or transient. 
If it is recurrent, then $V(x)$ should
be constant a.e.\ in $\Rd$, which is not the case.
Thus $\process{X}$ is transient (for a different argument, also based on the
above calculation, see  \cite[Theorem 3.3]{Strame-Tweedie-1994}).

We next turn to the case that $L(t)$ is an $\alpha$-stable process
(isotropic or not).
Here, we select constants  $0<\delta<1$, and $1<\gamma<\delta\alpha$
(for example we can let $\delta=\nicefrac{1+\alpha}{2\alpha}$,
and $\gamma=\nicefrac{3+\alpha}{4}$).
We claim that, there exists a constant $C$ such that 
\begin{equation}\label{PT3.3B}
\babs{\fI_\alpha[V](x)} \,\le\,
 \frac{C}{\babs{\Hat{h}(x)}^\gamma + 1}\,,\qquad x\in\Rd\,,
\end{equation}
and that the same is true for the anisotropic kernel $\hfI_\alpha$.
Since $\alpha>1$
and
$\fI_\alpha[V_\beta](x)=\beta^\alpha\fI_\alpha[V](\beta x)$,
then given a bound as in \cref{PT3.3B} we may select $\beta>0$ and
sufficiently small, so that $\Ag^X V_\beta(x)>0$, and the rest follows
by the same argument used for the Brownian motion.

To obtain \cref{PT3.3B}, we proceed as follows.
First, note that the anisotropic case follows from the one-dimensional isotropic.
This is because the generator of the anisotropic process is a sum of generators of
one-dimensional isotropic processes.
Second, observe that it suffices to prove \cref{PT3.3B} in the one-dimensional
situation only. Namely, since $V\in C^{1,1}(\Rd)$ (recall that $\gamma>1$), we have 
\begin{equation*}
\fI_\alpha[V](x)\,=\,\frac{1}{2}
\int_{\Rds} \bigl(V(x+y) + V(x-y) - 2 V(x)\bigr)\,\frac{\D{y}}
{\abs{y}^{d+\alpha}}\,,\qquad x\in\Rd\,.
\end{equation*}
Here, $C^{1,1}(\Rd)$ denotes the class of $C^{1}$-functions whose partial
derivatives are Lipschitz continuous.
Since $V(x) = G(\langle\tw,x\rangle)$, it is constant
on each set  $\{x\in\Rd\,\colon \langle\tw,x\rangle=\text{constant}\}$.
Thus, without loss of generality we may
choose $x=\zeta\tw$, for $\zeta\in\RR$, and $\tw$ to have unit length.
Consider an orthonormal transformation of the coordinates via a unitary matrix $S$
so that the first coordinate of $\Hat{y} = Sy$ is along $\tw$.
Due to the invariance of the kernel under orthonormal transformations,
without loss of generality, we may choose $\tw=e_1'$.
Then $\langle\tw,x\rangle=x_1$, and
\begin{equation*}
\fI_\alpha[V](x)\,=\,\frac{1}{2}
\int_{\Rds} \bigl(G(x_1+y_1) + G(x_1-y_1) - 2 G(x_1)\bigr)\,
\frac{\D{y}}{\abs{y}^{d+\alpha}}\,,\quad x\in\Rd\,.
\end{equation*}
Finally, since 
\begin{align*}
\int_{\RR^{d-1}_*}\frac{\D{y}_2\dotsb\D{y}_d}
{\bigl(y_1^2+\dotsb+y_d^2\bigr)^{\nicefrac{(d+\alpha)}{2}}}
&\,=\,\abs{y_1}^{-d-\alpha}\int_{\RR^{d-1}_*}\frac{\D{y}_2\dotsb\D{y}_d}
{\bigl(1+\nicefrac{y_2^2}{y_1^2}
+\dotsb+\nicefrac{y_d^2}{y_1^2}\bigr)^{\nicefrac{(d+\alpha)}{2}}}\\
&\,=\,\frac{C}{\abs{y_1}^{1+\alpha}}\,,
\end{align*}
we conclude that
\begin{equation*}
\fI_\alpha[V](x)\,=\,\frac{C}{2}
\int_{\Rds} \bigl(G(x_1+y_1) + G(x_1-y_1) + 2 G(x_1)\bigr)\,
\frac{\D{y_1}}{\abs{y_1}^{1+\alpha}}\,,\quad x\in\Rd\,.
\end{equation*}
Now, let us prove \cref{PT3.3B} in the one-dimensional case.
We decompose the integral
as in \cref{L5.1}, choosing a cutoff radius $r(t)= t^\delta\vee 1$ for this purpose.
First, we write
\begin{align}\label{PT3.3C}
\fI_\alpha[V](x) &\,=\,\int_{\sB\setminus\{0\}}
\biggl(\int_0^1(1-t)  y^2 V''(x+ty)\, \D{t}\biggr)\,
\frac{\D{y}}{\abs{y}^{1+\alpha}}\\
&\mspace{100mu}+\int_{\sB_{r(\abs{x})}\setminus\sB}
\biggl(\int_0^1y V'(x+ty)\, \D{t}\biggr)\,
\frac{\D{y}}{\abs{y}^{1+\alpha}}\nonumber\\
&\mspace{200mu}+\int_{\sB^c_{r(\abs{x})}}
\biggl(\int_0^1y V'(x+ty)\, \D{t}\biggr)\,
\frac{\D{y}}{\abs{y}^{1+\alpha}}\,.\nonumber
\end{align}
We first bound the third integral in \cref{PT3.3C}.
Provided $y\ne0$ (and recall that without loss of generality we may assume
that $\tw=1$), we have
\begin{align}\label{PT3.3D}
\babss{\int_0^1yV'(x+ty)\, \D{t}}&\,\le\,
\int_{0}^1 \frac{\babs{ y}}
{\babs{ x+ty}^\gamma + 1}\,\D{t}\\
&\,\le\, \int_{-\infty}^\infty
\frac{1}{\babs{ x+s}^\gamma + 1}\,\D{s}
\,=\, \norm{G}_\infty\,.\nonumber
\end{align}
Using \cref{PT3.3D}, the absolute value of
the third integral in \cref{PT3.3C} has the bound 
\begin{equation}\label{PT3.3E}
\norm{G}_\infty\,
\int_{\sB^c_{r(\abs{x})}}\,\frac{\D{y}}{\abs{y}^{1+\alpha}}
\,=\, \frac{\kappa_0}{\bigl(\abs{x}^\delta\vee1\bigr)^{\alpha}}
\,\le\, \frac{\kappa_1}{\babs{ x}^\gamma + 1}
\,=\,\frac{\kappa_1}{\babs{\Hat{h}(x)}^\gamma + 1}
\end{equation}
for some positive constants $\kappa_0$ and $\kappa_1$.
Next, we bound the second integral in \cref{PT3.3C},
which we denote by $\fI_{\alpha,2}[V](x)$.
For  $\abs{y}\le \abs{x}^\delta$ and $
\abs{x}\ge 2^{\nicefrac{1}{1-\delta}}$, it holds that $2\abs{x}^\delta\,\le\,\abs{x}$, and
$ 2\abs{y}\,\le\,\abs{x}$. Thus, 
$\abs{x}\,\le\,
2\abs{x+ty}$ for all $t\in[0,1]$.
So  we have
\begin{align}\label{PT3.3F}
\babs{\fI_{\alpha,2}[V](x)}&\,\le\,\int_0^1
\biggl(\int_{\sB_{r(\abs{x})}\setminus\sB}
\frac{\abs{ y}} {\abs{x+ty}^{\gamma}+1}\,
\frac{\D{y}}{\abs{y}^{1+\alpha}}\biggr)\,\D{t} \\
&\,\le\,\int_0^1
\biggl(\int_{\sB_{r(\abs{x})}\setminus\sB}
\frac{2\abs{ y}}
{\abs{ x}^\gamma + 1}\,
\frac{\D{y}}{\abs{y}^{1+\alpha}}\biggr)\,\D{t}
\nonumber\\
&\,\le\,\frac{2}{\abs{\Hat{h}(x)}^\gamma + 1}
\,\int_{\sB^c}\frac{\D{y}}{\abs{y}^{\alpha}}
\nonumber\\
&\,=\, \frac{4}{(\alpha-1)\bigl(\abs{\Hat{h}(x)}^\gamma + 1\bigr)}\,.\nonumber
\end{align}
For $\abs{x}\le 2^{\nicefrac{1}{1-\delta}}$ we use the following bound 
\begin{equation}\label{PT3.3G}
\babs{\fI_{\alpha,2}[V](x)}\,\le\, \bnorm{\fI_{\alpha,2}[V]}_\infty\,\le\,
\frac{\bigl(2^{\nicefrac{\gamma}{1-\delta}}+1\bigr)\,
\bnorm{\fI_{\alpha,2}[V]}_\infty}{\abs{\Hat{h}(x)}^{\gamma}+1}\,.
\end{equation}
Finally, we bound the first integral in \cref{PT3.3C}.
We use the second derivative of $G(t)$, and the inequality
\begin{equation*}
\frac{\abs{z}^{\gamma-1}}{\bigl(\abs{z}^\gamma+1\bigr)^2}\,\le\,
\frac{1}{\abs{z}^{\gamma}+1}\,,\qquad z\in\RR\,,
\end{equation*}
to obtain
\begin{align*}
\babss{\int_0^1(1-t) y^2 V''(x+ty)\, \D{t}}
&\,\le\,\gamma\int_0^1(1-t)\frac{\abs{x+ty}^{\gamma-1}}
{\bigl(\abs{x+ty}^\gamma+1\bigr)^2}\abs{y}^2\D{t}\\
&\,\le\,\gamma\int_0^1(1-t)\frac{\abs{y}^2}{\abs{x+ty}^{\gamma}+1}\D{t}\,.
\end{align*}
For $x\in\RR$, $y\in\sB\setminus\{0\}$ and $t\in[0,1]$, we have
\begin{align*}
\abs{x}^\gamma+1&\,=\,\abs{x+ty-ty}^\gamma+1\\
&\,\le\,\kappa_0\abs{ x+ty}^\gamma+\kappa_0\abs{t y}^\gamma+1\\
&\,\le\, \kappa_0\abs{x+ty}^\gamma+\kappa_1+1\\
&\,\le\,
\kappa_2(\abs{x+ty}^\gamma+1)
\end{align*}
for some positive constants $\kappa_0$, $\kappa_1$ and $\kappa_2$.
Thus,
\begin{align}\label{PT3.3H}
\babss{\int_{\sB\setminus\{0\}}
\biggl(\int_0^1(1-t) y^2V''(x+ty)\, \D{t}\biggr)\,
\frac{\D{y}}{\abs{y}^{1+\alpha}}}
&\,\le\, \frac{\gamma\kappa_2}{\abs{\Hat{h}(x)}^\gamma+1}
\int_{\sB\setminus\{0\}}\abs{y}^2\,\frac{\D{y}}{\abs{y}^{1+\alpha}}\\
&\,=\,\frac{\kappa_3}{\abs{\Hat{h}(x)}^\gamma+1}\nonumber
\end{align}
for some positive constant $\kappa_3$.
The inequality in \cref{PT3.3B} now follows by combining 
\cref{PT3.3E,PT3.3F,PT3.3G,PT3.3H}.

We now consider case (c).
We follow the same approach, but here scaling with $\beta$ has to be
argued differently.
Here, we can choose any constant $\gamma>1$.
We need to establish that
\begin{equation}\label{PT3.3I}
\Ag^X V_\beta(x) \,=\, \fJ_\nu[V_\beta](x)
+\langle \Tilde{b}(x),\nabla V_\beta(x)\rangle \,>\,0\,,\qquad x\in\Rd\,.
\end{equation}
Recall that $\nu(\D{y})$ is supported on $\{tw\,\colon t\in[0,\infty)\}$.
Let $\tilde\nu(\D{t})=\nu(\D(tw))$,
and define
\begin{equation*}
H_\beta(t,x)\,\df\,\int_{0}^1
\biggl( \frac{\abs{\beta\langle\tw, x\rangle}^\gamma + 1}
{\abs{\langle\tw,\beta (x+stw)\rangle}^\gamma+1}-1\biggr)\,\D{s}\,.
\end{equation*}
We have
\begin{align*}
\fJ_\nu[V_\beta](x)&\,=\,\int_{[0,\infty)}\bigl(V_\beta(x+tw)-V_\beta(x)\bigr)\,
\tilde\nu(\D{t})
- \int_{[0,\infty)} \bigl\langle tw,\nabla V_\beta(x)\bigr\rangle \tilde\nu(\D{t})\\
&\,=\,\int_{[0,\infty)}\int_0^ 1\bigl\langle tw,\nabla V_\beta(x+stw)
-\nabla V_\beta(x)\bigr\rangle\,\D{s}\,\tilde\nu(\D{t})\\
&\,=\,\int_{[0,\infty)}\frac{\beta t\langle \tw,w\rangle}
{\babs{\beta\langle \tw,x\rangle}^\gamma+1}\,
H_\beta(t,x)\,\tilde\nu(\D{t})\,.
\end{align*}
Clearly, if $\langle \tw,w\rangle=0$, then \cref{PT3.3I} trivially holds.
Assume now that $\langle \tw,w\rangle>0$.
According to \cref{PT3.3A}, a sufficient condition for
\cref{PT3.3I} is
\begin{equation}\label{G_beta}
\limsup_{\beta\to0}\;\sup_{x\in\Rd}\;
\int_{[0,\infty)} t\,H_\beta(t,x)\,\tilde\nu(\D{t})\,=\,0\,.
\end{equation}
Clearly, $-1\le H_\beta(t,x)\le 0$ for $t\ge0$ and $x\in\Rd$, 
$\langle \tw,x \rangle\ge0$.
Also,
\begin{align*}
\int_{0}^1 \frac{-\abs{\beta st\langle\tw,w\rangle}^\gamma}
{\abs{\beta st\langle\tw,w\rangle}^\gamma+1}\,\D{s}
&\,\le\,
\int_{0}^1 \biggl(\inf_{x\in\Rd,\,\langle\tw,x\rangle\ge0}
\frac{\abs{\beta\langle\tw, x\rangle}^\gamma + 1}
{\abs{\langle\tw,\beta (x+stw)\rangle}^\gamma+1}-1\biggr)\,\D{s}\\
&\,\le\,
\inf_{x\in\Rd,\,\langle\tw,x\rangle\ge0}H_\beta(t,x)\\
&\,\le\,
\sup_{x\in\Rd,\,\langle\tw,x\rangle\ge0}H_\beta(t,x)\\
&\,\le\, \int_{0}^1
\biggl( \sup_{x\in\Rd,\,\langle\tw,x\rangle\ge0}
\frac{\abs{\beta\langle\tw, x\rangle}^\gamma + 1}
{\abs{\langle\tw,\beta (x+stw)\rangle}^\gamma+1}-1\biggr)\,\D{s}\,=\,0\,.
\end{align*}
 This, together with reverse Fatou lemma, gives \cref{G_beta}
 on the set $\{x\in\Rd\,\colon\langle \tw,x\rangle\ge0\}$. In particular, this means
 that there exists some $\beta_0>0$ such that
\begin{equation}\label{PT3.3J}
\int_{[0,\infty)} t\,H_{\beta_0}(t,x)\,\nu(\D{t})\,\ge\,
-\frac{\bigl\langle \tw, \Tilde\ell\bigr\rangle}{2\langle\tw,w\rangle}\,,
\qquad x\in\Rd,\ \langle\tw,x\rangle\ge0\,.
\end{equation}
On the other hand, if $\langle\tw,x\rangle\le0$, then
$H_{\beta_0}(s,x)\ge H_{\beta_0}(s,-x)$, so that \cref{PT3.3J} holds
for all $x\in\Rd$.
In turn, \cref{PT3.3J} implies that
\begin{equation*}
\fJ_\nu[V_{\beta_0}](x)
+\langle \Tilde{b}(x),\nabla V_{\beta_0}(x)\rangle\,\ge\,
\frac{1}{2}\,\frac{\beta_0\,\bigl(\bigl\langle \tw, \Tilde\ell\bigr\rangle
+ 2\langle e,x\rangle^-\bigr)}{\abs{\Hat{h}(\beta_0 x)}^\gamma + 1}>0\,,\qquad 
x\in\Rd\,.
\end{equation*}
Finally, if $\langle\tw,w\rangle<0$, we proceed analogously.
This finishes the proof of case (c).

We now turn to the case $\Tilde\varrho=0$.
Suppose that the process $\process{X}$ has an invariant probability measure
$\uppi(\D{x})$.
Let $h_{1,\beta}(x)$ and $h_{2,\beta}(x)$ denote the two terms on the right hand side of
\cref{PT3.3A}, in the order they appear.
Applying It\^o's formula to \cref{PT3.3A} we have
\begin{equation}\label{PT3.3K}
\Exp^\uppi\bigl[V\bigl(\beta X(t\wedge\uuptau_r)\bigr)\bigr] - V(\beta x)
\,\ge\, \sum_{i=1,2}\Exp^\uppi\biggl[\int_0^{t\wedge\uuptau_r}
h_{i,\beta}\bigl(X(s)\bigr)\,\D{s}\biggr]\,,
\end{equation}
where $\uuptau_r$ denotes the first exit time from $\sB_r$, $r>0$.
Note that $h_{1,\beta}(x)$ is bounded and $h_{2,\beta}(x)$ is nonnegative.
Thus we can  take limits in \cref{PT3.3K} as $r\to\infty$,
using  dominated and monotone
convergence for the terms on the right hand side, and obtain 
\begin{equation*}
\Exp^\uppi\bigl[V\bigl(\beta X(t)\bigr)\bigr] - V(\beta x)
\,\ge\, t \sum_{i=1,2}
h_{i,\beta}(x)\uppi(\D{x})\,,\qquad t\ge0\,.
\end{equation*}
Now,  divide both the terms
by $t$ and $\beta$, and take limits  as $t\to\infty$.
Since $V(x)$ is bounded, we have
$$\int_{\Rd}\beta^{-1} h_{1,\beta}(x)\uppi(\D{x})
+\int_{\Rd}\beta^{-1} h_{2,\beta}(x)\uppi(\D{x})\,\le\,0\,.$$
Since $\beta^{-1} h_{1,\beta}(x)$ tends to $0$ uniformly in $x$
 as $\beta\searrow0$, and is bounded, we must have
$$\int_{\Rd}\beta^{-1} h_{2,\beta}(x)\uppi(\D{x})\to0$$ as $\beta\searrow0$.
However, since $\beta^{-1} h_{2,\beta}(x)$ is bounded
away from $0$ is the open set $\{x\in\Rd\,\colon\langle e,x\rangle^->1\}$,
this is a contradiction in view of the fact that $\uppi(\D{x})$ has
full support (due to open-set irreducibility of $\process{X}$).
It is clear that the proof for the $\alpha$-stable (isotropic or not) and the L\'evy
are exactly the same, since $\Ag^X V(\beta x)$ shares the
same structural property in all these cases. This completes the proof of the theorem.
\end{proof}

\begin{remark}\label{R5.1}
Since $M$ is a  nonsingular $M$-matrix,
its eigenvalues have positive real part.
According to this, it is well known that the so-called Lyapunov
equation $SM+M'S=\Id$ admits a unique positive definite symmetric solution $S$
(which is given by $S=\int_0^\infty\E^{-M't}\E^{-Mt}\D{t}$).
Further, recall the definition in \cref{E-cone}, and  assume that $a(x)$
satisfies \cref{ET3.5A}. 
It is straightforward to show that for any $\theta\in\Theta_c$
there exist positive constants
$\Tilde{c}_i$, $i=0,1,2$, and $\Tilde{\delta}$, such that
	\begin{equation*}
	\Ag^X V_{S,\theta}(x)\,\le\,
	\Tilde{c}_0 - \Tilde{c}_1 \abs{x}^\theta\Ind_{\cK_{\Tilde\delta}^c}(x)
	+ \Tilde{c}_2 \abs{x}^\theta\Ind_{\cK_{\Tilde\delta}}(x)\,,
	\qquad \,x\in\Rd\,,
	\end{equation*}
and over all Markov controls $v\in\Usm$.
This implies that any invariant probability measure $\overline\uppi(\D{x})$ of
$\process{X}$ (if it exists)
	satisfies
	\begin{equation}\label{ER5.1}
	\int_\Rd \abs{x}^\theta\,\overline\uppi(\D{x})
	\,\le\, \frac{\Tilde{c}_0}{\Tilde{c}_1}
	+ \frac{\Tilde{c}_2}{\Tilde\delta\Tilde{c}_1}
	\int_\Rd \bigl(\langle e,x\rangle^+\bigr)^\theta\,\overline\uppi(\D{x})\,.
	\end{equation} 
	Thus, if the integral on the right hand side of \cref{ER5.1} is finite, 
	then $\overline\uppi\in\cP_\theta(\Rd)$.
\end{remark}

We need to introduce some notation, which we fix throughout the
rest of the paper.

\begin{notation}\label{N5.1}
We let $\chi(t)$ be a smooth concave function such that
$\chi(t)=t$ for $t\le-1$, and $\chi(t)=-\nicefrac{1}{2}$ for $t\ge0$.
Also $\Breve\chi(t) \df -\chi(-t)$.
Thus this is a convex function with $\Breve\chi(t)=t$ for $t\ge1$
and $\Breve\chi(t) =\nicefrac{1}{2}$ for $t\le 0$.
We scale $\chi(t)$ to $\chi_R(t)\df R+\chi(t-R)$, $R\in\RR$.
So $\chi_R(t)= t$ for $t\le R-1$ and
$\chi_R(t) = R-\nicefrac{1}{2}$ for $t\ge R$.

We recall the definitions of $\tw= (M^{-1})'e$,
and $\Hat{h}(x) = \langle \tw,x\rangle$, from the proof
of \cref{T3.3}, and additionally define
\begin{equation*}
 F(x) \,\df\, \Breve\chi\bigl(\Hat{h}(x)\bigr)\,,\quad
\text{and\ \ } F_{\kappa,R}(x) \,\df\, \chi_R\circ F^\kappa(x)\,,\quad
x\in\Rd\,,\ \kappa>0\,,\ R>0\,,
\end{equation*}
where $F^\kappa(x)$ denotes the $\kappa^{\mathrm{th}}$ power of $F(x)$.
\end{notation}

Recall that \cref{T3.3}, under the assumption that $1\in\Theta_c$,
shows that if $\Tilde\varrho\le0$, then the process $\process{X}$
does not admit an invariant probability measure 
under any Markov control satisfying $\varGamma v(x)=0$ a.e.
In \cref{L5.7} which follows, we show that the same is the case if $1\notin\Theta_c$.
Therefore, $\Tilde\varrho>0$ and $1\in\Theta_c$ are both necessary
conditions for the existence of an invariant probability measure.

\begin{lem}\label{L5.7}
Suppose that \eqref{E-sde} is driven by either or both of
\begin{enumerate}[wide]
\item[\ttup i]
An $\alpha$-stable process (isotropic or not).

\item[\ttup{ii}]
A L\'evy process with finite L\'evy measure
$\nu(\D{y})$ which is supported on a half-line in $\Rd$
of the form $\{tw\colon t\in[0,\infty)\}$,
with $\langle e,M^{-1} w\rangle>0$.

\end{enumerate}
A diffusion component may be present in the noise, 
in which case we assume the growth condition \cref{ET3.2A}.
Under these assumptions, the following hold. 
\begin{enumerate}[wide]
\item[\ttup a]
If $1\notin\Theta_c$, then
the process $\process{X}$ is not ergodic under any Markov control $v\in\Usm$
satisfying $\varGamma v(x)=0$ a.e.
\item[\ttup b]
Suppose $1\in\Theta_c$, and that under a control $v\in\Usm$
such that $\varGamma v(x)=0$ a.e., the process $\process{X}$ has an invariant
probability measure $\overline\uppi(\D{x})$ satisfying
$\int_\Rd\bigl(\langle \tw,x\rangle^+\bigr)^{p-1}\,\overline\uppi(\D{x})<\infty$
for some $p>1$.
Then, necessarily $p\in\Theta_c$ and $\Tilde\varrho>0$.
\item[\ttup c]
In general, if an invariant probability measure
$\overline\uppi(\D{x})$ (under some Markov control) satisfies
$\int_\Rd\bigl(\langle \tw,x\rangle^+\bigr)^{p}\,\overline\uppi(\D{x})<\infty$
for some $p\ge1$, then necessarily $p\in\Theta_c$.
\end{enumerate}
\end{lem}

\begin{proof}
Recall \cref{N5.1}.
Note that $F_{\kappa,R}(x)$ is smooth, bounded, and has bounded derivatives.
Thus, if $\process{X}$ is positive recurrent with invariant probability
measure $\overline\uppi(\D{x})$, we must have $\overline\uppi(\Ag^X F_{\kappa,R})=0$.
Note also that $F(x)$ is positive and bounded away from $0$.
For $f\in C^2(\Rd)$, let
\begin{equation}\label{PL5.7A}
\Ag_\circ^X f(x) \,\df\, \frac{1}{2}\trace\bigl(a(x)\nabla^2f(x)\bigr)
+\begin{cases}
\mathfrak{J}_{\nu}[f](x)&\text{if\ \ }1\in\Theta_c\,,\\[3pt]
\mathfrak{J}_{1,\nu}[f](x)&\text{otherwise\,.}
\end{cases}
\end{equation}
We have
\begin{equation}\label{PL5.7B}
\Ag^X F_{\kappa,R}(x) \,=\, \Ag_\circ^X F_{\kappa,R}(x) + 
\chi_R'\bigl(F^\kappa(x)\bigr)
\bigl\langle \Tilde{b}(x),\nabla F^\kappa(x)\bigr\rangle\,.
\end{equation}
Let
\begin{equation*}
\begin{aligned}
h_{\kappa}(x)&\,\df\,\kappa\,
\Breve\chi'\bigl(\Hat{h}(x)\bigr)\,F^{\kappa-1}(x)\,,\\
\Tilde{h}_{\kappa}(x)&\,\df\,h_{\kappa}(x)\, \langle e,x\rangle^- \,,\\
\widetilde{F}_{\kappa,R}(x)&\,\df\, 
 \frac{1}{2}\chi_R''\bigl(F^\kappa(x)\bigr)\, (h_{\kappa}(x))^2\,
 \babs{\upsigma'(x)\tw}^2\,.
\end{aligned}
\end{equation*}
A simple calculation shows that
for any control $v(x)$ satisfying $\varGamma v(x)=0$ a.e., it holds that
\begin{equation}\label{PL5.7C}
\bigl\langle \Tilde{b}(x), \nabla F^\kappa(x)\bigr\rangle 
\,=\,  h_{\kappa}(x) \,
\bigl(-\Tilde\varrho + \langle e,x\rangle^-\bigr)\,.
\end{equation} 
We also have
\begin{equation*}
\Ag_\circ^X F_{\kappa,R}(x) \,=\, 
\frac{1}{2}\chi_R'\bigl(F^\kappa(x)\bigr) \trace\bigl(a(x)\nabla^2F^\kappa(x)\bigr)
+ \widetilde{F}_{\kappa,R}(x)
+\begin{cases}
\mathfrak{J}_{\nu}[F_{\kappa,R}](x)&\text{if\ \ }1\in\Theta_c\,,\\[3pt]
\mathfrak{J}_{1,\nu}[F_{\kappa,R}](x)&\text{otherwise\,.}
\end{cases}
\end{equation*}

Integrating \cref{PL5.7B} with respect to $\overline\uppi(\D{x})$, and using
\cref{PL5.7C} to rearrange terms, we obtain
\begin{equation}\label{PL5.7D}
\Tilde\varrho\,\overline\uppi\bigl(\chi_R'(F^\kappa)h_{\kappa}\bigr) \,=\,
\overline\uppi\bigl(\chi_R'\bigl(F^\kappa\bigr)\Ag_\circ^XF^\kappa\bigr)
+ \overline\uppi\bigl(\chi_R'(F^\kappa)\Tilde{h}_{\kappa}\bigr)
+\overline\uppi\bigl(\mathscr{J}_R[F^\kappa]\bigr)
+\overline\uppi\bigl(\widetilde{F}_{\kappa,R}\bigr)\,,
\end{equation}
with
\begin{equation*}
\mathscr{J}_{R}[F^\kappa](x)\,\df\, \begin{cases}
\mathfrak{J}_{\nu}[F_{\kappa,R}](x) - \chi_R'\bigl(F^\kappa(x)\bigr)
\mathfrak{J}_{\nu}[F^\kappa](x)&\text{if\ \ }1\in\Theta_c\,,\\[5pt]
\mathfrak{J}_{1,\nu}[F_{\kappa,R}](x) - \chi_R'\bigl(F^\kappa(x)\bigr)
\mathfrak{J}_{1,\nu}[F^\kappa](x)&\text{otherwise\,.}
\end{cases}
\end{equation*}
Note that we can always select $\chi_R(t)$
so that $\chi''_R(t)\ge - \nicefrac{1}{t+1}$, for $t\ge 0$.
Thus, in view of \cref{ET3.2A}, there exists some positive constant $\Tilde{C}$
such that
\begin{equation}\label{PL5.7E}
\babs{\widetilde{F}_{\kappa,R}(x)}\,\le\, \Tilde{C}\bigl(1 + F^{\kappa-1}(x)\bigr)\,,
\qquad x\in\Rd\,,\  R>0\,.
\end{equation}

Let $\kappa\in\Theta_c$, $\kappa\le1$.
Then the functions
$\mathscr{J}_{R}[F^\kappa](x)$ and $\widetilde{F}_{\kappa,R}(x)$ are
bounded, uniformly in $R$, and
converge to $0$, on compact sets as $R\to\infty$.
Thus, we can take limits in \cref{PL5.7D}
as $R\to\infty$, to obtain
\begin{equation}\label{PL5.7G}
\overline\uppi\bigl([\Ag_\circ^X F^\kappa]^-\bigr)
+ \Tilde\varrho\, \overline\uppi(h_{\kappa})
\,=\, \overline\uppi\bigl([\Ag_\circ^X F^\kappa]^+\bigr)
+ \overline\uppi\bigl(\Tilde{h}_{\kappa}\bigr)\,.
\end{equation}

We next prove part (a) of the lemma.
First consider the process in (ii).
It is clear that the map $x\mapsto\trace\bigl(a(x)\nabla^2F^\kappa(x)\bigr)$ is
bounded on $\Rd$, uniformly in $\kappa\in(0,1)$, and
\begin{equation}\label{PL5.7H}
\mathfrak{J}_{1,\nu}[F^\kappa](x)\,\ge\,\int_{\sB}\int_0^1(1-t)
\bigl\langle y,\nabla^2F^\kappa(x+ty)y\bigr\rangle\, \D{t}\,\nu(\D{y})\,,
\end{equation}
since $\langle \tw, w\rangle>0$.
Thus $[\Ag_\circ^X F^\kappa]^-(x)$ is bounded uniformly in $\kappa\in(0,1)$.
If $\Theta_c=(0,\theta_c)$,
then $\inf_{x\in\sB}\,\Ag_\circ^X F^\kappa(x)\to\infty$
as $\kappa\nearrow\theta_c$, and contradicts \cref{PL5.7G}, which
is valid for all $\kappa\in\Theta_c$. 
In the event that $\Theta_c=(0,\theta_c]$,  we express the integral of
\cref{PL5.7B} as
\begin{equation}\label{PL5.7I}
\overline\uppi\bigl(\bigl[\Ag_\circ^X F_{\kappa,R}\bigr]^-\bigr)
+\Tilde\varrho\,\overline\uppi\bigl(\chi_R'(F^\kappa)h_{\kappa}\bigr)\,=\,
\overline\uppi\bigl(\bigl[\Ag_\circ^X F_{\kappa,R}\bigr]^+\bigr)
+ \overline\uppi\bigl(\chi_R'(F^\kappa)\Tilde{h}_{\kappa}\bigr)\,,
\end{equation}
and evaluate \cref{PL5.7I} at any $\kappa\in(\theta_c,1]$.
Again, $\mathfrak{J}_{1,\nu}[F_{\kappa,R}](x)$ has the bound in \cref{PL5.7H}, implying
that $\bigl[\Ag_\circ^X F_{\kappa,R}(x)\bigr]^-$  is
uniformly bounded over $R\in(0,\infty)$. 
Thus, we can take limits in \cref{PL5.7I}
as $R\to\infty$, to reach the same contradiction.

In the case of the process in (i), a straightforward calculation,
using the estimates in the proof of \cref{L5.1}, and the one-dimensional
character of the singular integral as exhibited in
the proof of \cref{T3.2}, shows that there exists a positive constant $C_1$ such that
$\bigl[\fI_\alpha[F^\kappa]\bigr]^-(x)\le \nicefrac{C_1}{\alpha-\kappa}$ for all $x\in\Rd$.
This bound can be easily obtained by
using the first integral in \cref{PT3.3C} over $\Rd$ instead of $\sB$,
to compute $\fI_\alpha[F^\kappa](x)$.
Let
\begin{equation*}
D_{r_0}\df\{x\in\Rd\colon\, 0\le\langle\tw,x\rangle\le r_0\}\,,
\qquad
\widetilde{D}_{r_0}\df\{x\in\Rd\colon\, r_0<\langle\tw,x\rangle\}\,,
\end{equation*}
We fix any $r_0>1$, and estimate $\fI_\alpha[F^\kappa](x)$ for
$x\in D_{r_0}$.
The part from the first integral in \cref{PT3.3C} over $\sB$ is
bounded below uniformly in $\kappa\in(0,\alpha)$ by some constant $-C_2$.
We evaluate the remaining part by using the third integral in \cref{PT3.3C}
with $r(\abs{x})=1$.
Taking advantage of the one-dimensional character of this integration,
and assuming without loss of generality that
$\langle\tw,e_1\rangle=1$, in order to simplify the notation,
we obtain
\begin{equation*}
\int_{-\infty}^{-1}
\bigl(F^\kappa(x+te_1)-F^\kappa(x)\bigr)\,\frac{\D{t}}{\abs{t}^{1+\alpha}}\,\ge\,
-F^\kappa(x)\, \int_{-\infty}^{-1}\frac{\D{t}}{\abs{t}^{1+\alpha}}
\,\ge\, -C_3\,  r_0\,,
\end{equation*}
for a positive constant $C_3$ independent of $r_0$ or $\kappa$.
Using the inequality $\abs{s+t}^\kappa - \abs{s}^\kappa
\ge 2^{-\nicefrac{1}{\kappa}} \abs{t}^\kappa - \bigl(1- 2^{-\nicefrac{1}{\kappa}}\bigr)
\abs{s}^\kappa$, which is valid for any $0\le s\le t$ in $\RR$, we have
\begin{align*}
\int_{1}^{\infty} \bigl(F^\kappa(x+te_1)-F^\kappa(x)\bigr)\,
\frac{\D{t}}{\abs{t}^{1+\alpha}}&\,\ge\, 
2^{-\frac{1}{\kappa}}\,\int_{1}^{\infty}
\abs{t}^\kappa\,\frac{\D{t}}{\abs{t}^{1+\alpha}}
- \bigl(1- 2^{-\frac{1}{\kappa}}\bigr)\, r_0^\kappa
\int_{1}^{\infty}\,\frac{\D{t}}{\abs{t}^{1+\alpha}}\\
&\,\ge\, \frac{2^{-\nicefrac{1}{\alpha}}}{\alpha-\kappa}
- C_3\, r_0\,,
\end{align*}
where $C_3$ is the same constant used earlier.
Let $r_0$ be large enough so that
$\overline\uppi(\widetilde{D}_{r_0})\le\nicefrac{1}{2^{\nicefrac{\alpha+1}{\alpha}}C_1}
\overline\uppi( D_{r_0})$.
It is clear that $\fI_\alpha[F^\kappa](x)$ is nonnegative
if $\langle\tw,x\rangle\le0$, for all $\kappa$ sufficiently close to $\alpha$.
Therefore, combining the above, we have
\begin{align*}
\int_{\Rd}\fI_\alpha[F^\kappa](x)\,\overline\uppi(\D{x})
&\,\ge\,\int_{D_{r_0}}\fI_\alpha[F^\kappa](x)\,\overline\uppi(\D{x})
+\int_{\widetilde{D}_{r_0}}\fI_\alpha[F^\kappa](x)\,\overline\uppi(\D{x})\\
&\,\ge\, \overline\uppi( D_{r_0})\,
\biggl(\frac{2^{-\nicefrac{1}{\alpha}}}{\alpha-\kappa}- C_2-2C_3\,r_0\biggr)
- \overline\uppi(\widetilde{D}_{r_0})\,\frac{C_1}{\alpha-\kappa} \\
&\,\ge\,
\overline\uppi( D_{r_0})\,
\biggl(\frac{2^{-1-\nicefrac{1}{\alpha}}}{\alpha-\kappa}- C_2-2C_3\,r_0\biggr)\,,
\end{align*}
and we obtain a contradiction in \cref{PL5.7G} by letting
$\kappa\nearrow\alpha$.
Note that since $\tw$ has nonnegative components,
the preceding argument applies to both the isotropic and anisotropic
L\'evy kernels.
This completes the proof of part (a).

We now turn to part (b).
Suppose first that $\int_\Rd\abs{x}^{\theta_c-1}\,\overline\uppi(\D{x})<\infty$,
and $\Theta_c=(0,\theta_c)$.
We apply $\Ag^X$ to $F_{\kappa,R}(x)$, for $\kappa\in(1,\theta_c)$,
and note again that the function $\mathscr{J}_R[F^\kappa](x)$  is
bounded, uniformly in $R$, and converges to $0$, on compact sets as $R\to\infty$.
Moreover, since $\kappa\ge1$, and $\tw$ has nonnegative components
we have $\Ag_\circ^X F^\kappa\ge0$ by convexity.
Thus, taking limits as $R\to\infty$, \cref{PL5.7G} takes the form
\begin{equation}\label{PL5.7J}
\Tilde\varrho\, \overline\uppi(h_{\kappa})
\,=\, \overline\uppi\bigl(\Ag_\circ^X F^\kappa\bigr)
+ \overline\uppi\bigl(\Tilde{h}_{\kappa}\bigr)\,.
\end{equation}
It is thus immediately clear that $\Tilde\varrho>0$.
It also follows from \cref{PL5.7J} that
$\overline\uppi(h_{\kappa})\to\infty$ as $\kappa\nearrow\theta_c$, which contradicts
the original hypothesis that
$\int_\Rd\abs{x}^{\theta_c-1}\,\overline\uppi(\D{x})<\infty$.

Next, we consider the case $\Theta_c=(0,\theta_c]$.
If the $\alpha$-stable component is present, then necessarily $\theta_c<\alpha$,
and we can follow the technique used for part (a).
If a diffusion part is present we argue as follows.
We suppose that $\int_\Rd\abs{x}^{\theta_c-1+\epsilon}\,\overline\uppi(\D{x})<\infty$,
for some $\epsilon>0$,
and select $\kappa=\theta_c+\epsilon$.
Then $\overline\uppi\bigl(\widetilde{F}_{\kappa,R}\bigr)$ vanishes
as $R\to\infty$ by \cref{PL5.7E} and the hypothesis.
Thus, again we obtain from \cref{PL5.7J} that
$\overline\uppi(h_{\kappa})\to\infty$ as $R\to\infty$, which contradicts
the original hypothesis that
$\int_\Rd\abs{x}^{\theta_c-1+\epsilon}\,\overline\uppi(\D{x})<\infty$.
This, together with \cref{R5.1} proves (b).

Concerning part (c), in general, under a Markov control
$v\in\Usm$, \cref{PL5.7C} takes the form
\begin{equation*}
\bigl\langle \Tilde{b}(x), \nabla F^p(x)\bigr\rangle 
\,=\, p \Breve\chi'\bigl(\Hat{h}(x)\bigr)\,F^{p-1}(x)\,
\bigl(-\Tilde\varrho + \langle e,x\rangle^-
-\langle\tw,\varGamma{v}(x)\rangle\,\langle e,x\rangle^+\bigr)\,,
\end{equation*}
and following through the  earlier calculations for $p\ge1$,
we see that \cref{PL5.7I} now takes the form
\begin{align}\label{PL5.7L}
\overline\uppi\bigl(\Ag_\circ^X F_{p,R}\bigr)
+ \overline\uppi\bigl(\chi_R'(F^p)\Tilde{h}_{p}\bigr)
&\,\le\, 
\Tilde\varrho\, \overline\uppi\bigl(\chi_R'(F^p)h_{p}\bigr)\\
&\quad+\int_\Rd \chi_R'(F^p(x))p \Breve\chi'\bigl(\Hat{h}(x)\bigr)\,F^{p-1}(x)
\langle\tw,\varGamma{v}(x)\rangle\,\langle e,x\rangle^+\,
\overline\uppi(\D{x})\nonumber\,.
\end{align}
Taking limits in \cref{PL5.7L} as $R\to\infty$, the right hand side reaches
a finite value by the hypothesis and \cref{R5.1}.
In view of the decomposition of $\Ag_\circ^X F_{p,R}(x)$
in \cref{PL5.7D}, and the growth estimate in \cref{PL5.7E},
this implies that $\overline\uppi\bigl(\Ag_\circ^X F^p\bigr)<\infty$, from which
we deduce that $p\in\Theta_c$.
\end{proof}

\begin{remark}
It can be seen from the proof of \cref{L5.7},
that the conclusions of parts (b) and (c) are still valid if we replace the
$\alpha$-stable process in (i) with a subordinate Brownian motion,
such that $\theta_c<\infty$.
\end{remark}

\begin{remark}\label{R5.3}
The hypothesis $\langle \tw,w\rangle=\langle e,M^{-1} w\rangle>0$
in \cref{L5.7} (ii) cannot, in general, be relaxed.
The following example demonstrates this.
Let $M=\Id$, any constant control $v\in\varDelta$,
and $\langle \tw, w\rangle\le0$.
Consider the function  $V(x) = V_{Q,\theta}(x) +
\Tilde{V}(x)$, with  $\Tilde{V}(x)=\phi(\beta\langle e,x\rangle^+)$,
$\beta>0$,
and $\phi(t)=\E^t$ for $t\ge1$, and $\phi(t)=1$ for $t\le 0$, and smooth.
It is straightforward to verify that
$\Ag^X \Tilde{V}(x)=0$ if $\langle e,x\rangle\le0$,
and that $\Ag^X \Tilde{V}(x)\le \kappa_0' - \kappa_1'\Tilde{V}(x)$
if $\langle e,x\rangle\ge0$, where $\kappa_0'$ and $\kappa_1'$
are positive constants.
Adding this inequality to \cref{ET3.2C}, and since for some positive constants
$C$ and $R$ we have $\Tilde{V}(x)\ge C V(x)$ for $x\in\cK_\delta\cap\sB_R^c$,
we obtain
$\Ag^X V(x) \le \kappa_0'' - \kappa_1''V(x)$ for some positive
constants $\kappa_0''$ and $\kappa_1''$, and all $x\in\Rd$.
This shows that the process is exponentially ergodic.
This is an example where the direction of jumps is beneficial to
ergodicity.
\end{remark}

\begin{cor}\label{C5.1}
Assume the hypotheses of \cref{T3.1}, and \cref{ET3.2A}.
Then, every invariant probability measure $\overline\uppi(\D{x})$,
corresponding to a Markov control $v\in\Usm$ such that  $\varGamma{v}(x)=0$ a.e.,
under which the process $\process{X}$ is ergodic, satisfies
\begin{equation}\label{EC5.1A}
\Tilde\varrho\,=\,\int_\Rd \langle e,x\rangle^-\,\overline\uppi(\D{x})\,.
\end{equation}
\end{cor}

\begin{proof}
We scale $\Breve\chi(t)$ by defining $\Breve\chi_r(t)\df \Breve\chi(r+t)-r$
for $r\in\RR$.
Then, we observe that \cref{PL5.7D} holds with $\kappa=1$ and
$F^\kappa(x)$ replaced by $\Breve\chi_r\circ\Hat{h}(x)$ for any $r\in\RR$. 
Note that when $\kappa=1$,
all the integrands in \cref{PL5.7D} are uniformly bounded in $r\in(0,\infty)$,
and we have $\mathcal{A}_\circ^X\Breve\chi_r\circ\Hat{h}(x)\ge0$ due to convexity.
Moreover,
$\lim_{r\to\infty}\overline{\uppi}(\mathcal{A}_\circ^X\Breve\chi_r\circ\Hat{h})=0$,
$\lim_{r\to\infty}\overline\uppi(\Breve\chi_r\circ\Hat{h})=1$,
and $\lim_{r\to\infty}\overline\uppi(\Breve\chi_r(\cdot)
\langle e,\cdot\,\rangle^-)=\overline\uppi(\langle e,\cdot\,\rangle^-)$.
Thus, taking limits as $r\to\infty$, we obtain \cref{EC5.1A}. 
\end{proof}

\begin{remark}
\Cref{C5.1} has important implications for queueing systems in
the Halfin--Whitt regime.
Suppose that $\varGamma=0$, so that jobs do not abandon the queues.
Let $v\in\Usm$ be a scheduling control under which the process
is ergodic.
Then \cref{EC5.1A} asserts that the mean idleness in the
servers equals the spare capacity.
Note that there is a certain stiffness implied by this.
The mean idleness does not depend on the particular Markov control.
\end{remark}

\begin{proof}[Proof of \cref{T3.4}] Part (b) follows from
\cref{T3.2}\,(i), \cref{L5.7}, \cref{R5.1}, and \cref{C5.1}.

We continue with part (a).
The upper bounds follow from \cite[Theorem~3.2]{Douc-Fort-Guilin-2009}.
We next exhibit a lower bound for the rate of convergence.
Consider first the case when the SDE is driven by an $\alpha$-stable process.
We apply \cite[Theorem~5.1]{Hairer-16}, and use the same notation to help the reader.
We choose $G(x) = F^{\alpha-\epsilon}(x)$, for arbitrary $\epsilon\in(0,\alpha-1)$.
We have shown in \cref{L5.7} that
$\overline\uppi(F^{\alpha-\epsilon})=\infty$.
Further, from the Lyapunov equation in \cref{ET3.2C} by It\^o's formula,
and setting $\theta=\alpha-\epsilon$, we have
\begin{equation*}
\Exp^{x}\bigl[V_{Q,\alpha-\epsilon}\bigl(X(t)\bigr)\bigr]
- V_{Q,\alpha-\epsilon}(x) \,\le\, c_0(\alpha-\epsilon) t\,,\qquad x\in\Rd\,.
\end{equation*} 
Dominating $F^{\alpha-\epsilon}(x)$ with $V_{Q,\alpha-\epsilon}(x)$ and write
\begin{equation*}
\Exp^{x}\bigl[F^{\alpha-\epsilon}\bigl(X(t)\bigr)\bigr]\,\le\,
C_1\bigl(c_0(\alpha-\epsilon) t+ V_{Q,\alpha-\epsilon}(x)\bigr)\;=:\;g(x,t)
\end{equation*}
for some positive constant $C_1$.
Next, we compute a lower bound for
$\overline\uppi\bigl(\{x\colon G(x)\ge t\}\bigr)$.
From \cref{PL5.7I}, with $\kappa=1$, we have
\begin{equation*}
\Tilde\varrho\,\overline\uppi\bigl(\chi_t'(F)h_1\bigr)\,=\,
\overline\uppi\bigl(\Ag_\circ^X F_{1,t}\bigr)
+ \overline\uppi\bigl(\chi_t'(F)\Tilde{h}_1\bigr)\,.
\end{equation*}
Subtracting this equation from \cref{PL5.7J}, we obtain
\begin{equation*}
\Tilde\varrho\, \overline\uppi(h_1-\chi_t'(F)h_1)
\,=\, \overline\uppi\bigl(\Ag_\circ^X (F-F_{1,t})\bigr)
+ \overline\uppi\bigl(\Tilde{h}_1-\chi_t'(F)\Tilde{h}_1\bigr)\,.
\end{equation*}
Note that all the terms are nonnegative.
Moreover, $\Ag_\circ^X (F-F_{1,t})(x)$ is nonnegative by convexity,
and thus
\begin{align*}
\overline\uppi\bigl(\Ag_\circ^X (F-F_{1,t})\bigr)
&\,\ge\,\inf_{x\in\sB}\,\bigl(\Ag_\circ^X (F-F_{1,t})(x)\bigr)
\,\overline\uppi(\sB)\\
&\,\ge\,\Ag_\circ^X (F-F_{1,t})(0)\,\overline\uppi(\sB)\,.
\end{align*}
Recalling the definitions of the functions in \cref{N5.1}, 
it is then evident that
\begin{align}\label{PT3.4A}
\overline\uppi\bigl(\{x\colon \langle\tw,x\rangle > t\}\bigr)
&\,\ge\,\overline\uppi(h_1-\chi_t'(F)h_1)\\
&\,\ge\,\Tilde\varrho^{-1}\,\overline\uppi(\sB)
\Ag_\circ^X (F-F_{1,t})(0)\nonumber\\
&\,\ge\, C_2 t^{1-\alpha}\,.\nonumber
\end{align}
Therefore, by \cref{PT3.4A}, we have
\begin{align*}
\overline\uppi\bigl(\{x\colon G(x)\ge t\}\bigr)&\,=\,
\overline\uppi\bigl(\{x\colon (\langle\tw,x\rangle)^{\alpha-\epsilon} > t\}\bigr)\\
&\,=\, \overline\uppi\bigl(\{x\colon \langle\tw,x\rangle
> t^{\frac{1}{\alpha-\epsilon}}\}\bigr)  \\
&\,\ge\, C_2 t^{\frac{1-\alpha}{\alpha-\epsilon}}\;=:\;f(t)\,.
\end{align*}
We solve $y f(y) = 2g(x,t)$ for $y=y(t)$, to obtain
$y=\bigl(C_2^{-1}2g(x,t)\bigr)^{\nicefrac{\alpha-\epsilon}{1-\epsilon}}$, and
\begin{equation*}
f(y)\,=\, C_2\bigl(C_2^{-1}2g(x,t)\bigr)^{\frac{1-\alpha}{1-\epsilon}}\,=\,
C_3 \bigl(c_0(\alpha-\epsilon) t
+ V_{Q,\alpha-\epsilon}(x)\bigr)^{\frac{1-\alpha}{1-\epsilon}}\,,
\end{equation*}
with
\begin{equation*}
C_3\,\df\,  \bigl(2C_1\bigr)^{\frac{1-\alpha}{1-\epsilon}}
C_2^{\frac{\alpha-\epsilon}{1-\epsilon}}\,.
\end{equation*}
Therefore, by \cite[Theorem~5.1]{Hairer-16}, and since $\epsilon$ is arbitrary, we have
\begin{align}\label{PT3.4B}
\bnorm{\delta_{x} P^X_t(\D{x}) - \overline\uppi(\D{x})}_\tv
&\,\ge\, f(y) -\frac{g(x,t)}{y}\\
&\,=\,\frac{C_3}{2} \bigl(c_0(\alpha-\epsilon) t
+ V_{Q,\alpha-\epsilon}(x)\bigr)^{\frac{1-\alpha}{1-\epsilon}}\nonumber
\end{align}
for all $t\ge0$ and $\epsilon\in(0,\alpha-1)$.

We next derive a suitable estimate for the constant $c_0(\alpha-\epsilon)$
as a function of $\epsilon$, for a fixed choice of $Q$. 
First, evaluating at $x=0$, it follows from \cref{ET3.2C} that
$c_0(\alpha-\epsilon)\ge \fI_\alpha[V_{Q,\alpha-\epsilon}](0)$.
Thus, for some positive constants $\kappa_0$   and $\kappa_1$ independent
of $\epsilon$, we have
\begin{equation*}
c_0(\alpha-\epsilon)\,\ge\,
\fI_\alpha[V_{Q,\alpha-\epsilon}](0)\,\ge\,
\int_{\sB^c} \bigl(V_{Q,\alpha-\epsilon}(y)-V_{Q,\alpha-\epsilon}(0)\bigr)\,
\frac{\D{y}}{\abs{y}^{d+\alpha}}
\,\ge\, \frac{\kappa_0}{\epsilon} +\kappa_1\,.
\end{equation*}
On the other hand, as shown in the proof of \cref{T3.2}, $c_0(\theta)$
can be selected as the sum of the supremum of $\fJ_\nu[V_{Q,\theta}](x)$ on $\Rd$,
and a constant that does not depend on $\theta$.
An upper bound for $\fI_\alpha[V_{Q,\alpha-\epsilon}](x)$ can be obtained
from the proof of \cref{T3.2} by using \cref{PL5.1Aa,PL5.1C,PL5.1D},
together with the fact that the radius $\Bar{r}$ defined in the proof
is bounded over the range of $\epsilon$.
However, we follow a more direct approach.
Note that $\nabla V_{Q,\theta}(x)$ is H\"older continuous with exponent
$\theta-1$ on $\Rd$,
i.e., it satisfies
$\abs{\nabla V_{Q,\theta}(x)-\nabla V_{Q,\theta}(y)}\le \kappa\abs{x-y}^{\theta-1}$
for a positive constant $\kappa$ which is independent of $\theta\in(1,2]$.
Using this property, and the fact that the second derivatives
of $V_{Q,\theta}(x)$ are uniformly bounded for $\theta\in(1,2]$,
decomposing the integral as in \cref{PL5.1B}, we obtain
\begin{align*}
\babs{\fI_\alpha  [V_{Q,\alpha-\epsilon}](x)}
&\,\le\, \tilde\kappa_0 +
\babss{\int_{\sB^c}\int_0^1\bigl\langle y,\nabla V_{Q,\alpha-\epsilon}(x+ty)
-\nabla V_{Q,\alpha-\epsilon}(x)\bigr\rangle\, \D{t}\,
\frac{\D{y}}{\abs{y}^{d+\alpha}}}\\
&\,\le\, \tilde\kappa_0 + \frac{\tilde\kappa_1}{\epsilon}
\end{align*}
for some positive constants $\tilde\kappa_0$ and $\tilde{\kappa}_1$
which do not depend on $\epsilon$.
Using this estimate in \cref{PT3.4B}, we obtain the lower bound in \cref{ET3.4A}.

For a L\'evy  process in (ii), following \cref{PT3.4A},
we obtain
\begin{equation}\label{PT3.4E}
\overline\uppi\bigl(\{x\colon \langle\tw,x\rangle > t\}\bigr)
\,\ge\, C_2 \int_{\{\langle\tw,x\rangle\ge t\}} \abs{x}\,\nu(\D{x})\,.
\end{equation}
Since \cref{PT3.4E} does not give rise to an explicit estimate
as in \cref{PT3.4A}, we apply \cite[Corollary~5.2]{Hairer-16}.
For $\epsilon\in(0,\nicefrac{1}{3})$, we define
\begin{equation*}
W(x)\,\df\, F^{\theta_c-1+\epsilon}(x)\,,
\qquad \widehat{F}(t) \,\df\, t^{\frac{\theta_c-\epsilon}{\theta_c-1+\epsilon}}\,,
\qquad h(t) \,\df\, t^{-1-\frac{\epsilon}{\theta_c-1+\epsilon}}\,,
\end{equation*}
and
\begin{equation*}
g(x,t)\,\df\,
C_1\bigl(c_0(\theta_c-\epsilon) t+ V_{Q,\theta_c-\epsilon}(x)\bigr)\,.
\end{equation*}
Then the hypotheses in \cite[Corollary~5.2]{Hairer-16} are satisfied.
By the preceding definitions, we have
$\widehat{F}(t)h(t) = t^{\nicefrac{1-3\epsilon}{\theta_c-1+\epsilon}}$.
Thus
\begin{equation*}
(\widehat{F}\cdot h)^{-1}(y)
\,=\, y^{\frac{\theta_c-1+\epsilon}{1-3\epsilon}}\,,
\qquad \text{and}\qquad  h\bigl((\widehat{F}\cdot h)^{-1}(y)\bigr) \,=\,
y^{-\frac{\theta_c-1+2\epsilon}{1-3\epsilon}}\,.
\end{equation*}
Therefore, by \cite[Corollary~5.2]{Hairer-16}, for every $x\in\Rd$, there exists
a sequence $\{t_n\}_{n\in\NN}\subset[0,\infty)$, $t_n\to\infty$,
such that
\begin{equation*}
\norm{\delta_{x} P^X_{t_n}(\D{x})
-\overline\uppi(\D{x})}_\tv
\;\ge\; g(x,t_n)^{-\frac{\theta_c-1+2\epsilon}{1-3\epsilon}}\,,
\end{equation*}
which establishes \cref{ET3.4B}.
This completes the proof.
\end{proof}

\subsection{Proof of \texorpdfstring{\cref{T3.5}}{}}

We start with the following lemma.

\begin{lem}\label{L5.8}
Under the assumptions \ttup{i} or \ttup{ii} of \cref{T3.5},
there exists $Q\in\cM_+$ such that
\begin{equation} \label{EL5.8A}
MQ+QM\,\succ\,0\,,
\ \text{and}\ \  (M- ev'(M-\varGamma))Q + Q(M- (M-\varGamma)v e')\,\succ\,0\,.
\end{equation}
\end{lem}

\begin{proof}
First consider \ttup{i} of \cref{T3.5}.
Since $M$ is a nonsingular M-matrix, $\Bar{v} \df M^{-1}(M v - \varGamma v)$
is a nonnegative vector which satisfies $e'\Bar v<1$.
The result then follows by \cref{EL5.5B}.

Next suppose that $M$ is a diagonal matrix and $\varGamma v\ne0$.
If $d=1$, the assertion is trivially satisfied.
Assume $d\ge2$ and define $\Tilde v\df M^{-1}\varGamma v$,
$\Hat v\df\Tilde v-v$ and $A_k\df M^k(\mathbb{I}+ \Hat ve')$ for $k=1,2$.
By assumption $\Tilde v\ne0$. Further, observe that $M-A_1=(\varGamma-M)ve'$
has rank one.
Thus, according to \cite[Theorem~1]{King},
in order to assert the existence of a positive definite matrix $Q$ satisfying
\cref{EL5.8A}, it suffices to show that the spectrum of $A_1$ lies in
the open right half of the complex plane and that $A_2$ does not have real
negative eigenvalues.

We first show that $A_2$ does not have real negative eigenvalues.
Suppose that $-\lambda$, with $\lambda\ge0$, is such an eigenvalue.
Then
\begin{align*}
0 \,=\,\det \bigl((\lambda\mathbb{I}+M^2)^{-1}\bigr)
\det \bigl(\lambda\mathbb{I}+A_2\bigr)
&\,=\, \det \bigl(\Id + (\lambda\mathbb{I}+M^2)^{-1}M^2\Hat{v}e'\bigr)\\
&\,=\, 1 + e' (\lambda\mathbb{I}+M^2)^{-1}M^2\Hat{v}\,,
\end{align*}
which implies that $e' (\lambda\mathbb{I}+M^2)^{-1}M^2\Hat{v}=-1$.
But 
\begin{align*}
e' (\lambda\mathbb{I}+M^2)^{-1}M^2\Hat{v}&\,\ge\,
e' (\lambda\mathbb{I}+M^2)^{-1}M^2\Tilde{v} -
e' (\lambda\mathbb{I}+M^2)^{-1}M^2 v\\
&\,>\, \Bigl(\min_i\,\tfrac{m_i^2}{\lambda+m_i^2}\Bigr) e'\Tilde{v}-
\Bigl(\max_i\,\tfrac{m_i^2}{\lambda+m_i^2}\Bigr)\,>\,-1\,,
\end{align*}
which is a contradiction. 

Next, we show that the spectrum of $A_1$ lies in the open right half
of the complex plane. Suppose that $\imath\lambda$, $\lambda\in\RR$,
is an eigenvalue of
$A_1$.
Then, since
\begin{equation*}
0 \,=\, \det\bigl( A_1-\imath\lambda\Id\bigr)
\,=\, \det\bigl( M-\imath\lambda \Id\bigr)
\det \bigl(\Id + (M-\imath\lambda\mathbb{I})^{-1}M\Hat{v}e'\bigr)\,,
\end{equation*}
 we have that
\begin{align*}
0 \,=\, \det \bigl(\Id + (M-\imath\lambda\mathbb{I})^{-1}M\Hat{v}e'\bigr)
&\,=\,1 + e' (M-\imath\lambda\mathbb{I})^{-1}M\Hat{v}\\
&\,=\, 1+ \sum_{k=1}^d \frac{ m_k^2 \Hat{v}_k}{m^2_k + \lambda^2}
+ \imath \sum_{k=1}^d \frac{ \lambda m_k\Hat{v}_k}{m^2_k + \lambda^2}\,.
\end{align*}
However, it holds that
\begin{equation*}
1+ \sum_{k=1}^d \frac{ m_k^2 \Hat{v}_k}{m^2_k + \lambda^2}\,>\,
1- \sum_{k=1}^d \frac{ m_k^2 v_k}{m^2_k + \lambda^2}\,\ge\,1-e'v\,=\,0\,.
\end{equation*}
Thus we reach a contradiction.
This shows that the matrix $A_1(t)\df M(\Id + (t\Tilde{v}-v)e')$ cannot
have any imaginary eigenvalues for any $t>0$, nor does it have
a zero eigenvalue.
Moreover, for all small enough $t>0$ the spectrum of
$A_1(t)$ is in the open right half of the complex plane by \cref{L5.5}.
Hence, by the continuity of the spectrum of $A_1(t)$ as a function of $t$, it follows
that the eigenvalues of $A_1(t)$ are in the open right
half complex plane for all $t>0$, which concludes the proof.
\end{proof}

\begin{proof}[Proof of \cref{T3.5}]
Consider the function $V_{Q,\theta}(x)$, $\theta >0$, with $Q$ as in \cref{EL5.8A}. 
Let $\Breve{b}(x) \df b(x) + \vartheta$, $x \in \RR^d$. 
Thus $V_{Q,\theta}(x)$ is an inf-compact function, and satisfies
\begin{equation}\label{PT3.5A}
\langle \Breve{b}(x),\nabla V_{Q,\theta}(x)\rangle
\,\le\,\kappa'_0-\kappa'_1\,V_{Q,\theta}(x)\,,\qquad x\in\Rd\,,
\end{equation}
for some constants $\kappa'_0>0$ and $\kappa'_1>0$ by \cref{L5.8}.
Then by \cref{L5.1,L5.3}, and
mimicking the proof of \cref{T3.2},
we obtain \cref{ET3.5B}, while \cref{ET3.5C}
follows from \cite[Theorem~6.1]{Meyn-Tweedie-AdvAP-III-1993}.

 We now turn to the last statement of the theorem.
It is well known that \cref{ET3.5B} implies that
$\int_\Rd V_{Q,\theta}(x)\,\overline\uppi(\D{x})\le\nicefrac{\Bar c_0}{\Bar c_1}$,
see \cite[Theorem 4.3]{Meyn-Tweedie-AdvAP-III-1993}. 
Thus $\int_\Rd \abs{x}^\theta\,\overline\uppi(\D{x})<\infty$.
It remains to show that if $q>0$ and $q\notin\Theta_c$, then
$\int_\Rd \abs{x}^q\,\overline\uppi(\D{x})=\infty$.
Recall the definition in \cref{PL5.7A}. We write
\begin{equation}\label{PT3.5B}
\Ag^XV_{Q,\theta}(x) \;=\; \bigl(\Ag_\circ^X V_{Q,\theta}(x)\bigr)^+
- \bigl(\Ag_\circ^X V_{Q,\theta}(x)\bigr)^-
+ \bigl\langle \Breve{b}(x),\nabla V_{Q,\theta}(x)\bigr\rangle\,.
\end{equation}
It is standard to show by using \cref{PT3.5A,PT3.5B}, together with
\cref{L5.1}, which holds for all $\theta\in\Theta_c$, and
the arguments in the proof of \cite[Lemma~3.7.2]{book}, that
\begin{equation}\label{PT3.5C}
- \int_\Rd \langle \Breve{b}(x),\nabla V_{Q,\theta}(x)\rangle\,\overline\uppi(\D{x})
+\int_\Rd \bigl(\Ag_\circ^X V_{Q,\theta}(x)\bigr)^-\,\overline\uppi(\D{x})
\,=\, \int_\Rd \bigl(\Ag_\circ^X V_{Q,\theta}(x)\bigr)^+\,\overline\uppi(\D{x})
\end{equation}
for all $x\in\Rd$ and $\theta\in\Theta_c$.
Note that there exist positive constants $C_0$ and $C_1$ such that
\begin{equation}\label{PT3.5D}
-C_0 - C_1^{-1} \bigl\langle \Breve{b}(x), \nabla V_{Q,\theta}(x)\bigr\rangle \,\le\,
\abs{x}^\theta \,\le\, C_0 - C_1 \bigl\langle \Breve{b}(x),
\nabla V_{Q,\theta}(x)\bigr\rangle
\end{equation}
for all $x\in\Rd$ and $\theta\in\Theta_c$.

First we consider the case $\Theta_c=(0,\theta_c)$.
A standard calculation shows that
\begin{equation}\label{PT3.5E}
\inf_{x\in\sB}\; \Ag_\circ^X V_{Q,\theta}(x)
\,\xrightarrow[\theta\nearrow\theta_c]{}\,\infty\,.
\end{equation}
Thus, combining \cref{PT3.5C,PT3.5D,PT3.5E}, and since
$\bigl(\Ag_\circ^X V_{Q,\theta}(x)\bigr)^-\le \kappa(1+\abs{x}^\theta)$,
we obtain $\int_\Rd \abs{x}^\theta\,\overline\uppi(\D{x})\to\infty$
as $\theta\nearrow\theta_c$.
This implies the result by \cref{R5.1}.

It remains to consider the case $\Theta_c=(0,\theta_c]$.
Suppose that $\int_\Rd \abs{x}^\theta\,\overline\uppi(\D{x})<\infty$ for
some $\theta>\theta_c$.
Recall the function $\chi_R(t)$ from \cref{N5.1}, and let
$V_R(x)\df \chi_R\circ V_{Q,\theta}(x)$.
Since $V_R(x)-R-2$ is compactly supported,
we have $\overline\uppi(\Ag^XV_R)=0$.
Thus,
\begin{equation*}
\Ag^X V_R(x) = \Ag_\circ^X V_R(x) + 
\chi_R'\bigl(V_{Q,\theta}(x)\bigr)
\bigl\langle \breve{b}(x), \nabla V_{Q,\theta}(x)\bigr\rangle\,,
\end{equation*}
and integrating this with respect to $\overline\uppi(\D{x})$, and using
\cref{PT3.5D}, we obtain
\begin{equation}\label{PT3.5F}
C_1\int_{\Rd}
\chi_R'\bigl(V_{Q,\theta}(x)\bigr)\,\bigl(C_0+\abs{x}^\theta\bigr)\,\overline\uppi(\D{x})
+ \int_{\Rd} \bigl[\Ag_\circ^X V_R\bigr]^-(x)\,\overline\uppi(\D{x})
\,\ge\, \int_{\Rd} \bigl[\Ag_\circ^X V_R\bigr]^+(x)\,\overline\uppi(\D{x})\,.
\end{equation}
It is important to note that since $a(x)$ satisfies \cref{ET3.5A},
the estimate in \cref{PL5.7E} here takes the form
$\babs{\widetilde{F}_{\kappa,R}(x)}\,\le\, \Tilde{C}\bigl(1 + F^{\theta}(x)\bigr)$,
and thus $\overline\uppi(\widetilde{F}_{\kappa,R})\to0$ as
$R\to\infty$, by hypothesis.
Since have a similar bound for $\bigl(\mathfrak{J}_{\nu}[V_R]\bigr)^-$,
and $\bigl(\mathfrak{J}_{1,\nu}[V_R]\bigr)^-$,
the left hand side of \cref{PT3.5F} is bounded uniformly in $R$, whereas the right
hand side diverges as $R\to\infty$, since $\theta>\theta_c$.
Thus we reach a contradiction.
This completes the proof.
\end{proof}

\subsection{Some results on general drifts}\label{S5.4}
In this section, we discuss ergodic properties of the solution to
\cref{E-sde} in the case when it is governed by a more general drift function.
In this section we assume that $b(x)$ is locally Lipschitz continuous,
and there exists $\kappa_0>0$ such that
$\langle x,b(x)\rangle\le\kappa_0(1+\abs{x}^{2})$ for all $x\in\Rd$.
Then, \cref{E-sde} again admits a unique nonexplosive strong solution
$\process{X}$ which is a strong Markov process and it satisfies the $C_b$-Feller property
(see \cite[Theorem~3.1, and Propositions 4.2 and 4.3]{Albeverio-Brzezniak-Wu-2010}).
Furthermore, its infinitesimal generator $(\Ag^X,\mathcal{D}_{\Ag^X})$
satisfies $C_c^{2}(\Rd)\subseteq\mathcal{D}_{\Ag^X}$, and 
$\Ag^X\bigl|_{C_c^{2}(\Rd)}$ takes the form in \cref{E1.2}.
Therefore, the corresponding extended domain contains the set $\mathcal{D}$
(defined in \cref{ext}), and on this set for $\Bar{\mathcal{A}}^Xf(x)$ we can
take exactly $\Ag^Xf(x)$.
Irreducibility and aperiodicity of $\process{X}$ with this general
drift can be established as in \cref{T3.1}.
The following corollary provides sufficient conditions on the drift function such
that the process $\process{X}$ exhibits subexponential or exponential ergodicity
properties analogous to \cref{T3.2,T3.5}.
The proof is similar to the proofs of those two theorems. 

\begin{cor}\label{C5.2}
Suppose that $\process{X}$ satisfies the assumptions of
\cref{T3.1}, and $\theta\in\Theta_c$. 
\begin{enumerate}[wide]
\item[\ttup{i}]
If $\theta\ge1$,
$a(x)$ satisfies \cref{ET3.2A},
and there exists $Q\in\cM_+$ such that
\begin{equation*}
\limsup_{\abs{x}\to\infty}\;\frac{\bigl\langle b(x)+r+\int_{\sB^c}y\nu(\D{y}),
Qx\bigr\rangle}{\abs{x}} \,<\, 0\,,
\end{equation*}
then the conclusion of \cref{T3.2}~\ttup{i}
holds with rate $r(t)\approx t^{\theta_c-1}$. 

\item[\ttup{ii}]

If $\theta\in(0,1)$,  
\begin{equation*}
\limsup_{\abs{x}\to\infty}\;\frac{\norm{a(x)}}{\abs{x}^{1+\theta}}\,=\,0\,,
\quad\text{and}\quad
\limsup_{\abs{x}\to\infty}\;\frac{\langle b(x), Qx\rangle}{\abs{x}^{1+\theta}} < 0
\end{equation*}
for some $Q\in\cM_+$,
then \cref{ET3.2D}
holds with rate $r(t)=t^{\nicefrac{\theta_c+\theta-1-\epsilon}{1-\theta}}$
for $\epsilon\in(0,\theta_c+\theta-1)$.

\item[\ttup{iii}] 
If $a(x)$ satisfies \cref{ET3.5A}, and 
$\limsup_{\abs{x}\to\infty}\;\nicefrac{\langle b(x), Qx\rangle}{\abs{x}^{2}} < 0$
for some $Q\in\cM_+$,
then there exist positive constants
$c_0$, $c_1$, such that
\begin{equation*}
\Ag^X V_{Q,\theta} (x) \,\le\, c_0 - c_1 V_{Q,\theta} (x)\,,\qquad x\in\Rd\,.
\end{equation*}
The process $\process{X}$ admits a unique invariant probability measure
$\overline\uppi\in\cP(\Rd)$, and for any $\gamma\in(0,c_1)$,
\begin{equation*}
\lim_{t\to\infty}\;\E^{\gamma t}\,\norm{\uppi P_t(\D{y})
-\overline\uppi(\D{y})}_\tv\,=\,0\,,\qquad \uppi\in\cP_\theta(\Rd)\,.
\end{equation*}
\item[\ttup{iv}]
Suppose that $\upsigma(x)$ is bounded, and there exist $\theta>0$
and $Q\in\cM_+$ such that
\cref{ET3.2E} holds and
\begin{equation*}
\limsup_{\abs{x}\to\infty}\;\frac{\bigl\langle b(x)+r+\int_{\sB^c}y\nu(\D{y}),
Qx\bigr\rangle}{\abs{x}} \,<\, 0\,.
\end{equation*}
Then the conclusion of \cref{T3.2}~\ttup{ii} follows.
\end{enumerate}
\end{cor}

\begin{proof}
In cases (i)--(iii) we use $V_{Q,\theta}(x)$, while in case (iv) we use
$\Tilde V_{Q,p}(x)$ with $0<p<\theta\norm{Q}^{-\nicefrac{1}{2}}$.
The assertions now follow from \cref{L5.1,L5.2}.
\end{proof}

We next present some general criteria for the convergence rate to
be no better than polynomial.
These extend \cref{T3.4}.

\begin{cor}\label{C5.3}
We assume that $\process{X}$ satisfies the assumptions of
\cref{T3.1}, $\theta_c\in[1,\infty)$,
and the drift satisfies, for some constant
$\gamma\in(0,1)$, one of the following.
\begin{enumerate}[wide]
\item[\ttup i]
There exists some $x_0\in\Rd$ and a positive constant $C$, such that
\begin{equation*}
\langle x_0, b(x)\rangle \,\ge\, -C\bigl(1 + \langle x_0,x\rangle^\gamma\bigr)\,,
\qquad  \langle x_0,x\rangle\ge0\,.
\end{equation*}
\item[\ttup{ii}]
There exists a positive definite symmetric matrix $Q$ 
and a positive constant $C$, such that
\begin{equation*}
\langle Qx, b(x)\rangle\,\ge\, -C\bigl(1 + \abs{x}^{1+\gamma}\bigr)\,,
\qquad x\in\Rd\,.
\end{equation*}
\end{enumerate}
In addition, suppose that there exists an inf-compact function
$\Bar{V}\in C^2(\Rd)$ having
strict polynomial growth of order $\abs{x}^\beta$ for some $\beta>\theta_c+\gamma-1$,
such that $\Ag^X \Bar{V}$ is bounded from above in $\Rd$.
Then, if the process is ergodic, we have the following lower bounds.
In the case of the $\alpha$-stable process (isotropic or not) there exists a positive
 constant $\Tilde{C}_1$ such that for all $\epsilon\in(0,1-\gamma)$
we have
\begin{equation}\label{EC5.3A}
\Tilde{C}_1 \Bigl(\frac{t}{\epsilon}
+ \Bar{V}(x)\Bigr)^{\frac{1-\gamma-\alpha}{1-\gamma-\epsilon}}
\,\le\, \norm{\delta_{x} P^X_t(\D{y})-\overline\uppi(\D{y})}_\tv
\end{equation}
for all $t>0$, and all $x\in\Rd$.

In the case of a L\'evy process in \cref{T3.4}\,\ttup{ii} we obtain 
a lower bound of the same type as in \cref{ET3.4B}.
There exists a positive constant $\Tilde{C}_3(\epsilon)$ such that for all
$0<\epsilon<\nicefrac{1}{2}\bigl(\beta-\theta_c-\gamma+1\bigr)$,
and all $x\in\Rd$, we have
\begin{equation}\label{EC5.3B}
\norm{\delta_{x} P^X_{t_n}(\D{y}) -\overline\uppi(\D{y})}_\tv
\,\ge\,\Tilde{C}_3(\epsilon)\,
\bigl (t_n+\Bar{V}(x)\bigr)^{-\frac{\theta_c+\gamma-1+2\epsilon}
{\beta-(\theta_c+\gamma-1+2\epsilon)}}
\end{equation}
for some sequence $\{t_n\}_{n\in\NN}\subset[0,\infty)$, $t_n\to\infty$, depending on $x$.
\end{cor}

\begin{proof}
We use a test function of the form
$F^\delta(x)$, $0<\delta<1-\gamma$,
with $F(x) = \chi\bigl(\langle x_0,x\rangle\bigr)$
($\chi(t)$ is as in \cref{N5.1})  for case (i), or
$F(x) = V_{Q,\delta}(x)$, with $V_{Q,\delta}(x)$ as in \cref{N3.1}, for case (ii).
We proceed with the technique in \cref{L5.7}, and show
that $F^{\kappa-1+\gamma}(x)$ cannot be integrable under $\overline\uppi(\D{x})$,
unless $\kappa\in\Theta_c$.
We continue by mimicking the proof of \cref{T3.4}.
Using $\Bar{V}$, we have
\begin{equation*}
\Exp^{x}[F^\beta\bigl(X(t)\bigr)]\,\le\,
\bigl(C_1 t+ \Bar{V}(x)\bigr)\;=:\;g(x,t)\,.
\end{equation*}
Note that necessarily $\beta\in\Theta_c$.

We estimate the tail of $\overline\uppi(\D{x})$ using
$F_{(1-\gamma),R}(x)$ (instead of $F_{1,R}(x)$) as
\begin{equation*}
\overline\uppi\bigl(\{y\colon \langle x,y\rangle > t\}\bigr)
\,\ge\, C_2 t^{1-\gamma-\alpha}\,.
\end{equation*}
With $G(x)= F^{\alpha-\epsilon}(x)$, $0<\epsilon<1-\gamma$, we have
\begin{align*}
\overline\uppi\bigl(\{y\colon G(y)\ge t\}\bigr)&\,=\,
\overline\uppi\bigl(\{y\colon (\langle x,y\rangle)^{\alpha-\epsilon} > t\}\bigr)\\
&\,=\, \overline\uppi\bigl(\{y\colon \langle x,y\rangle
> t^{\frac{1}{\alpha-\epsilon}}\}\bigr) \\
&\,\ge\, C_2 t^{\frac{1-\gamma-\alpha}{\alpha-\epsilon}}\;=:\;f(t)\,,
\end{align*}
and solving $yf(y) = 2 g(x,t)$ we obtain
$f(y)= C_3 \bigl(t+ \Bar{V}(x)\bigr)^{\nicefrac{1-\gamma-\alpha}{1-\gamma-\epsilon}}$,
and thus we obtain \cref{EC5.3A}.

In the case of the L\'evy, we define
\begin{equation*}
W(x)\,\df\, F^{\theta_c+\gamma-1+\epsilon}(x)\,,
\qquad \widehat{F}(t) \,\df\, t^{\frac{\beta}{\theta_c+\gamma-1+\epsilon}}\,,
\qquad h(t) \,\df\, t^{-1-\epsilon_1}\,,
\end{equation*}
with $\epsilon_1=\nicefrac{\epsilon}{\theta_c+\gamma-1+\epsilon}$,
and proceed as in the proof of \cref{T3.4}, to establish \cref{EC5.3B}.
This completes the proof.
\end{proof}

\begin{remark}
If we combine \cref{C5.2}\,(ii) and \cref{C5.3}\,(ii), in the
case of an $\alpha$-stable process (isotropic or not),
we obtain the following.
First note that the hypothesis in \cref{C5.2}\,(ii) allows us
to use the Lyapunov function
$V_{Q,\alpha-\epsilon}$ for any $\epsilon>0$, so that
the assumption $\beta>\alpha+\gamma-1$ in \cref{C5.3} comes for free.
Using this Lyapunov function,
and applying  \cite[Theorem~3.2]{Douc-Fort-Guilin-2009} for the
upper bound, then in combination with \cref{EC5.3A}, we obtain
\begin{equation*}
\Tilde{C}_1 \Bigl(\frac{t}{\epsilon}
+ \abs{x}^{\alpha-\epsilon}\Bigr)^{\frac{1-\gamma-\alpha}{1-\gamma-\epsilon}}
\,\le\, \bnorm{\delta_{x} P^X_t(\cdot)-\overline\uppi(\cdot)}_\tv
\,\le\, \Tilde{C}_2(\epsilon)\,
(t\vee1)^\frac{1+\epsilon-\alpha-\theta}{1-\theta}\,
\abs{x}^{\alpha-\epsilon}\,.
\end{equation*}
Note that necessarily $\gamma\ge\theta$ by hypothesis.
\end{remark}

\appendix
\section{Proof of Theorem~\ref{T3.1}} \label{S-A}
In this section we prove \cref{T3.1}. 
The assertion for case (i) is shown in \cite[proof of Proposition 3.1]{Masuda-2007}
(see also \cite[Theorem~3.1]{Kwon-Lee-1999}).
We prove the assertions in cases (ii), (iii) and (iv).

\begin{proof}[Proof of Case $(ii)$]

First, observe that $P^X_t(x,O)>0$ for any $t>0$, $x\in\Rd$ and open set $O\subseteq\Rd$.
Indeed, fix $0<\rho\le\nicefrac{1}{4}$ and $0<\varepsilon<\rho$,
and let $x_0,y_0\in\Rd$ be such
that $\abs{x_0-y_0}=2\rho$.
Let $f\in C_c^{2}(\Rd)$ be such that $0\le f(x)\le 1$,
$\support f\subset \sB_{\rho-\nicefrac{\varepsilon}{2}}(y_0)$,
and $f|_{\Bar{\sB}_{\rho-\varepsilon}(y_0)}=1$.
Recall that $\process{X}$ is a $C_b$-Feller process with generator
$(\Ag^X,\mathcal{D}_{\Ag^X})$ given in \cref{E1.2}.
Now, since $\lim_{t\to 0}\;\bnorm{\nicefrac{P^X_t f-f}{t}-\Ag^Xf}_\infty = 0$,
we conclude that 
\begin{align*}
\liminf_{t\searrow0}\;\inf_{x\in \sB_{\rho-\nicefrac{\varepsilon}{2}}(x_0)}\;
\frac{P^X_t\bigl(x,\sB_{\rho-\nicefrac{\varepsilon}{2}}(y_0)\bigr)}{t}
& \,\ge\,
\liminf_{t\searrow0}\;\inf_{x\in \sB_{\rho-\nicefrac{\varepsilon}{2}}(x_0)}\;
\frac{P^X_tf(x)}{t}\\
&\,=\,\liminf_{t\searrow0}\;\inf_{x\in \sB_{\rho-\nicefrac{\varepsilon}{2}}(x_0)}
\;\babss{\frac{P^X_tf(x)}{t}-\Ag^Xf(x)+\Ag^Xf(x)}\\
&\,=\,\inf_{x\in \sB_{\rho-\nicefrac{\varepsilon}{2}}(x_0)}\;\babs{\Ag^Xf(x)}\\
&\,\ge\, \inf_{x\in \sB_{\rho-\nicefrac{\varepsilon}{2}}(x_0)}\;
\int_{\Rds}f(y+x)\,\nu(\D{y})\\
&\,\ge\, \inf_{x\in \sB_{\rho-\nicefrac{\varepsilon}{2}}(x_0)}\;
\nu\bigl(\sB_{\rho-\varepsilon}(y_0-x)\bigr)\,.
\end{align*} 
Observe that
\begin{equation*}
\bigcup_{x\in \sB_{\rho-\nicefrac{\varepsilon}{2}}(x_0)}
\sB_{\rho-\varepsilon}(y_0-x)\,\subseteq\, \sB\setminus\sB_\varepsilon(0)\,.
\end{equation*}
We claim that
\begin{equation*}
\inf_{x\in \sB_{\rho-\nicefrac{\varepsilon}{2}}(x_0)}\;
\nu\bigl(\sB_{\rho-\varepsilon}(y_0-x)\bigr)\,>\,0\,.
\end{equation*}
Suppose not. Then there exists a sequence
$\{x_n\}_{n\in\NN}\subseteq \sB_{\rho-\nicefrac{\varepsilon}{2}}(x_0)$
converging to some $x_\infty\in\Bar{\sB}_{\rho-\nicefrac{\varepsilon}{2}}(x_0)$,
such that
$\lim_{n\to\infty}\,\nu\bigl(\sB_{\rho-\varepsilon}(y_0-x_n)\bigr)=0$.
On the other hand, by the dominated convergence theorem
(observe that $\bigcup_{n\in\NN}\sB_{\rho-\varepsilon}(y_0-x_n)
\subseteq \sB\setminus\sB_\varepsilon(0)$ and
$\nu(\sB\setminus\sB_\varepsilon(0))<\infty$), we have that
\begin{equation*}
\lim_{n\to\infty}\;\nu\bigl(\sB_{\rho-\varepsilon}(y_0-x_n)\bigr)\,=\,
\nu(\sB_{\rho-\varepsilon}\bigl(y_0-x_\infty)\bigr)\,,
\end{equation*}
which is strictly positive by hypothesis.
Thus, we conclude that there exists $t_0>0$ such that
$P^X_t\bigl(x,\sB_{\rho-\varepsilon}(y_0)\bigr)>0$
for all $t\in(0,t_0]$, and all $x\in \sB_{\rho-\varepsilon}(x_0)$.
The claim now follows by employing the Chapman--Kolmogorov equality.

Next, define $\nu_1(\D{y})\df \nu(\D{y}\cap \sB)$,
and $\nu_2(\D{y})\df \nu(\D{y}\cap \sB^c)$.
Let $\process{W}$, $\process{L_1}$, and $\process{L_2}$ be a mutually independent
standard Brownian motion, a L\'evy process with drift $\vartheta$ and L\'evy measure
$\nu_1(\D{y})$ and a L\'evy process with zero drift and L\'evy measure $\nu_2(\D{y})$,
respectively.
Observe that $\process{L}$ and $\{L_1(t)+L_2(t)\}_{t\ge0}$ have the same
(finite-dimensional) distribution.
Now, define
\begin{equation*}
\D \Bar{X}(t)\,\df\,b(\Bar{X}(t))\,\D{t} +\upsigma(\Bar{X}(t))\, \D W(t)+\D L_1(t)
+\D L_2(t)\,,\quad \Bar{X}(0)=x\in\Rd\,,
\end{equation*}
and 
\begin{equation*}
\D \Hat X(t)\,\df\,b(\Hat X(t))\,\D{t} +\upsigma(\Hat X(t))\, \D W(t)
+\D L_1(t)\,,\quad \Hat X(0)=x\in\Rd\,.
\end{equation*}
It is clear that the processes $\process{X}$ and $\process{\Bar{X}}$ have the same
(finite-dimensional) distribution, and by the same reasoning as above,
$P^{\Bar X}_t(x,O)>0$
and $P^{\Hat X}_t(x,O)>0$ for any $t>0$,
$x\in\Rd$ and open set $O\subseteq\Rd$.
Next, define
\begin{equation*}
\uptau\,\df\,\inf\,\bigl\{t\ge0\,\colon \abs{\Bar{X}(t)-\Bar{X}(t-)}\ge1\bigr\}
\,=\,\inf\,\bigl\{t\ge0\,\colon \abs{L_2(t)-L_2(t-)}\ne0\bigr\}\,.
\end{equation*}
Now, by construction, we conclude that
$\Prob^x(\uptau>t)=\E^{-\nu_2(\Rd)t}=\E^{-\nu(\sB^c)t}$,
and $\process{\Bar{X}}$ and $\process{\Hat{X}}$ coincide on $[0,\uptau)$. 
Consequently, for any $t>0$, $x\in\Rd$ and $B\in\mathfrak{B}(\Rd)$,
we have
\begin{align*}
P^{\Bar X}_t(x,B)&\,\ge\,
\Prob^x(\Bar{X}(t)\in B,\, \uptau>t)\\
&\,=\,\Prob^x(\Hat{X}(t)\in B,\, \uptau>t)\\
&\,=\,\Exp^x\Bigl[\Exp^x\bigl[\Ind_{\{\Hat{X}(t)\in B\}}
\Ind_{\{\uptau>t\}}\bigm| \sigma\{L_2(t)\,,\;t\ge0\}\bigr]\Bigr]\\
&\,=\,P^{\Hat X}_t(x,B)\,\Prob^x(\uptau>t)\,.
\end{align*}
Thus, according to \cite[Theorem 3.2]{Tweedie-1994}, in order to conclude
open-set irreducibility and aperiodicity of $\process{X}$, it suffices to prove
that $\process{\Hat{X}}$ is a strong Feller process.
Further, by \cite[Lemma 2.2]{Peszat-Zabczyk-1995} 
for $\process{\Hat{X}}$ to have he strong Feller property
it is sufficient that that for any $t>0$ there
exists $c(t)>0$ such that
\begin{equation}\label{PT3.1A}
\babs{P^{\Hat X}_t f(x)-P^{\Hat X}_t f(y)}
\,\le\, c(t,\kappa,\delta)\norm{f}_\infty\,\abs{x-y}
\end{equation}
for all $f\in C_b^2(\Rd)$ and $x,y\in\Rd$.
This is what we show in the rest of the proof.

Let $\chi\in C_c^{\infty}(\Rd)$
satisfying $0\le\chi(x)\le1$ for all $x\in \Rd$,
$\support\chi\subseteq \Bar{\sB}$
and $\int_{\Rd}\chi(x)\,\D{x}=1$. For $\varepsilon>0$, define
$\chi_\varepsilon(x)\,\df\,\varepsilon^{-d}\chi(\nicefrac{x}{\varepsilon})$,
$x\in\Rd$.
By definition, $\chi_\varepsilon\in C_c^{\infty}(\Rd)$,
$\support\chi_\varepsilon\subseteq \Bar{\sB}_\varepsilon(0)$
and $\int_{\Rd}\chi_\varepsilon(x)\,\D{x}=1$. 
The Friedrich's mollifiers $b_n(x)$ ad $\upsigma_n(x)$
of $b(x)$ and $\upsigma(x)$, respectively, are defined as
\begin{equation*}
b_n(x)\,\df\,n^{d}\int_{\Rd}\chi_{\nicefrac{1}{n}}(x-y)b(y)\,\D{y}
\,=\,\int_{\Bar{B}_1(0)}\chi(y)b\bigl(x-\frac{y}{n}\bigr)\,\D{y}\,,
\end{equation*}
and analogously for $\upsigma_n$.
Let $\kappa>0$ be larger than the Lipschitz constants of
$b(x)$ and $\upsigma(x)$.
Since $b_n\in C^{\infty}(\Rd,\Rd)$ and
$\upsigma_n\in C^{\infty}(\Rd,\RR^{d\times d})$, we have 
\begin{equation*}
\babs{b_n^i(x)-b^i(x)}\,\le\,\frac{1}{n}\,,\qquad
\qquad \babs{\upsigma_n^{ij}(x)-\upsigma^{ij}}\,\le\,\frac{1}{n}\,,
\end{equation*}
\begin{equation*}
\babs{\partial_ib_n^j(x)}\,\le\,\kappa\,,
\quad\babs{\partial_i\upsigma_n^{jk}(x)}\,\le\,\kappa\,,\quad
\babs{\partial_{ij}b_n^k(x)}\,\le\,\kappa_n\,,
\quad\text{and\ } \babs{\partial_{ij}\upsigma_n^{kl}(x)}\le\kappa_n\,,
\end{equation*}
with $i,j,k,l\in\{1,\dotsc,d\}$,
$b_n(x)=(b_n^i(x))_{i=1,\dotsc,d}$, and
$\upsigma_n(x)=(\upsigma^{ij}_n(x))_{i,j=1,\dotsc,d}$. 
Now, define
\begin{equation*}
\D \Hat X_n(t)\,=\,b_n(\Hat X(t))\,\D{t} +\upsigma_n(\Hat X_n(t))\, \D W(t)
+\D L_1(t)\,,\qquad \Hat X_n(0)=x\in\Rd\,.
\end{equation*}
In \cite[Lemma 2.3]{Kwon-Lee-1999} it has been shown that 
for each fixed $t>0$ there is a constant $c(t,\kappa,\delta)>0$ such that
\begin{equation}\label{PT3.1B}
\babs{P^{\Hat{X}_n}_t f(x)-P^{\Hat{X}_n}_t f(y)}
\,\le\, c(t,\kappa,\delta)\norm{f}_\infty\,\abs{x-y}
\end{equation}
for all $f\in C_b^2(\Rd)$, $x,y\in\Rd$ and $n\in\NN$.
Recall that $\delta=\sup_{x\in\Rd}\,\bnorm{\upsigma^{-1}(x)}>0$.
As we have already commented, this automatically implies strong Feller property
of $\process{\Hat X_n}$.
For any $t>0$ and $x\in\Rd$, by employing It\^o's formula, we have
\begin{align*}
\frac{\D}{\D{t}}\Exp^x\babs{\Hat X_n(t)-\Hat X(t)}^2
&\,=\,\Exp^x \bnorm{\upsigma_n(\Hat X_n(t))
-\upsigma(\Hat X(t))}^2+2\Exp^x\bigl\langle b_n(\Hat X_n(t))
-b(\Hat X(t)),\Hat X_n(t)-\Hat X(t)\bigr\rangle\\
&\,\le\, 2\Exp^x \bnorm{\upsigma_n(\Hat X_n(t))-\upsigma_n(\Hat X(t))}^2
+2\Exp^x \bnorm{\upsigma_n(\Hat X(t))-\upsigma(\Hat X(t))}^2\\
&\mspace{100mu}
+2\Exp^x\bigl\langle b_n(\Hat X_n(t))-b_n(\Hat X(t)),\Hat X_n(t)-\Hat X(t)\bigr\rangle\\
&\mspace{200mu}
+2\Exp^x\bigl\langle b_n(\Hat X(t))-b(\Hat X(t)),\Hat X_n(t)-\Hat X(t)\bigr\rangle\\
&\,\le\, 2\kappa^2\,\Exp^x \babs{\Hat X_n(t)-\Hat X(t)}^2
+2\Exp^x \bnorm{\upsigma_n(\Hat X(t))-\upsigma(\Hat X(t))}^2\\
&\mspace{100mu} +2\kappa\,\Exp^x \babs{\Hat X_n(t)-\Hat X(t)}^2
+2\Exp^x \babs{b_n(\Hat X(t))-b(\Hat X(t))}^2\\
&\,\le\, 2(\kappa+\kappa^2)\,\Exp^x \babs{\Hat X_n(t)-\Hat X(t)}^2+\frac{4}{n^2}\,.
\end{align*}
By Gronwall's lemma we obtain
\begin{equation*}
\Exp^x\babs{\Hat X_n(t)-\Hat X(t)}^2\,\le\,\frac{4}{n^2}t\E^{2(\kappa+\kappa^2)t}\,.
\end{equation*}
Hence, for each fixed $t>0$ and $x\in\Rd$, $\Hat X_n(t)$ converges to
$\Hat X(t)$ in $L^2(\Omega,\Prob^x)$.
Now, for fixed $t>0$ and $x,y\in\Rd$, using \cref{PT3.1B}, we have
\begin{align*}
\babs{P^{\Hat X}_t f(x)-P^{\Hat X}_t f(y)}\,&\le\,
\babs{P^{\Hat X}_t f(x)-P^{\Hat{X}_n}_t f(x)}
+\babs{P^{\Hat{X}_n}_t f(x)-P^{\Hat{X}_n}_t f(y)}
+\babs{P^{\Hat{X}_n}_t f(y)-P^{\Hat X}_t f(y)}\\
&\le\,\babs{P^{\Hat X}_t f(x)-P^{\Hat{X}_n}_t f(x)}
+\babs{P^{\Hat{X}_n}_t f(y)-P^{\Hat X}_t f(y)}
+c(t,\kappa,\delta)\norm{f}_\infty\abs{x-y}\,.
\end{align*}
By letting $n\to\infty$, \cref{PT3.1A} follows, and the proof is complete.
\end{proof}

We next prove the assertions in Case (iii). 
\begin{proof}[Proof of Case $(iii)$]
By \cite[Theorem~3.2]{Tweedie-1994}, in order to prove open-set
irreducibility and aperiodicity, it 
suffices to show that $\process{X}$ satisfies the strong Feller property and
$P^X_t(x, O)>0$ for all $t>0$, all $x\in\Rd$, and all open sets $O\subseteq\Rd$. 
The strong Feller property of $\process{X}$ follows from 
\cite[Theorem~2.1]{Wang-Wang-2014} and \cite[Proposition 2.3]{Wang-2010}. 
Recall that $\process{X}$ is a $C_b$-Feller process with generator
$(\Ag^X,\mathcal{D}_{\Ag^X})$ given in \cref{E1.2}. 
Now, let
$O\subseteq\Rd$ be an arbitrary open set and let $f\in C_c^{2}(\Rd)$ be such 
that $\support f\subset O$ and $0\le f(x)\le 1$.
Since
$$\lim_{t\to 0}\;\Bnorm{\tfrac{P^X_t f-f}{t}-\Ag^Xf}_\infty \,=\, 0\,,$$
we conclude that for any bounded set $B\subseteq O^c$, 
\begin{align*}
\liminf_{t\searrow0}\;\inf_{x\in B}\;\frac{P^X_t(x,O)}{t}&\,\ge\,
\liminf_{t\searrow0}\;\inf_{x\in B}\;\frac{P^X_tf(x)}{t}\\
&\,=\,\liminf_{t\searrow0}\;\inf_{x\in B}\;
\babss{\frac{P^X_tf(x)}{t}-\Ag^Xf(x)+\Ag^Xf(x)}\\
&\,=\,\inf_{x\in B}\;\babs{\Ag^Xf(x)}
\,\ge\, \inf_{x\in B}\;\int_{\Rds}f(y+x)\,\nu(\D{y})\,.
\end{align*}
Now, since $\process{L}$ has a subordinate Brownian motion component, we conclude that
$\nu(\D{y})$ has full support (see \cite[Theorem~30.1]{Sato-Book-1999}).
This automatically implies that the right hand side in the above relation is
strictly positive.
Namely, if this was not the case then there would exist a sequence
$\{x_n\}_{n\in\NN}\subseteq B$ converging to some $x_\infty\in\Rd$
(recall that $B$ is bounded), such that
\begin{equation*}
\lim_{n\to\infty}\int_{\Rds}f(y+x_n)\,\nu(\D{y})\,=\,0\,.
\end{equation*}
Now, by employing Fatou's lemma we conclude that
$\int_{\Rds}f(y+x_\infty)\,\nu(\D{y})=0$,
which is impossible.
Hence, there exists $t_0>0$ such that $P^X_t(x,O)>0$
for all $t\in(0,t_0]$, and all $x\in B$.
The assertion now follows by employing the Chapman-Kolmogorov equality.
\end{proof}

Lastly, we prove the assertions in Case (iv).
\begin{proof}[Proof of Case $(iv)$]

Let us first show that the solution to
\begin{equation*}
\D \Hat X(t)=b(\Hat X(t))\,\D{t}+\D L_1(t),\qquad \Hat X(0)=x\in\Rd\,,
\end{equation*}
is a strong Feller process.
 Clearly, $\process{L_1}$ admits a transition density function
\begin{equation*}
P_t^{L_1}(x, \D{y})\,=\,p^t(x,y)\D{y}\,=\,p^t(y-x)\D{y}\,,\qquad x,y\in\Rd\,,\ \ t>0\,,
\end{equation*}
satisfying
\begin{equation*}
p^t(x)\,=\,p_1^t(x_1)\dotsb p_d^t(x_d)\,,\qquad x=(x_1,\dotsc,x_d)\in\Rd,\ \ t>0\,,
\end{equation*}
since the corresponding components are independent,
where $\{p_i^t(u)\}_{u\in\RR,\, t>0}$ is a transition density of $\process{ A_1^i}$,
$i=1,\dotsc,d$, (a one-dimensional symmetric $\alpha$-stable L\'evy process
with scale parameter $\eta_i>0$).
Observe that $\process{L_1^i}$ is a subordinate Brownian motion with
$\nicefrac{\alpha}{2}$-stable subordinator $\process{S_i}$ with scale parameter
$\eta_i>0$, $i=1,\dotsc,d$.
According to \cite[Corollary 3.5 and Example 4.4]{Kusuoka-Marinelli-2014},
$P_t^{L_1^i}f\in C_b^1(\RR)$,
\begin{equation*}
\babss{\frac{\partial}{\partial u} P_t^{L_1^i}f(u)}
\,\le\, c_1 t^{-\frac{\alpha}{4}}\,\norm{f}_\infty \,
\Exp\bigl[S_i(1)^{-\frac{1}{2}}\bigr]\,,
\end{equation*}
and
\begin{equation*}
\abs{u} \babss{\frac{\partial}{\partial u} P_t^{L_1^i}f(u)}
\,\le\, c_2 t^{-\frac{\alpha}{4}}\,\norm{f}_\infty \,
\Exp\bigl[S_i(1)^{-\frac{1}{2}}\bigr]
\end{equation*}
for all $u\in\RR$, $t>0$ and $f\in\mathcal{B}(\RR)$ with compact support.
Here, the constant $c_1>0$ does not depend on $u$, $t$ and $f(u)$, and $c_2>0$ depends
on $\support f$ only.
According to \cite[Proposition 3.11]{Kusuoka-Marinelli-2014}, we have
$\Exp[S_i(1)^{-\nicefrac{1}{2}}]<\infty$.
Further, using the scaling property and asymptotic behavior (at infinity)
of one-dimensional symmetric stable densities
(see \cite[page 87]{Sato-Book-1999}),
we deduce that for any fixed $t_0>0$,
there exist positive constants $c_3$ and $c_4$,
which depend only on $t_0$ and $\text{supp}\, f$,
such that
\begin{align*}
\abs{u} \babs{ P_t^{L_1^i}f(u)} &\,\le\, t^{-\frac{1}{\alpha}} \abs{u}
\norm{f}_\infty
\int_{\support f}p_i^1\bigl(t^{-\frac{1}{\alpha}}(v-u)\bigr)\,\D v\\
 &\,\le\, c_3\norm{f}_\infty\abs{u}^{-\alpha}t\,
 \Ind_{\bigl\{\abs{u}t_0^{-\nicefrac{1}{\alpha}}\ge c_3\bigr\}}(u)
 +c_4\norm{f}_\infty
\Ind_{\bigl\{\abs{u}t_0^{-\nicefrac{1}{\alpha}}\le c_3\bigr\}}(u)
\end{align*}
for all $u\in\RR$, and $t\in(0,t_0]$.

According to \cite[Lemma 2.4]{Kusuoka-Marinelli-2014} a $C_b$-Feller semigroup
$\{P_t\}_{t\ge0}$ enjoys the strong Feller property if, and only if,
$P_t f\in C_b(\Rd)$ for any $t>0$ and $f\in\mathcal{B}_b(\Rd)$ with compact support. 
Finally, fix $t>0$ and $f\in\mathcal{B}_b(\Rd)$ with compact support.
By Duhamel's formula we have
\begin{equation*}
P_t^{\Hat X}f(x)\,=\,
P_t^{L_1}f(x)+\int_0^t P_{t-s}^{\Hat X}
\bigl\langle b(x),\nabla P_s^{L_1}f(x)\bigr\rangle\, \D{s}\,.
\end{equation*}
Hence, it remains to prove that the second term in the above relation is continuous.

We have
\begin{align*}
\bigl\langle b(x),\nabla P_s^{L_1}f(x)\bigr\rangle&
\,=\,\bigl\langle b(x)(1+\abs{x})^{-1},(1+\abs{x})\nabla P_s^{L_1}f(x)\bigr\rangle\\
&\,\le\,(1+\abs{x})^{-1}\,\abs{b(x)}+(1+\abs{x})\,\babs{\nabla P_s^{L_1}f(x)}\,.
\end{align*}
Clearly, $(1+\abs{x})^{-1}\,\abs{b(x)}<c$ for some $c>0$.
Since
\begin{align*}
\partial_{x_i} P_s^{L_1}f(x)&\,=\,\partial_{x_i}\int_{\Rd}f(y_1,\dotsc,y_d)
\prod_{i=1}^dp_i^s(y_i-x_i)\,\D{y}_i\\
&\,=\,\int_{\RR} \partial p_i^s(y_i-x_i)\,\D{y}_i
\int_{\RR^{d-1}}f(y_1,\dotsc,y_d)\prod_{j\ne i}p_j^s(y_j-x_j)\,\D{y}_j\,,
\end{align*}
the map
\begin{equation*}
y_i\mapsto\int_{\RR^{d-1}}f(y_1,\dotsc,y_d)\prod_{j\ne i}p_j^s(y_j-x_j)\,\D{y}_j
\end{equation*}
is bounded and has compact support.
We use the estimates 
\begin{align*}
\abs{x_i}\,\babs{\partial_{x_j} P_s^{L_1}f(x)}\,
\bigl(\Exp\bigl[S_j(1)^{-\frac{1}{2}}\bigr]\bigr)^{-1}
&\,\le\,
\Bar c_2\norm{f}_\infty\, s^{-\frac{\alpha}{4}}\,,\quad\text{if\ }i=j\,,
\intertext{and}
\abs{x_i}\,\babs{\partial_{x_j} P_s^{L_1}f(x)}\,
\bigl(\Exp\bigl[S_j(1)^{-\nicefrac{1}{2}}\bigr]\bigr)^{-1}
&\,\le\,
\Bar c_3\norm{f}_\infty^2\abs{x_i}^{-\alpha}s^{1-\frac{\alpha}{4}}
\Ind_{\bigl\{\abs{x_i}t^{-\nicefrac{1}{\alpha}}\ge \Bar c_3\bigr\}}(x_i)\\[3pt]
&\mspace{100mu}
+\Bar c_4\norm{f}^2_\infty\, s^{-\frac{\alpha}{4}}
\Ind_{\bigl\{\abs{x_i}t^{-\nicefrac{1}{\alpha}}\le \Bar c_3\bigr\}}(x_i)\\[3pt]
&\,=\,\bigl(\Bar c_3^{1-\alpha}+\Bar c_4\bigr)
\norm{f}_\infty^2\,s^{-\frac{\alpha}{4}}
\,,\quad\text{if\ }i\ne j\,,
\end{align*}
to obtain
\begin{equation}\label{PT3.1C}
\bigl\langle b(x),\nabla P_s^{L_1}f(x)\bigr\rangle
\,\le\, 2c+M(f)\, s^{-\frac{\alpha}{4}}
\max_{i=1,\dotsc,d}\,\Exp\bigl[S_i(1)^{-\frac{1}{2}}\bigr]\,.
\end{equation}
with
\begin{equation*}
M(f) \,\df\, \Bar c_1\norm{f}_\infty+
\norm{f}_\infty\sqrt{d\Bar c_2^2+d(d-1)(\Bar c_3^{1-\alpha}+\Bar c_4)^2
\norm{f}_\infty^2}\,.
\end{equation*}
In \cref{PT3.1C}, the constants
$\Bar c_2$, $\Bar c_3$ and $\Bar c_4$ depend on $t$ and $\support f$ only.
Therefore, the map
$x\mapsto\bigl\langle b(x),\nabla P_s^{L_1}f(x)\bigr\rangle$
is continuous and bounded for any $s\in(0,t]$.
In particular, due to the $C_b$-Feller property, the map
$x\mapsto P_{t-s}^{\Hat X}\bigl\langle b(x),\nabla P_s^{L_1}f(x)\bigr\rangle$
is also continuous and bounded for any $s\in(0,t]$.
Finally, according to dominated convergence theorem and \cref{PT3.1C} we conclude that
\begin{equation*}
x\,\mapsto\,
\int_0^tP_{t-s}^{\Hat X}\bigl\langle b(x),\nabla P_s^{L_1}f(x)\bigr\rangle\,\D{s}
\end{equation*}
is continuous, which concludes the proof.

Now, let us show that $\process{\Hat X}$ and the solution to 
\begin{equation*}
\D \Bar{X}(t)\,=\,b(\Bar{X}(t))\,\D{t} +\D L_1(t)+\D L_2(t),\qquad \Bar{X}(0)=x\in\Rd\,,
\end{equation*}
are irreducible and aperiodic. The one-dimensional case is covered by (iii).
Assume $d\ge2$. Define
$\uptau\df\inf\bigl\{t\ge0\,\colon \abs{L_2(t)-L_2(t-)}\ne0\bigr\}$.
By construction we conclude that
$\Prob^x(\tau>t)=\E^{-\nu_{L_2}(\Rd)t}$,
and $\process{\Hat{X}}$ and $\process{\Bar{X}}$ coincide on $[0,\uptau)$.
Consequently, 
\begin{align*}
P_t^{\Bar X}(x,B)&\,\ge\, \Prob^x(\Bar{X}(t)\in B,\, \uptau>t)\\
&\,=\,\Prob^x(\Hat{X}(t)\in B,\, \uptau>t)\\
&\,=\,\Exp^x\Bigl[\Exp^x\bigl[\Ind_{\{\Hat{X}(t)\in B\}}
\Ind_{\{\uptau>t\}}\bigm|\sigma\{L_2(t)\,,\;t\ge0\}\bigr]\Bigr]\\
&\,=\,P_t^{\Hat X}(x,B)\,\Prob^x(\uptau>t)
\end{align*}
for any $t>0$, $x\in\Rd$ and $B\in\mathfrak{B}(\Rd)$.
Thus, due to the previous observation, the fact that $\process{\Hat X}$ is a
strong Feller process and \cite[Theorem 3.2]{Tweedie-1994}, it suffices to show that
$P_t^{\Hat X}(x,B)>0$ for any $t>0$, $x\in\Rd$
and open set $O\subseteq\Rd$ containing $0$.
First, assume that $x_0,y_0\in\Rd$ lie on the same coordinate axis.
Fix $\rho>0$ and $\varepsilon>0$, and let $f\in C_c^{2}(\Rd)$ be such that
$0\le f(x)\le 1$,
$\support f\subset \sB_{\rho+\varepsilon}(y_0)$
and $f|_{\bar\sB_{\rho+\nicefrac{\varepsilon}{2}}(y_0)}=1$.
Now, since $\lim_{t\to 0}\;\bnorm{\nicefrac{P^{\Hat X}_t f-f}{t}
-\mathcal{A}^{\Hat X}f}_\infty = 0$
($(\mathcal{A}^{\Hat X},\mathcal{D}_{\mathcal{A}^{\Hat X}})$ is the generator of
$\process{\Hat X}$ given in \cref{E1.2}),
we conclude that 
\begin{align*}
\liminf_{t\searrow0}\;\inf_{x\in \sB_{\rho}(x_0)}\;
\frac{P_t^{\Hat X}\bigl(x,\sB_{\rho+\varepsilon}(y_0)\bigr)}{t}
&\,\ge\,
\liminf_{t\searrow0}\;\inf_{x\in \sB_{\rho}(x_0)}\;
\frac{P^{\Hat X}_tf(x)}{t}\\
&\,=\,\liminf_{t\searrow0}\;\inf_{x\in \sB_{\rho}(x_0)}
\;\babss{\frac{P^{\Hat X}_tf(x)}{t}-\mathcal{A}^{\Hat X}f(x)
+\mathcal{A}^{\Hat X}f(x)}\\
&\,=\,\inf_{x\in \sB_{\rho}(x_0)}\;\babs{\mathcal{A}^{\Hat X}f(x)}\\
&\,\ge\, \inf_{x\in \sB_{\rho}(x_0)}\;
\int_{\Rds}f(y+x)\,\nu(\D{y})\\
&\,\ge\, \inf_{x\in \sB_{\rho}(x_0)}\;
\nu\bigl(\sB_{\rho+\nicefrac{\varepsilon}{2}}(y_0-x)\bigr)\,,
\end{align*} which is strictly positive.
Hence, there is $t_0>0$ such that
$P_t^{\Hat X}\bigl(x,\sB_{\rho+\varepsilon}(y_0)\bigr)>0$
for all $t\in(0,t_0]$ and $x\in \sB_{\rho}(x_0)$.
Further, let $x\in\Rd$ be such that it does not lie on any coordinate axis,
and let $\rho>0$.
Define $r_1\df \abs{\langle x,e_1\rangle}+\varepsilon_1$, where $\varepsilon_1>0$ is such
that $r_1<\rvert x\rvert$.
Then, as above, we conclude that there is $t_1>0$ such that
$P_t^{\Hat X}\bigl(x,\sB_{r_1}(0)\bigr)>0$ for all $t\in(0,t_1]$.
Next, inductively, define $r_n\df\nicefrac{r_{n-1}}{\sqrt{d}}+\varepsilon_n$, $n\ge2$,
where $\varepsilon_n>0$ is such that
$\varepsilon_n<r_{n-1}\sqrt{d-\nicefrac{1}{2}}$.
Clearly, $r_n\to0$ as $n\to\infty$, and there is $t_n>0$ such that
$P_t^{\Hat X}\bigl(x,\sB_{r_n}(0)\bigr)>0$ for all $t\in(0,t_n]$
and $x\in\sB_{r_{n-1}}(0)$.
The claim now follows by employing the Chapman-Kolmogorov equality.
\end{proof}

\section*{Acknowledgements}
This research was supported in part by 
the Army Research Office through grant W911NF-17-1-001,
in part by the National Science Foundation through grants DMS-1715210,
CMMI-1538149 and DMS-1715875,
and in part by Office of Naval Research through grant N00014-16-1-2956.
Financial support through Croatian Science Foundation under the project 3526
(for N. Sandri\'c) is gratefully acknowledged.
We also thank the anonymous referee for the helpful comments that have led to
significant improvements of the results in the paper.

\def\cprime{$'$}

\end{document}